\documentclass[12pt, reqno]{amsart}
\usepackage[cp1251]{inputenc}
\usepackage{amsmath,amsopn,amssymb,amsthm, wasysym}
\usepackage[backref=page, breaklinks=true,colorlinks=true,linkcolor=blue,citecolor=blue,urlcolor=blue]{hyperref}

\usepackage[english]{babel}
\usepackage{graphicx}

\renewenvironment{proof}[1][Proof]{\textbf{#1.} }{\ \rule{0.5em}{0.5em}}

\DeclareMathOperator{\Isom}{Isom}

\DeclareMathOperator{\Id}{Id}

\DeclareMathOperator{\Lin}{Lin}
\DeclareMathOperator{\Symm}{Symm}

\DeclareMathOperator{\I}{I}

\DeclareMathOperator{\conv}{conv}

\renewenvironment{proof}[1][Proof]{\textbf{#1.} }
{\ \rule{0.5em}{0.5em}}

\newtheorem{theorem}{Theorem}
\newtheorem{prop}{Proposition}
\newtheorem{lemma}{Lemma}
\newtheorem{corollary}{Corollary}

\newtheorem{problem}{Problem}

\theoremstyle{definition}
\newtheorem{definition}{Definition}
\newtheorem{remark}{Remark}
\newtheorem{example}{Example}

\begin{document}
\selectlanguage{english}

\title
[On $m$-point homogeneous polyhedra in $3$-dimensional Euclidean space]
{On $m$-point homogeneous polyhedra \\in $3$-dimensional Euclidean space}

\author{V.N.~Berestovski\u\i, Yu.G.~Nikonorov}

\address{Berestovski\u\i\  Valeri\u\i\  Nikolaevich \newline
Sobolev Institute of Mathematics of the SB RAS, \newline
4 Acad. Koptyug Ave., Novosibirsk, 630090, RUSSIA}
\email{vberestov@inbox.ru}

\address{Nikonorov\ Yuri\u\i\  Gennadievich\newline
Southern Mathematical Institute of VSC RAS \newline
53 Vatutina St., Vladikavkaz, 362025, RUSSIA}
\email{nikonorov2006@mail.ru}

\thanks{The work of the first author was carried out within the framework of the State Contract to the \\ IM SB RAS, project
FWNF-2022-0006.}

\begin{abstract}
This paper is devoted to the study of the $m$-point homogeneity property
for the vertex sets of convex polytopes in Euclidean spaces.
In particular, we present the classifications of $2$-point and $3$-point homogeneous polyhedra in $\mathbb{R}^3$.

\vspace{2mm}

\vspace{2mm} \noindent Key words and phrases:
Archimedean solid,
finite homogeneous metric space,  $m$-point homogeneous metric space, $2$-point homogeneous polytope,
 Platonic solid,  regular polytope,
semiregular polytope.
\end{abstract}

\maketitle

\section{Introduction}\label{sec.1}

The main object of our study are finite metric spaces with special properties.
A finite metric space $(M,d)$ is homogeneous if its isometry group $\Isom(M)$  acts transitively on
$M$.
A finite homogeneous metric subspace of Euclidean space
represents the vertex set of a compact convex polytope
with isometry group that is transitive on the vertex set;
in each case, all vertices lie on a sphere.
We emphasize that in this paper we work only with convex polyhedra.
In \cite{BerNik19, BerNik21, BerNik21n}, the authors  obtained the complete description of the metric properties
of the vertex sets of regular and semiregular polytopes
in Euclidean spaces from the point of view of the normal homogeneity and the Clifford--Wolf homogeneity, see also surveys \cite{BerNik21nn} and \cite{BerNik24}.
Perfect and almost perfect homogeneous polytopes are studied in \cite{BerNik23}.
Some properties of $m$-point homogeneous finite subspaces of Euclidean spaces were discussed in \cite{BerNik22}.

\begin{definition}\label{de:mpoint}
A metric space $(M,d)$ is called {\it $m$-point homogeneous}, $m \in \mathbb{N}$, if for every pair of $m$-tuples
$(A_1, A_2,\dots, A_m)$ and $(B_1,B_2,\dots, B_m)$  of elements of $M$ such that $d(A_i,A_j)=d(B_i,B_j)$ for all $i,j =1,\dots, m$,
there is an isometry $f \in \Isom(M)$  with the following property: $f(A_i)=B_i$ for all $i=1,\dots, m$.
\end{definition}

In this paper, we also deal with finite subsets of Euclidean space~$\mathbb{R}^n$ with some degree of homogeneity.
It is assumed that any such set $M$ is supplied with the metric $d$
induced from $\mathbb{R}^n$, hence, is a finite metric space itself.

Since the barycenter of a finite system of point masses (all of the same mass) in any Euclidean space is preserved under any bijection
(in particular, any isometry)
of this system, we have the following result (see, e.~g, \cite[Theorem 3.41]{Cox63} or \cite{BerNik19}).

\begin{prop} \label{pr.efhs}
Let $M=\{x_1, \dots, x_q\}$, $q\geq n+1$, be a finite homogeneous metric subspace of Euclidean
space $\mathbb{R}^n$, $n\geq 2$. Then
$M$ is the vertex set of  a convex polytope~$P$, that is situated in some sphere in $\mathbb{R}^n$ with radius
$r>0$ and center $x_0=\frac{1}{q}\cdot\sum_{k=1}^{q}x_k$.
In particular, $\Isom(M,d)\subset O(n)$.
\end{prop}

This result shows that {\it the theory of convex polytopes} is very important for the study of finite homogeneous subspaces of Euclidean spaces.
\smallskip

Let $P$ be a non-degenerate convex polytope in $\mathbb{R}^n$  with the barycenter at the origin $O=(0,0,\dots,0)\in \mathbb{R}^n$
(the property to be non-degenerate means that $\Lin(P)=\mathbb{R}^n$ or, equivalently, $O$ is in the interior of $P$).
The symmetry group $\Symm(P)$ of $P$ is the group of isometries of $\mathbb{R}^n$ that preserve $P$. It is clear that
each $\psi \in \Symm(P)$ is an orthogonal transformation of $\mathbb{R}^n$ (obviously, $\psi(O)=O$ for any symmetry $\psi$ of~$P$).

If $M$ is the vertex set of a polytope $P$ supplied with the metric $d$ induced from the Euclidean one in $\mathbb{R}^n$,
then the isometry group $\Isom(M,d)$ of the metric space $(M,d)$ is the same as $\Symm(P)$.

\begin{definition}\label{de_poly_shom}
A polytope~$P$ in~$\mathbb{R}^n$
is {\it homogeneous\/} (or {\it vertex-transitive}) if
its symmetry (isometry) group $\Symm(P)$ acts transitively on the set of its vertices.
\end{definition}

We easily get the following corollary from Proposition \ref{pr.efhs}.

\begin{corollary}\label{co.inscr.face.1}
Let $P$ be a non-degenerate homogeneous polytope in $\mathbb{R}^n$ and $F$ a face.
Then $F$ is an inscribed polytope in the plane $L$ of minimal dimension, containing $F$, i.~e., all vertices of $F$ are situated on some sphere in $L$.
\end{corollary}

\begin{proof} Since $P$ is homogeneous, then it is inscribed by Proposition \ref{pr.efhs},
i.~e., all vertices of $P$ are situated on some sphere $S$ in $\mathbb{R}^n$.
It is clear that all vertices of $F$ are situated on the set $S\cap L$ which is a sphere in $L$.
\end{proof}

\begin{definition}\label{de_poly_hom}
A convex polytope $P$ in $\mathbb{R}^n$ is called $m$-point homogeneous if its vertex set {\rm(}with induced metric $d$ from $\mathbb{R}^n${\rm)}
is $m$-point homogeneous.
\end{definition}

In particular, the property to be $1$-point homogeneous means just the property to be homogeneous.
Recall that regular as well as semiregular polytopes (see the relevant definitions in Section~\ref{sec.2})
in Euclidean spaces are homogeneous.

It is natural to consider the following classification problem.

\begin{problem}[\cite{BerNik22}]\label{pr:kpoin.euclid}
Classify all convex polytopes $P$ in $\mathbb{R}^n$ whose vertex sets are $m$-point homogeneous, where $m \geq 2$.
\end{problem}

We started to study systematically $m$-point homogeneus polytops in Euclidean spaces in \cite{BerNik22}, where we obtain some important results.
For instance, it is proved that the vertex set of the $n$-cube is not $4$-point homogeneous for $n\geq 4$ \cite[Proposition 5.2]{BerNik22}.
In this regard, it should be noted that the vertex set of the square or the three-dimensional cube are $m$-point homogeneous for any $m\geq 4$.

The main goal of this paper is to obtain a complete classification of $m$-point homogeneous polyhedra in three-dimensional Euclidean space for $m \geq 2$.
Since all such polyhedra are homogeneous, then we need a detailed description of the class of homogeneous $3$-dimensional polyhedra.

For our purposes, we could apply the well-known classification of convex uniform polyhedra in $3$-dimensional Euclidean space.
Recall that uniform polyhedra  (not necessarily convex) in $3$-dimensional Euclidean space, are vertex-transitive polyhedra, allowing only
regular convex polygons and star-polygons as faces.
A corresponding list of $75$ uniform polyhedra (excluding two infinite sequences of prisms and antiprisms)
in $3$-dimensional Euclidean space was obtained by Coxeter, Longuet-Higgins,
and Miller in 1954 \cite{CLHM53}. Sopov \cite{Sop70} proved completeness of the list of 75 polyhedra in 1970. Independently,
Skilling \cite{Ski75} also came to the same conclusion in 1975.

For the class of convex uniform polyhedra (the same class that forms convex semiregular polyhedra), the enumeration
was obtained by Coxeter back in 1940, see \S 1.5 of~\cite{Cox40}, see also some discussion in \cite{BerNik21n}.
The famous Wythoff’s construction is very useful for describing them.
Related works
by Coxeter \cite{Cox85, Cox88, Cox88n}, Gr\"{u}nbaum and Shephard \cite{GrShe81, GrShe84}, and Leopold \cite{Leo17}
further elucidate this approach (and its limitations) in more general settings.
Since Gr\"{u}nbaum and Shephard’s investigation of self-intersection-free
polyhedra with positive genus and vertex-transitive symmetry in 1984 \cite{GrShe84}, the question
of complete classification of such objects in $3$-dimensional Euclidean space has been open.

It should also be noted that convex homogeneous polyhedra in $3$-dimensional Euclidean space can have faces that are not regular polygons.
A detailed description of the class of $3$-dimensional homogeneous polyhedra can be found in \cite{GrShe81}, \cite{RobCar}, and \cite{RobCarMor}.

However, we decided to choose a different approach. We chose to obtain all the properties of homogeneous polyhedra necessary for our study
independently of the classical sources mentioned above. This allows us to get
a more self-contained presentation and an alternative derivation of some well-known properties of homogeneous polyhedra.
The corresponding results are presented in
Section \ref{sec.7}.

It should be noted that any $n$-point homogeneous polytope in $\mathbb{R}^n$ is $m$-point homogeneous for every $m \in \mathbb{N}$, see
Corollary~\ref{co:if-homog}. Therefore, for homogeneous polyhedra in $\mathbb{R}^3$, it suffices to check $m$-point homogeneity only for $m=2,3$.
We know the following results.

\begin{theorem}[\cite{BerNik22}]\label{th:reg_pol}
The tetrahedron, cube, octahedron, and icosahedron are $m$-point homogeneous polyhedra for every natural $m$, while
the dodecahedron is $2$-point homogeneous but is not $3$-point homogeneous.
\end{theorem}

\begin{theorem}[\cite{BerNik22}]\label{th:edge}
If all edges of a given $2$-point homogeneous polyhedron $P\subset \mathbb{R}^3$ have the same length, then $P$ is either the cuboctahedron or a regular polyhedron.
\end{theorem}

\begin{corollary}\label{co.archimed.1}
All semiregular polyhedra except the cuboctahedron are homogeneous but are not $2$-point homogeneous.
\end{corollary}


As a rule, we use standard notation. Otherwise, an explanation of the notation is given in the text of the paper.
When we deal with some metric space $(M,d)$,
for given $c\in M$ and $r \geq 0$, we consider $S(c,r)=\{x\in M\,|\, d(x,c)=r\}$,
the sphere with the center $c$ and radius $r$. For $x,y \in \mathbb{R}^n$, the symbol $[x,y]$ means the closed interval in $\mathbb{R}^n$
with the ends $x$ and $y$.
\smallskip

The paper is organized as follows.
In Section \ref{sec.2} we consider some important examples of homogeneous polytopes.
In Section \ref{sec.3} we discuss stable facets of homogeneous polytopes.
In Section \ref{sec.4} we recall some important properties of $2$-point homogeneous polytopes.
Section \ref{sec.5} is devoted to $2$-dimensional $m$-point homogeneous polygons.
In Section \ref{sec.6} we discuss the $m$-point homogeneity property of right prisms and antiprisms, as well as of some generalizations.
In Section \ref{sec.7} we discuss some metric properties of homogeneous polyhedra.
Finally, in Section \ref{sec.8} we obtain the classification of $2$-point homogeneous and the classification of $3$-point homogeneous
polyhedra (Theorems \ref{th.3dim.2point}, and \ref{th.3dim.mpoint}, respectively).
\smallskip

\begin{figure}[t]
\vspace{5mm}
\begin{center}
\begin{minipage}[h]{0.3\textwidth}
\center{\includegraphics[width=0.8\textwidth, trim=0in 0in 0mm 0mm, clip]{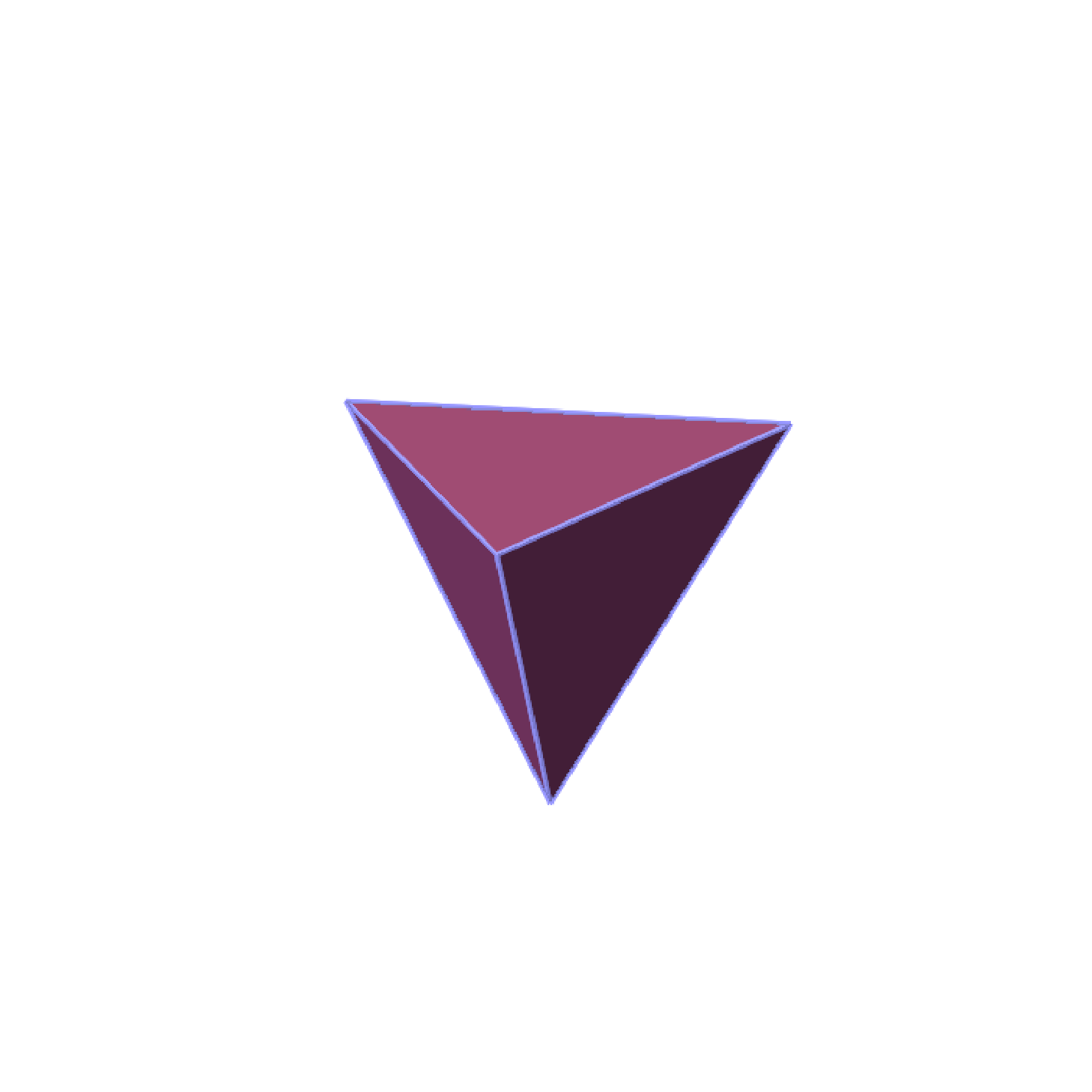} \\ a)}
\end{minipage}
\begin{minipage}[h]{0.3\textwidth}
\center{\includegraphics[width=0.7\textwidth, trim=0in 0in 0mm 0mm, clip]{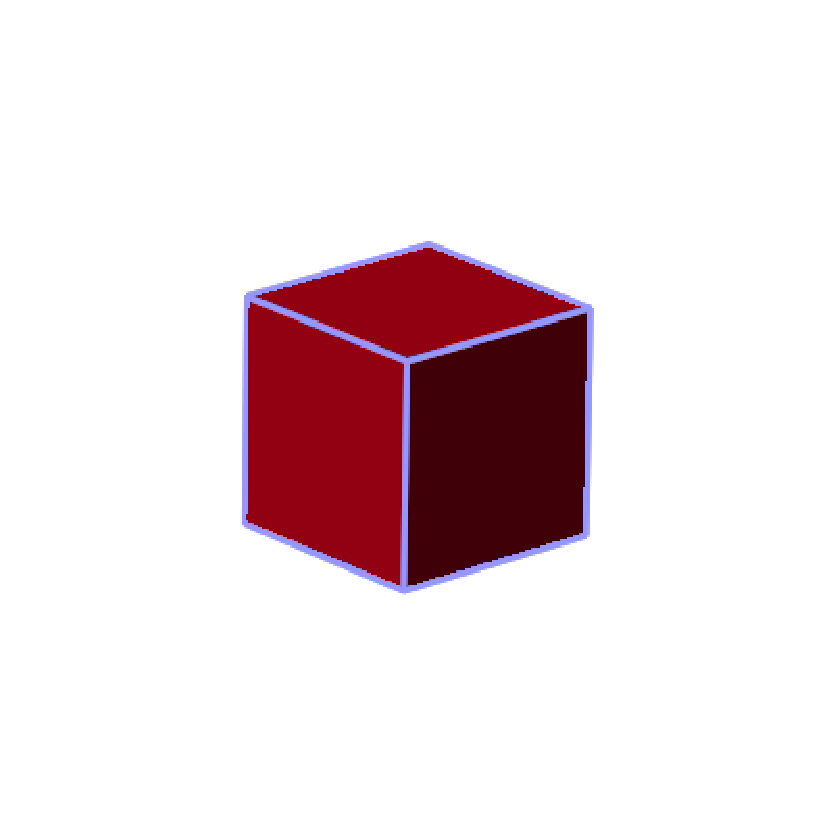} \\ b)}
\end{minipage}
\begin{minipage}[h]{0.3\textwidth}
\center{\includegraphics[width=0.7\textwidth, trim=0in 0in 0mm 0mm, clip]{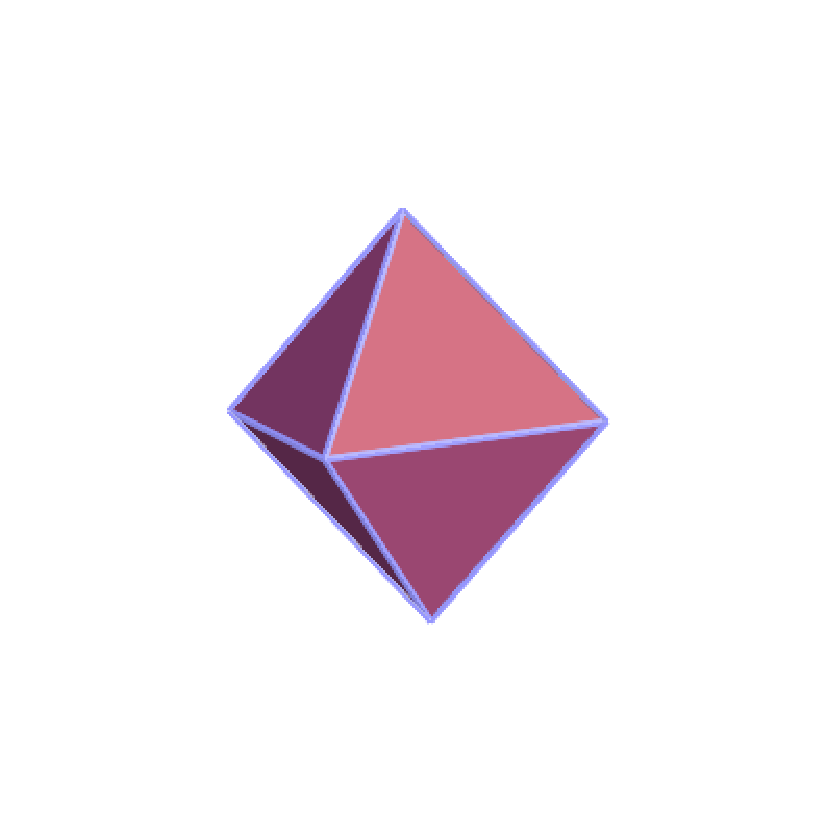} \\ c)}
\end{minipage}
\vfill
\begin{minipage}[h]{0.3\textwidth}
\center{\includegraphics[width=0.7\textwidth, trim=0in 0in 0mm 0mm, clip]{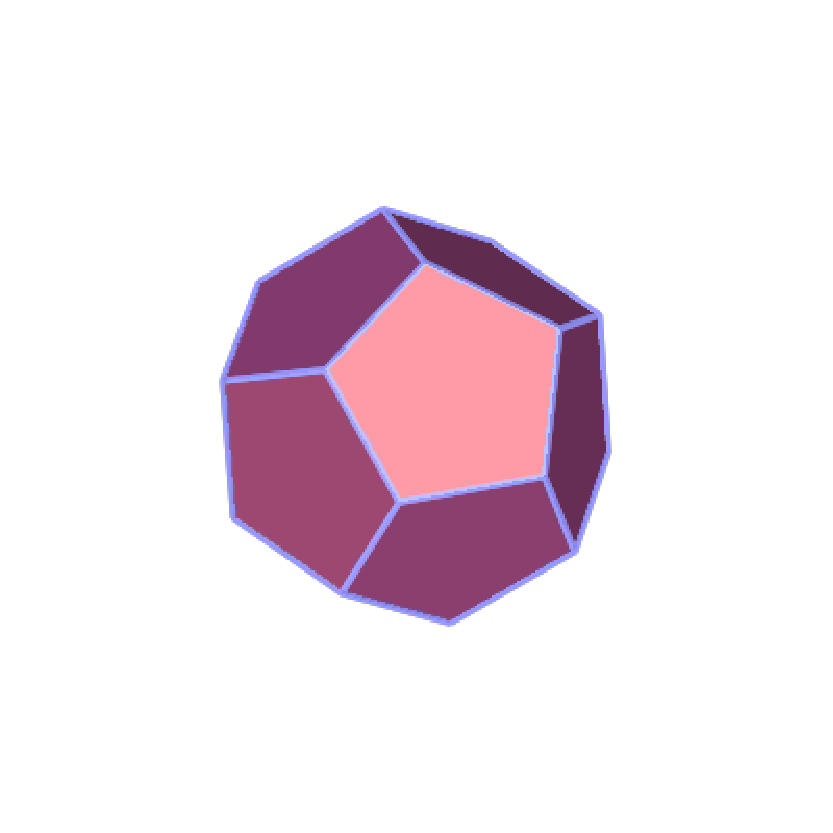} \\ d)}
\end{minipage}
\begin{minipage}[h]{0.3\textwidth}
\center{\includegraphics[width=0.8\textwidth, trim=0in 0in 0mm 0mm, clip]{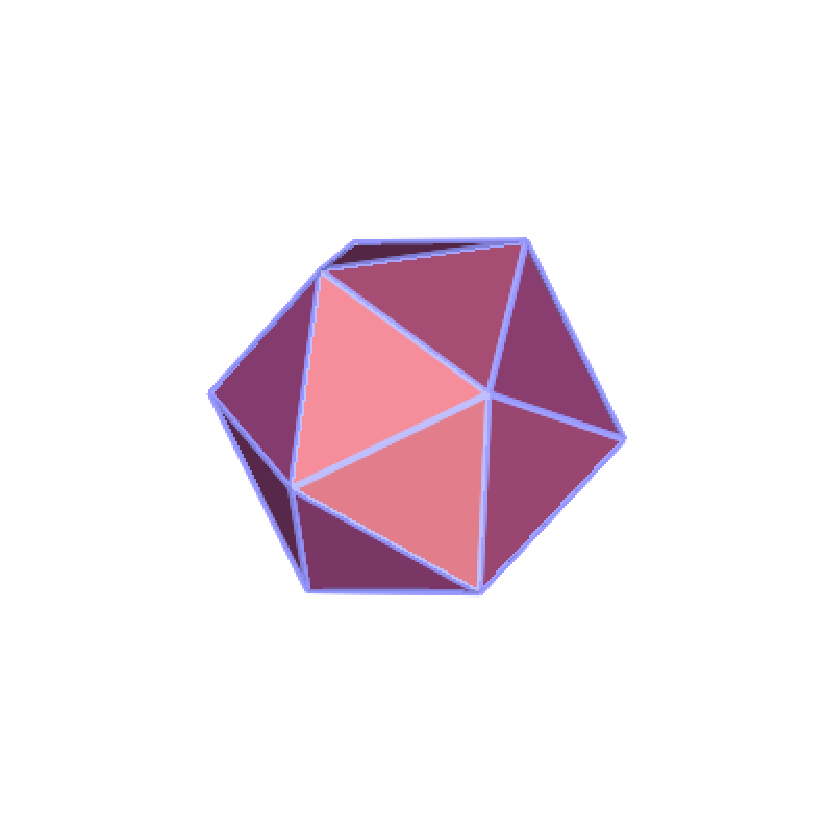} \\ e)}
\end{minipage}
\begin{minipage}[h]{0.3\textwidth}
\center{\includegraphics[width=0.8\textwidth, trim=0in 0in 0mm 0mm, clip]{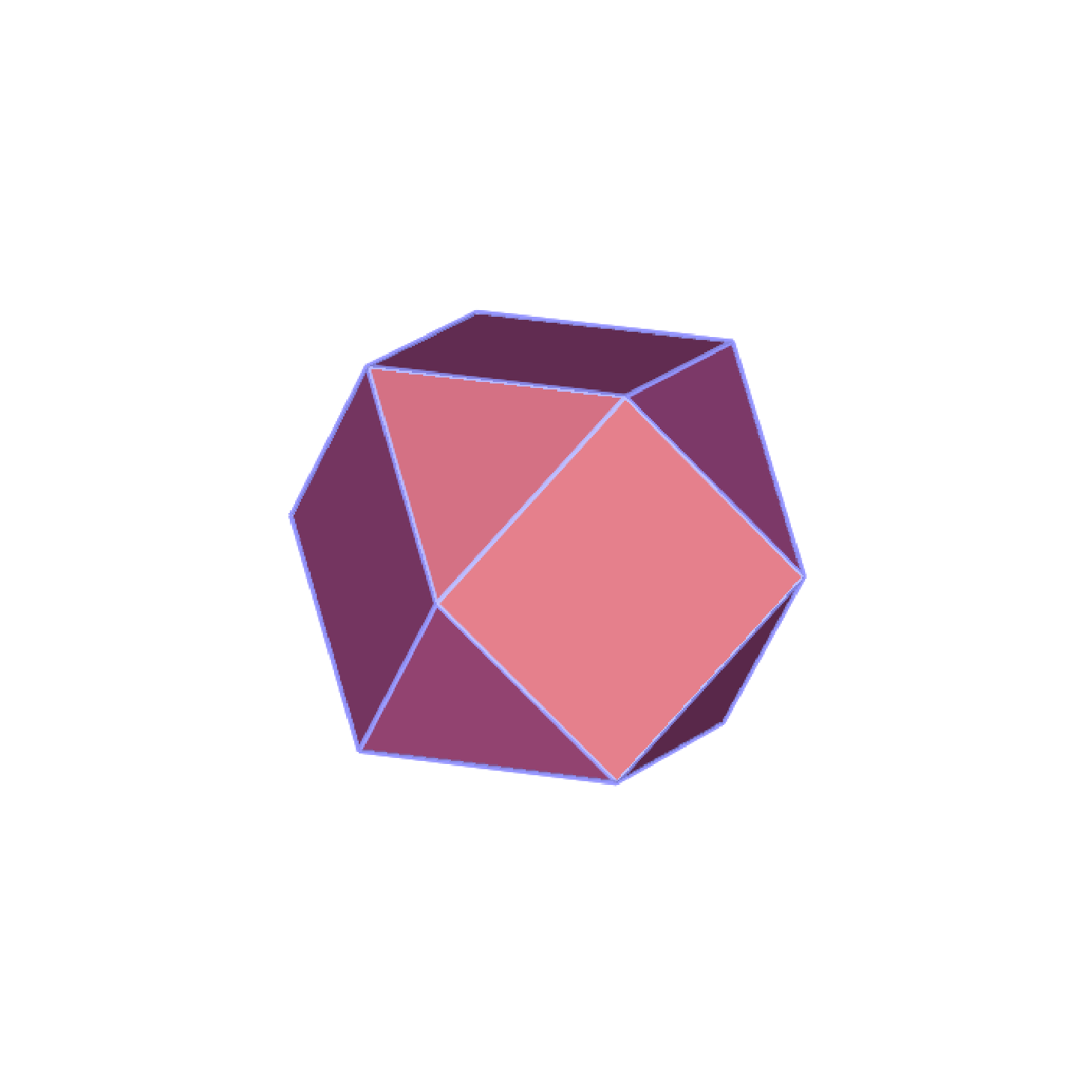} \\ f)}
\end{minipage}
\caption{{{\normalsize{ Platonic solids and cuboctahedron:}}
\normalsize{a) tetrahedron; b) cube (hexahedron); c) octahedron; d) dodecahedron; e) icosahedron, f) cuboctahedron.}}}
\label{Fig1}
\end{center}
\end{figure}

\section{Some natural examples}\label{sec.2}

Well known examples of homogeneous polytopes in Euclidean spaces are {\it regular polytopes}.
A one-dimensional polytope is a closed segment,
bounded by two endpoints. It is considered regular by definition.
Two-dimensional regular polytopes are regular
polygons on Euclidean plane.
A convex $n$-dimensional polytope for $n \geq 3$ is called {\it regular},
if it is homogeneous
and all its facets are regular polytopes congruent to each other, and
this definition is equivalent to other definitions of regular convex
polytopes (see~\cite{Mar}).
It is well known that there are only five regular convex three-dimensional
polyhedra:
the tetrahedron, cube, octahedron, dodecahedron, and icosahedron.
These polyhedra are traditionally called the {\it Platonic solids\/}
(see Fig.\,\ref{Fig1}).

We also recall the definition of a wider class of semiregular convex polyhedra.
For $n=1$ and $n=2$, semiregular polytopes are defined as regular.
A convex $n$-dimensional polytope for $n\geq 3$ is called
{\it semiregular} if it is homogeneous
and all its facets are regular polytopes.

In three-dimensional space (in addition to the Platonic solids), there are
the following semiregular polyhedra:
the 13 {\it Archimedean solids\/} (see more details and related references in \cite{BerNik21}) and two infinite sequences
of right prisms and antiprisms.

\begin{figure}[t]
\begin{minipage}[h]{0.45\textwidth}
\center{\includegraphics[width=0.7\textwidth, trim=0in 0in 0mm 0mm, clip]{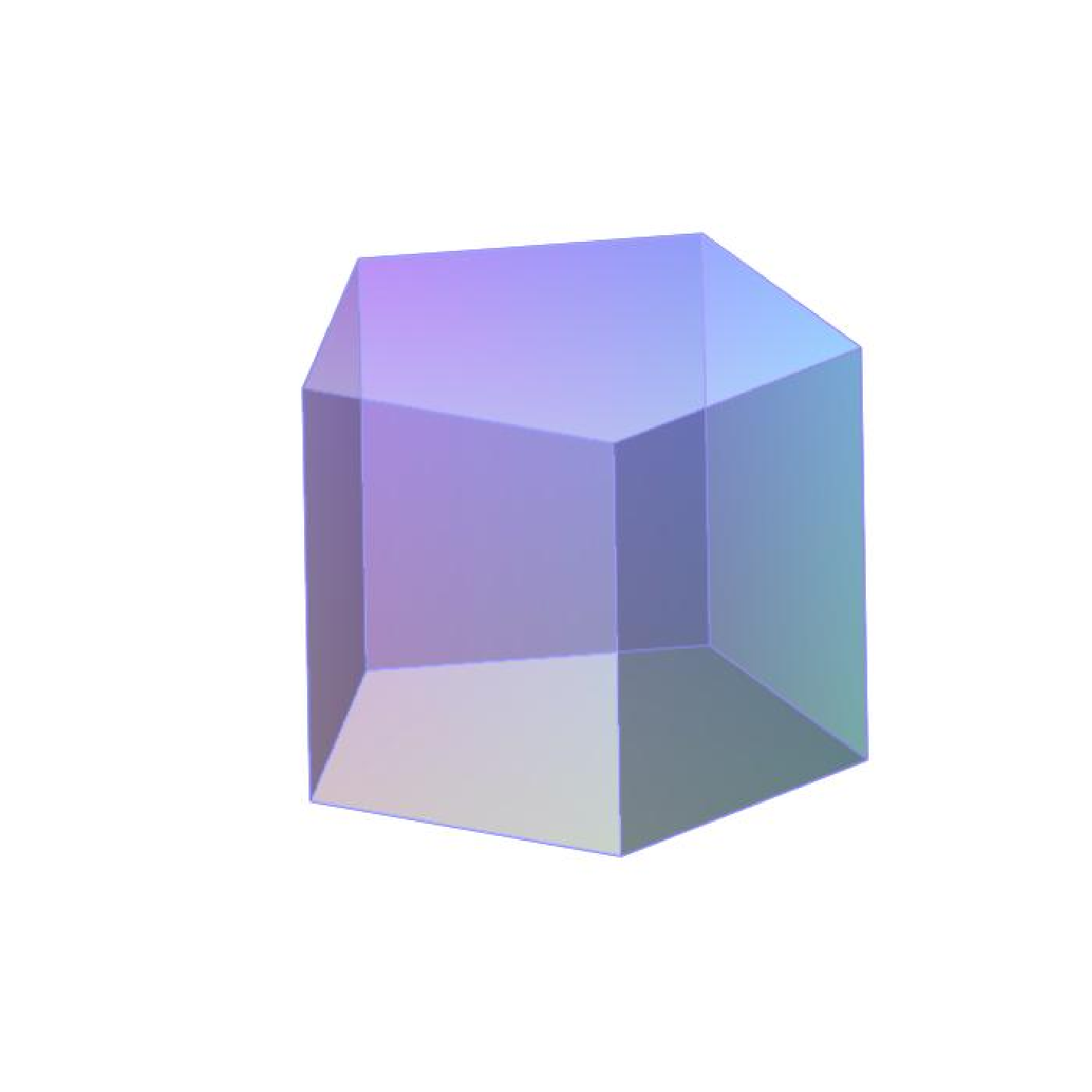} \\ a)}
\end{minipage}
\begin{minipage}[h]{0.45\textwidth}
\center{\includegraphics[width=0.9\textwidth, trim=0in 0in 0mm 0mm, clip]{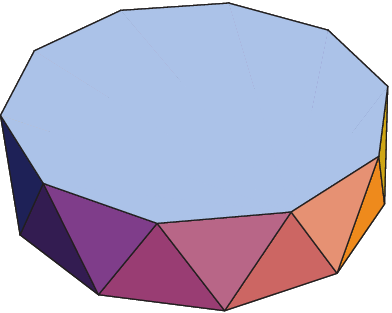} \\ b)}
\end{minipage}
\caption{{{\normalsize{Homogeneous polyhedra:}}
\normalsize{a) right prism with 5-gonal regular bases;
b)~right $10$-gonal antiprism.}}}
\label{Fig3}
\end{figure}

A {\it right prism} is a polyhedron whose two faces (called {\it bases}) are congruent (equal) regular polygons,
lying in parallel planes, while the remaining faces (called {\it lateral} ones) are rectangles (perpendicular to
the bases), see Fig.\,\ref{Fig3}~a).
A {\it homogeneous prism} is a right prism with homogeneous bases (see Lemma \ref{le.homepolygon.1} and Theorem \ref{th.2dim.1} that describe all homogeneous polygons).
A {\it regular prism} is a right prism with regular bases.
If in addition the lateral faces are squares then the prism under consideration is a semiregular convex polyhedron.

A {\it right antiprism} is a polyhedron, whose two parallel faces ({\it bases}) are equal regular $n$-gons twisted by an angle $\pi/n$,
with $2n$ ({\it lateral}) faces that are isosceles triangles. It is called also {\it a right $n$-gonal antiprism}, see Fig.\,\ref{Fig3}~b).
If the lateral faces are equilateral triangles then the antiprism under consideration is called {\it uniform}  and is a semiregular convex polyhedron.
\smallskip

Many interesting examples of $m$-point homogeneous polytopes were constructed in~\cite{BerNik22}.
For example, any $n$-cube, $n=1,2,3$, is $m$-point homogeneous for every $m \in \mathbb{N}$, while any $n$-cube, $n\geq 4$, is $3$-point homogeneous but is not
$4$-point homogeneous, see \cite[Proposition 5.4]{BerNik22}. Moreover, in \cite{BerNik22}, the authors studied
the standard $n$-dimensional demihypercubes $DC_n$, $n\geq 1$, as the convex hull
of the points $(\pm 1, \pm 1, \dots, \pm 1) \in \mathbb{R}^n$, where the quantity of the signs ``$-$'' is even;
the truncated $n$-dimensional simplex $TS_n$ for any $n\geq 1$; Gosset semiregular polytopes; and some other special polytopes.
\smallskip

In what follows we consider other interesting examples of homogeneous polytopes in~$\mathbb{R}^n$.

\begin{example}\label{ex.tetra.1}
Let us consider a tetrahedron (a simplex) $S$ in $\mathbb{R}^3$ with three given  edge lengths $0 < a \leq b \leq c$ such that every pair of edges without
common vertex have equal Euclidean lengths and each of the number $a, b, c$ is the length of edges
for some pair, see Fig.~\ref{Fig2}~a). It is easy to prove that $S$ is homogeneous (see Example \ref{ex.new.1}).

If $\alpha, \beta, \gamma$ are angles opposite to the sides of the triangle with lengths $a \leq b \leq c$, then $\gamma<\pi/2$;
this is equivalent to the inequality $c^2<a^2+b^2$.
Indeed, since the tetrahedron $S$ is homogeneous, then every vertex has exactly three adjacent faces with the incident angles $\alpha, \beta, \gamma$.
It is clear that $\pi-\gamma =\alpha+\beta> \gamma$ (the angles $\alpha$, $\beta$, and $\gamma$ are plane angles of a trihedral angle).
See also, Remark \ref{re.hom.simp}.
On the other hand, if $0 < a \leq b \leq c$ and $\alpha+\beta > \gamma$, then the above homogeneous tetrahedron $S$ does exist.
\end{example}

\begin{example}\label{ex.new.1}
The isometries of $\mathbb{R}^n$ of the type $(x_1,x_2,\cdots,x_n) \mapsto (\pm x_1, \pm x_2,\cdots, \pm x_n)$ with all possible choices of the signs constitute the group
$G$ isomorphic to $\mathbb{Z}^n_2$. If we consider the above isometries only with even numbers of ``$-$'', then we obtain a subgroup $G^{\ast}$ of $G$, that is isomorphic to
$\mathbb{Z}^{n-1}_2$.

Now, if we take a point $A=(a_1,a_2,\cdots, a_n)$ with all $a_i$ positive, then the orbit of this point by the action of  $G$ is the vertex set of a parallelotope
$[-a_1,a_1]\times[-a_2,a_2]\times \cdots \times [-a_n,a_n]$. The orbit of  $A$ by the action of $G^{\ast}$ is the vertex set of a homogeneous polytope $P$,
that is a hemihypercube for $a_1=a_2=\cdots =a_n$. To prove the homogeneity of $P$ let us consider any point $B=(b_1, b_2,\cdots, b_n)$
from the orbit of $A$ under $G^{\ast}$.
If $I:=\{i=1,2,\dots n\,|\, b_i=-a_i\}$, then $I$ has even numbers of elements and $b_j=a_j$ for any $j\not\in I$.
Let us consider a map $\eta: (x_1,x_2,\cdots,x_n) \mapsto (\pm x_1, \pm x_2,\cdots, \pm x_n)$ such that $\pm x_i = - x_i$ if and only if $i \in I$.
It is clear that $\eta$ is an isometry of $P$ and $\eta(B)=A$. This implies that the polytope $P$ is homogeneous.

The polytope $P$ is a segment for $n=2$. For $n=3$, a polyhedron $P$ is a homogeneous simplex with the vertices
$A_1=(k,l,m)$, $A_2=(-k,-l,m)$, $A_3=(-k,l,-m)$, $A_4=(k,-l,-m)$,
where $k,l,m \in \mathbb{R}$.

If $d_{ij}=d(A_i,A_j)$, then $d_{12}=d_{34}=2\sqrt{k^2+l^2}$, $d_{13}=d_{24}=2\sqrt{k^2+m^2}$, and $d_{14}=d_{23}=2\sqrt{l^2+m^2}$.

Obviously, any triangle with the side lengths $d_{12}$, $d_{13}$, and $d_{23}$ (for instance, the face $A_1 A_2 A_3$ of $P$)
is acute (compare with Example \ref{ex.tetra.1}). Indeed,
$$
d^2_{12}+d^2_{13}-d^2_{23}=8 k^2 >0, \quad d^2_{12}+d^2_{23}-d^2_{13}=8 l^2 >0, \quad d^2_{13}+d^2_{23}-d^2_{12}=8 m^2 >0,
$$
which implies that all angles of the triangle $A_2A_1A_3$ are acute.
Moreover, any such acute triangle has such side lengths for suitable
$k,l,m \in \mathbb{R}$. Indeed, for every $0 < a \leq b \leq c$ with $c^2<a^2+b^2$, we easily can find $0<k\leq l \leq m$ such that
$2\sqrt{k^2+l^2}=a$, $2\sqrt{k^2+m^2}=b$, and $2\sqrt{l^2+m^2}=c$.

It is clear that the center of $P$ is the origin $O=(0,0,0)$ and the polyhedron $-P$ has the vertices
$-A_1=(-k,-l,-m)$, $-A_2=(k,l,-m)$, $-A_3=(k,-l,m)$, $-A_4=(-k,l,m)$. In particular, $\conv \bigl(P\bigcup (-P)\bigr)=[-k,k]\times[-l,l]\times [-m,m]$.
In the general case we have  $\conv\bigl(P\bigcup (-P)\bigr)=[-a_1,a_1]\times[-a_2,a_2]\times \cdots \times [-a_n,a_n]$.
\end{example}

\begin{figure}[t]
\begin{minipage}[h]{0.45\textwidth}
\center{\includegraphics[width=0.8\textwidth, trim=0in 0in 0mm 0mm, clip]{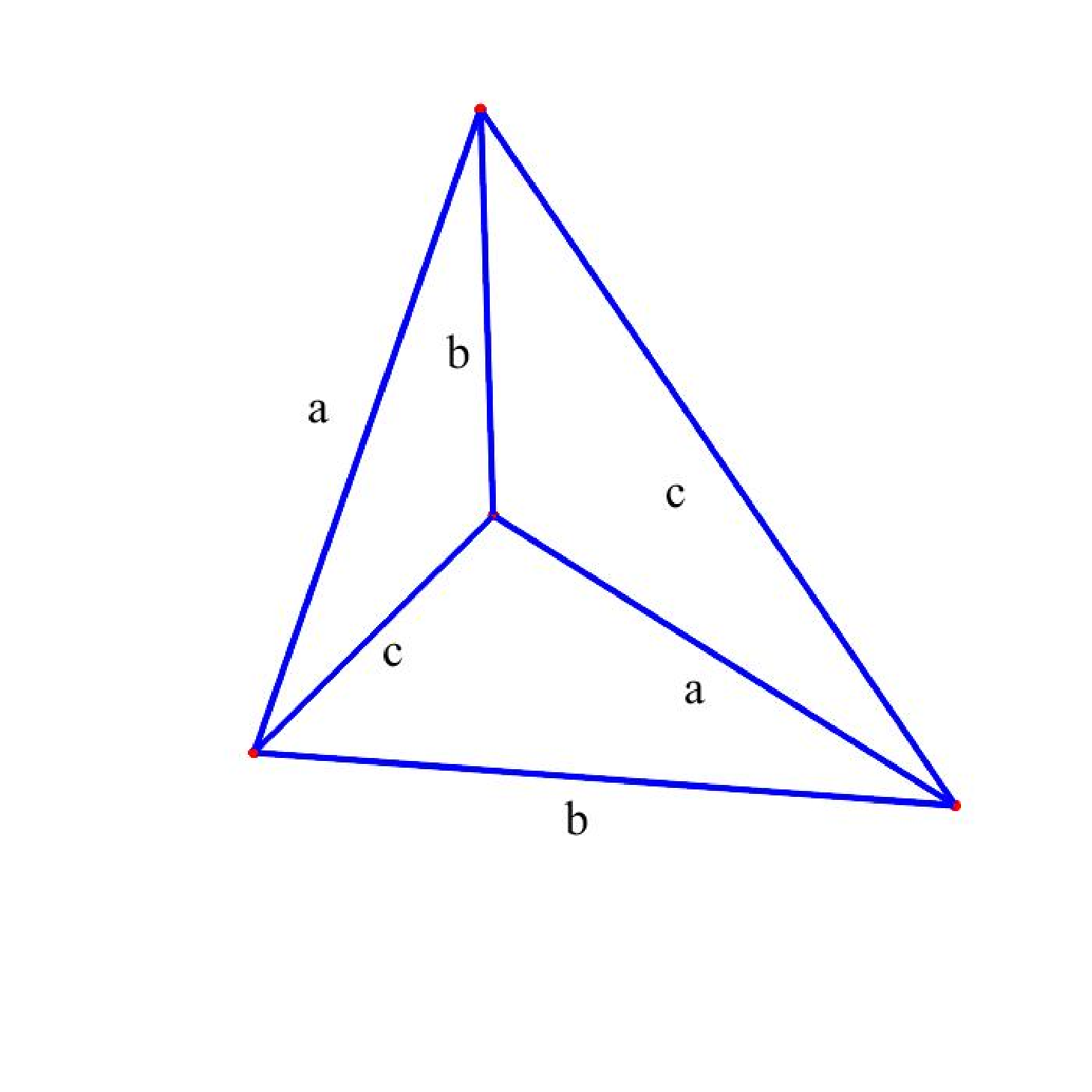} \\ a)}
\end{minipage}
\begin{minipage}[h]{0.45\textwidth}
\center{\includegraphics[width=0.9\textwidth, trim=0in 0in 0mm 0mm, clip]{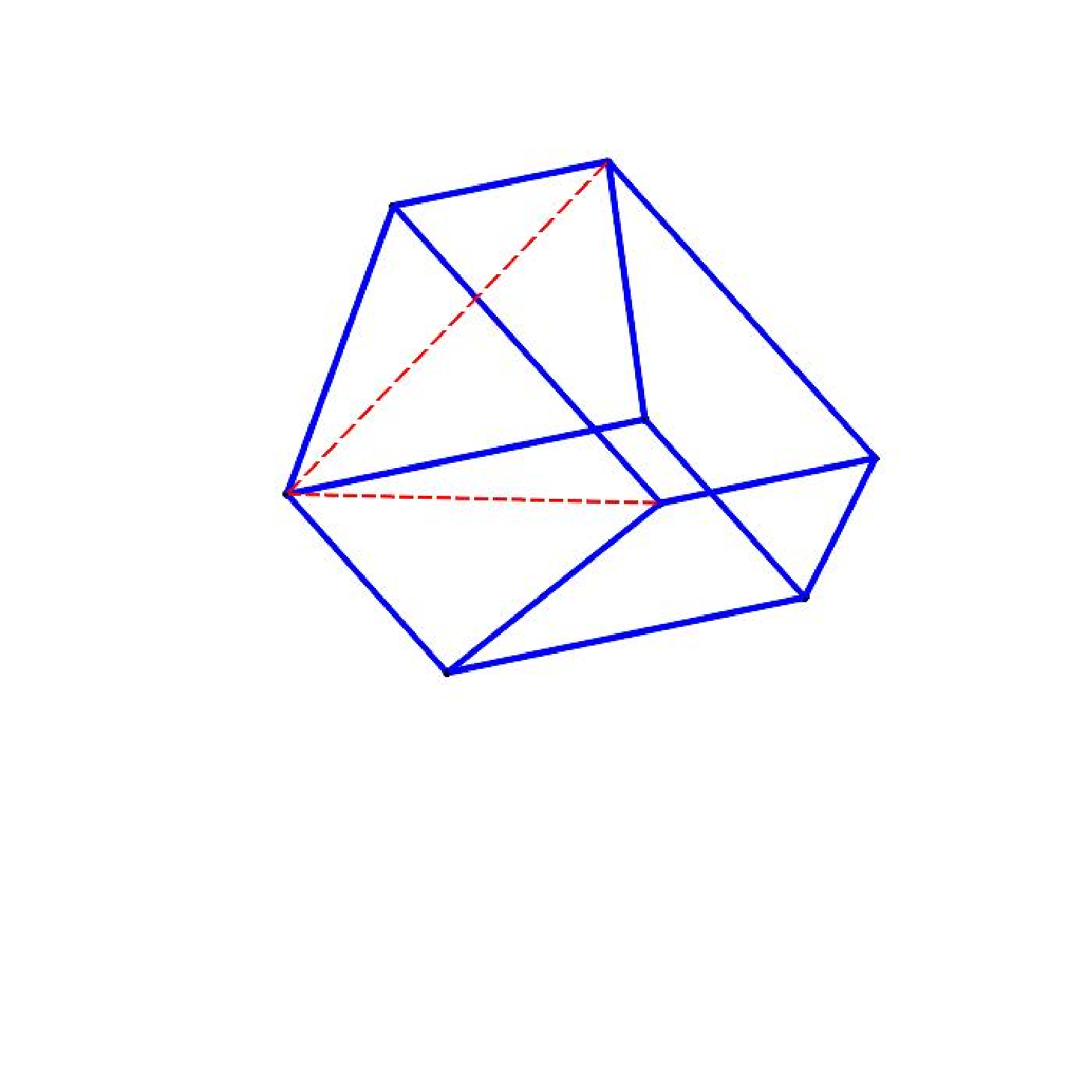} \\ b)}
\end{minipage}
\caption{{{\normalsize{Homogeneous polyhedra:}}
\normalsize{a) homogeneous tetrahedron;
b)~antiprism with nonregular bases.}}}
\label{Fig2}
\end{figure}

\begin{example}\label{ex.par_trun1}
Let $P$ be a regular polyhedron in $\mathbb{R}^3$ with edges of unit length. Let us fix a number $t\in (0,1/2)$ and consider the convex hull $P_t$ of all points
of edges of $P$ on the distance $t$ from the nearest vertex of $P$. It is easy to check that $P_t$ is a homogeneous polyhedron,
which will be called partial truncation
of $P$ (with parameter~$t$).
For the tetrahedron $T$, the polyhedron $T_t$ has four homogeneous $6$-gonal faces and four triangular faces. For the cube $C$, $C_t$ has
$6$ homogeneous $8$-gonal faces and $8$ triangular faces. For the octahedron $O$, $O_t$ has
$8$ homogeneous $6$-gonal faces and $6$ square faces. For the dodecahedron $D$, $D_t$ has
$12$ homogeneous $10$-gonal faces and $20$ triangular faces. For the icosahedron $I$, $I_t$ has
$20$ homogeneous $6$-gonal faces and $12$ pentagonal faces.
\end{example}

\medskip

Isosceles trapezoids are faces of several types of homogeneous polyhedra, for instance, of antiprisms on
a non-regular $2n$-gon with the vertex type $(n,n)$ (i.~e. orbits of the group $D_n$), where  $n\geq 2$; see detail in \cite{RobCar}.
The following explicit example is useful.

\begin{example}\label{ex.isotrap.1}
Let us consider a polytope $P$ that is the convex hull of the following eight points in $\mathbb{R}^3$:
$(\pm a, \pm b, 0)$ and $(\pm b, \pm a, c)$, where $b>a>0$ and $c>0$, see Fig.~\ref{Fig2}~b).
Obviously, $P$ is an antiprism on a rectangle. The bases of $P$ are two rectangles with side lengths $2a$ and $2b$.
The other four faces of $P$ are isosceles trapezoids with base lengths $2a$ and $2b$ and the two other sides of length $\sqrt{2(b-a)^2+c^2}$.
It is easy to see that $P$ is a homogeneous polyhedron.
If an isometry $\tau$ of $P$ fixes the vertex $(a,b,0)$, then it should fix the base of $P$ with this vertex, hence, all vertices of~$P$.
Therefore, the isotropy group at this (hence, every) vertex is trivial. On the other hand,  the points
$(-b,a,c)$ and $(b,-a,c)$ are situated on the sphere with center
$(a,b,0)$ and radius $\sqrt{2a^2+2b^2+c^2}$. Since the isotropy group is trivial, there is no isometry of $P$ that fixes the vertex $(a,b,0)$
and moves the point $(-b,a,c)$ to the point $(b,-a,c)$.
Therefore, $P$ is not $2$-point homogeneous.
\end{example}

Now, we consider two more important examples. Let us recall that the standard hypercube $C_n$ in $\mathbb{R}^n$ is the convex hull of the points
$(\pm 1,\pm 1,\dots,\pm 1)\in \mathbb{R}^n$. Below we consider some variants of the truncation of $C_n$.

At first we consider the convex hull of the points $(\pm 1, \dots, 0, \dots, \pm 1)$, where
exactly one coordinate is $0$, and all other are $\pm 1$.
This polytope is called the fully truncated hypercube and is denoted by $TC_n$, see details in \cite{Zef}.
It is clear that $TC_2$ is a square and $TC_3$ is the cuboctahedron,
see Fig.~\ref{Fig1}~f).

We know that   $TC_2$ and $TC_3$ are $m$-point homogeneous for all natural $m$, see \cite[Proposition 6.1]{BerNik22}. Here we prove the following results.

\begin{prop}\label{pr.trcube.1}
The fully truncated hypercube $TC_n$ is $2$-point homogeneous for $n\geq 2$.
\end{prop}

\begin{proof}
It is clear that the isometry group $\Isom(TC_n)$ of $TC_n$ coincides with the isometry group of $C_n$.
Since $C_n$ is $2$-point homogeneous, then $TC_n$ is homogeneous
(if an isometry moves one edge to an another one, then it moves the midpoint of the first edge
to the midpoint of the second edge).
Let us fix the vertex $A_0=(1,1,\dots,1,0) \in TC_n$.
The isotropy subgroup $\I(A_0)$ is generated by all substitutions of $n-1$ first coordinates and by the isometry
$(x_1,\dots, x_{n-1}, x_n) \mapsto (x_1,\dots, x_{n-1}, -x_n)$ (the multiplication of the last coordinate by $-1$).
We consider two kinds of vertices for~$TC_n$ with respect to $A_0$.

A vertex  $A\in TC_n$ is of the first kind if there are exactly one $0$ and $k$ numbers $-1$ among the first $n-1$ coordinates of $A$.
It is clear that $d(A_0,A)=\sqrt{4k+2}$ for such a vertex.
A vertex  $A\in TC_n$ is of the second kind if the $n$-th its coordinate is $0$ and there are exactly $l$ numbers $-1$ among the first $n-1$ coordinates of $A$.
It is clear that $d(A_0,A)=\sqrt{4l}$ for such a vertex.
Obviously, $\sqrt{4k+2}\neq \sqrt{4l}$ for all natural $k$ and $l$ ($1\neq 2(l-k)$).
Therefore, the vertices of different kinds could not be situated on one distant sphere $S(A_0,r)$ with the center $A_0$.

Now, if we have two vertices $A_1$ and $A_2$ of the first kind on the sphere $S(A_0,r)$ with center $A_0$ and radius $r>0$,
then these points have equal numbers of $-1$ among the first $n-1$ coordinates (this number is $(r^2-2)/4$).
Therefore, there is a substitution of the first $n-1$ coordinates that moves $0$ and all $-1$ of $A_1$ respectively to $0$ and all $-1$ of $A_2$.
Moreover, we can change the sign of the last coordinates (if it is needed). Therefore, there is an isometry $\tau \in \I(A_0)$ that moves $A_1$ to $A_2$.

If we have two vertices $A_1$ and $A_2$ of the second kind on the sphere $S(A_0,r)$ with center $A_0$ and radius $r>0$,
then these points have equal numbers of $-1$ among the first $n-1$ coordinates (this number is $r^2/4$).
Therefore, there is a substitution of the first $n-1$ coordinates that moves all $-1$ of $A_1$ to all $-1$ of $A_2$.
Therefore, in this case we also have an isometry $\tau \in \I(A_0)$ that moves $A_1$ to $A_2$.
Hence, the proposition is proved.
\end{proof}
\smallskip

\begin{figure}[t]
\begin{minipage}[h]{0.45\textwidth}
\center{\includegraphics[width=0.85\textwidth, trim=0in 0in 0mm 0mm, clip]{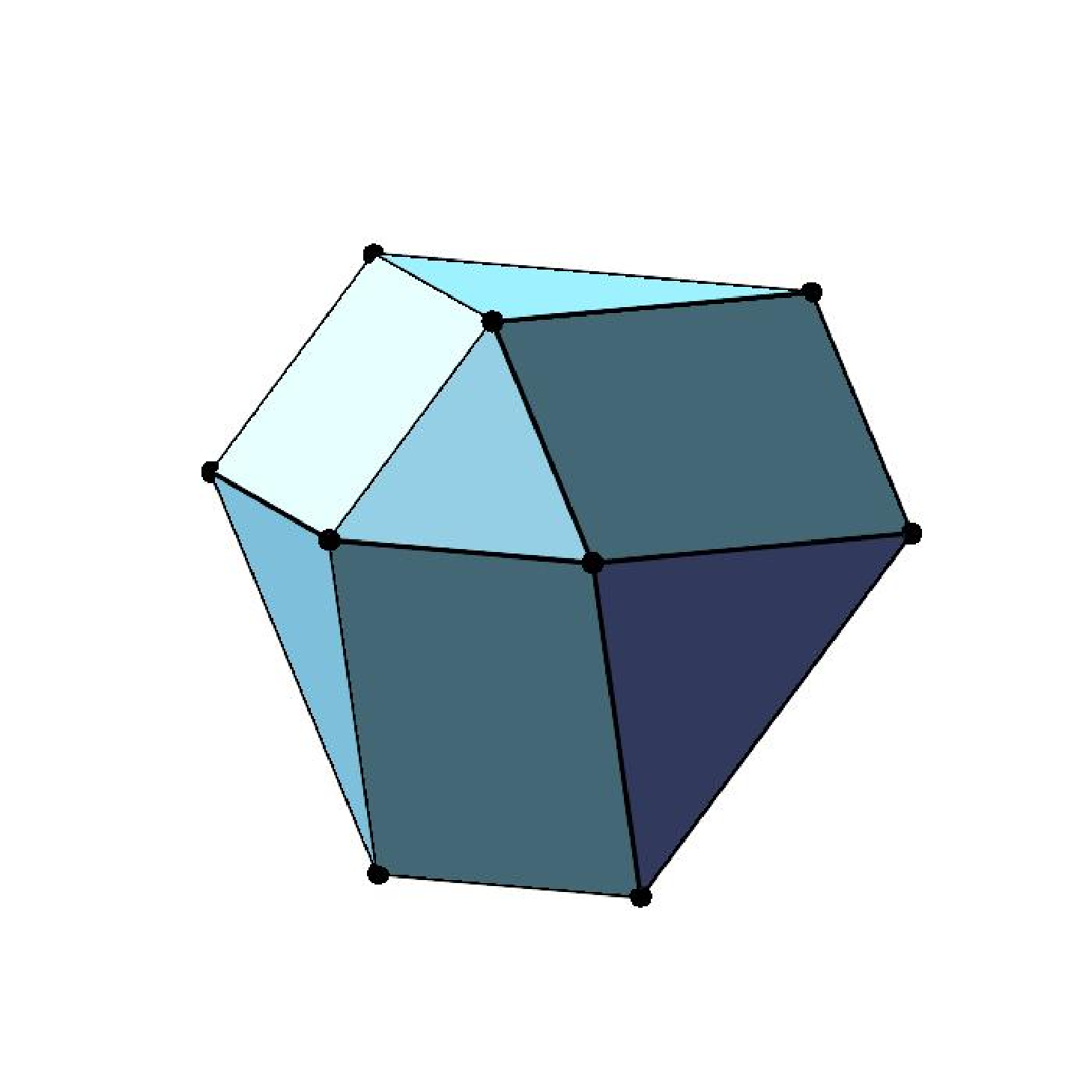} \\ a)}
\end{minipage}
\begin{minipage}[h]{0.45\textwidth}
\center{\includegraphics[width=0.78\textwidth, trim=0in 0in 0mm 0mm, clip]{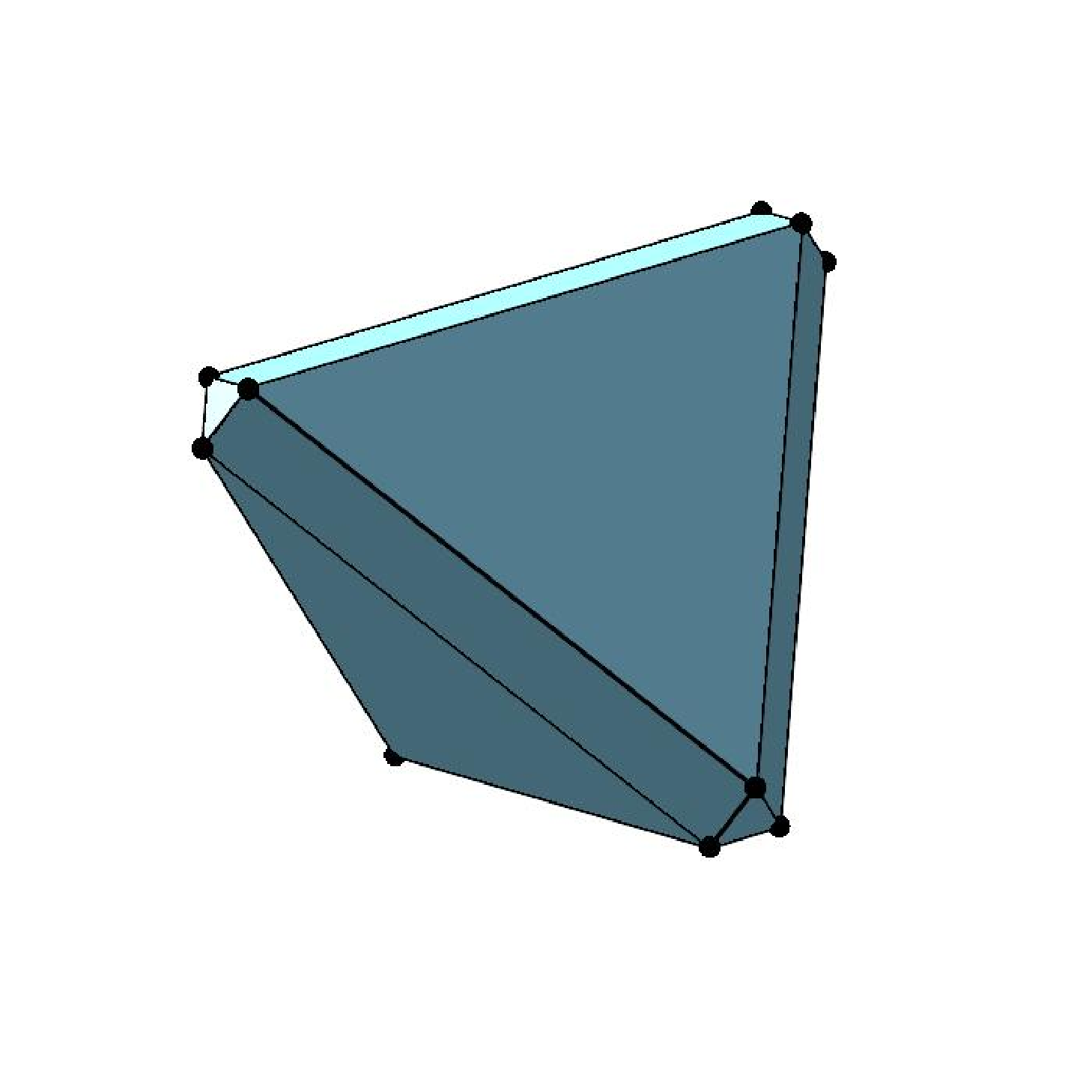} \\ b)}
\end{minipage}
\caption{{{\normalsize{Nonregularly truncated cubes $TC_3(c)$:}}
\normalsize{a) for $c=1/4$;
b)~for $c=7/9$.}}}
\label{Fig4}
\end{figure}

Now, let us fix a number $c\in(0,1)$
and consider the polyhedron $TC_n(c)$ that is the convex hull of the points $(\pm 1,\dots, \pm c,\dots,\pm 1)\in \mathbb{R}^n$,
where exactly one coordinate has the absolute value $c$, and all other are $\pm 1$, and the common quantity of the signs $``-"$ is even.

We will call any polytope $TC_n(c)$ for $c\in (0,1)$ a {\it nonregularly truncated hypercube}, see Fig.~\ref{Fig4}.
For $n=1$ we have a one-point set, for $n=2$ we get a rectangle, distinct from squares.

If we consider a limiting case $c=1$ for $TC_n(c)$, then we get the $n$-dimensional demihypercube $DC_n$.
We know, that $DC_n$ for $n = 1, 2, 3$ is $m$-point homogeneous for any natural $m$, while
$DC_n$ is $3$-point homogeneous but is not $4$-point homogeneous for all $n \geq 4$, see \cite[Corollary~6]{BerNik22}.
Another limiting case is $c=0$, where we get the fully truncated hypercube $TC_n$.

\begin{prop}\label{ex.trcube.1}
Any nonregularly truncated hypercube $TC_n(c)$, $c\in (0,1)$, is homogeneous, but is not $2$-point homogeneous for $n\geq 3$.
\end{prop}

\begin{proof}
It is clear that the isometry group of $TC_n(c)$ is a subgroup of the isometry group of $C_n$.
More exactly $\Isom(TC_n(c))$ is generated by all substitutions of coordinates and by multiplications of even quantity of coordinates by $-1$
(the same isometry group has the $n$-dimensional demihypercube $DC_n\subset C_n$).
Therefore, $TC_n(c)$ is homogeneous.

Let us fix the vertex $A_0=(1,1,\dots,1,c) \in TC_n(c)$.
The isotropy subgroup $\I(A_0)$ is the group of all substitutions of $n-1$ first coordinates.
Now, let us consider the vertices  $A_1=(1,\dots,1,-1,c,-1)$ and $A_2=(1,\dots,1,-1,-c,1)$ in  $TC_n(c)$.
It is clear that $d(A_0,A_1)=d(A_0,A_2)=\sqrt{6+2c^2}$. On the other hand, (since $0<c<1$) there is no isometry from $\I(A_0)$ that moves $A_1$ to $A_2$.
Therefore, $TC_n(c)$, $c\in (0,1)$, is not $2$-point homogeneous for $n\geq 3$ by Proposition \ref{pr:two-point homogeneous}.
\end{proof}

\begin{remark}\label{rem.5triangle1}
For  $c\in (0,1)$,  $TC_3(c)$ has $4$ faces that are regular triangles with the side length $\sqrt{2}\cdot(1-c)$,
$4$ faces that are regular triangles with the side length $\sqrt{2}\cdot(1+c)$, and $6$ faces that are rectangles with the side lengths
$\sqrt{2}\cdot(1-c)$, $\sqrt{2}\cdot(1+c)$. Each of the latter six faces can be divided in two right triangles such that any vertex of
$TC_3(c)$ is incident to five triangles. Hence we get a degenerate homogeneous polyhedron with all triangular faces some of that are right
triangles.
\end{remark}

\section{On stable faces of a homogeneous polytope}\label{sec.3}

\begin{prop}\label{pr.grav.1}
For any homogeneous polytope $P$ in $\mathbb{R}^n$, the center $O$ of $P$, that is the center of mass of all vertices (with equal masses) of $P$
is also the center of mass of $P$, supplied with the unit density.
\end{prop}

\begin{proof} Without loss of generality, we may suppose that the interior of $P$ is non-empty.
Let us denote the center of mass of $P$ by $\widetilde{O}$. It is clear that $O$ and $\widetilde{O}$ both are preserved by any isometry of $P$,
in particular, all vertices of $P$
situated on some sphere with center $O$ and on  some sphere with center $\widetilde{O}$. It is clear that these two spheres are coincide.
Hence, $\widetilde{O}=O$.
\end{proof}
\smallskip

Let us recall the following definition.

\begin{definition}\label{de.stab.1}
A facet $F$ of a convex polytope $P$ in $\mathbb{R}^n$ is called {\it stable}, if the orthogonal projection of the center of mass of $P$ to the hyperplane through
$F$ is in the interior of~$F$.
\end{definition}

It is known that any convex polytope has at least one stable facet. Any convex polygon has at least two stable sides.
On the other hand, there are polytopes with exactly one stable facet (unistable polytopes) in $\mathbb{R}^n$, $n \geq 3$.
For instance, there are unistable polyhedra in $\mathbb{R}^3$ with $14$ faces and more, see \cite{Resh14}, where much additional information on stable faces and
unistable polytopes can be found.

\begin{remark}
Any homogeneous (proper) polytope $P$ in $\mathbb{R}^n$ has at least $2$ stable faces. It is clear that $P$ has one stable face $F$. Since $P$ is a proper
polytope in $\mathbb{R}^n$, then there is a vertex $V$ of $P$ that is not in the hyperplane, containing $F$. Since $P$ is homogeneous, it has some isometry $\psi$
that moves a vertex of $F$ to $V$, then (due to homogeneity) the face $\psi(F)$ is also stable and $\psi(F)\neq F$.
\end{remark}

\medskip

\begin{figure}[t]
\begin{minipage}[h]{0.45\textwidth}
\center{\includegraphics[width=0.92\textwidth, trim=0in 0in 0mm 0mm, clip]{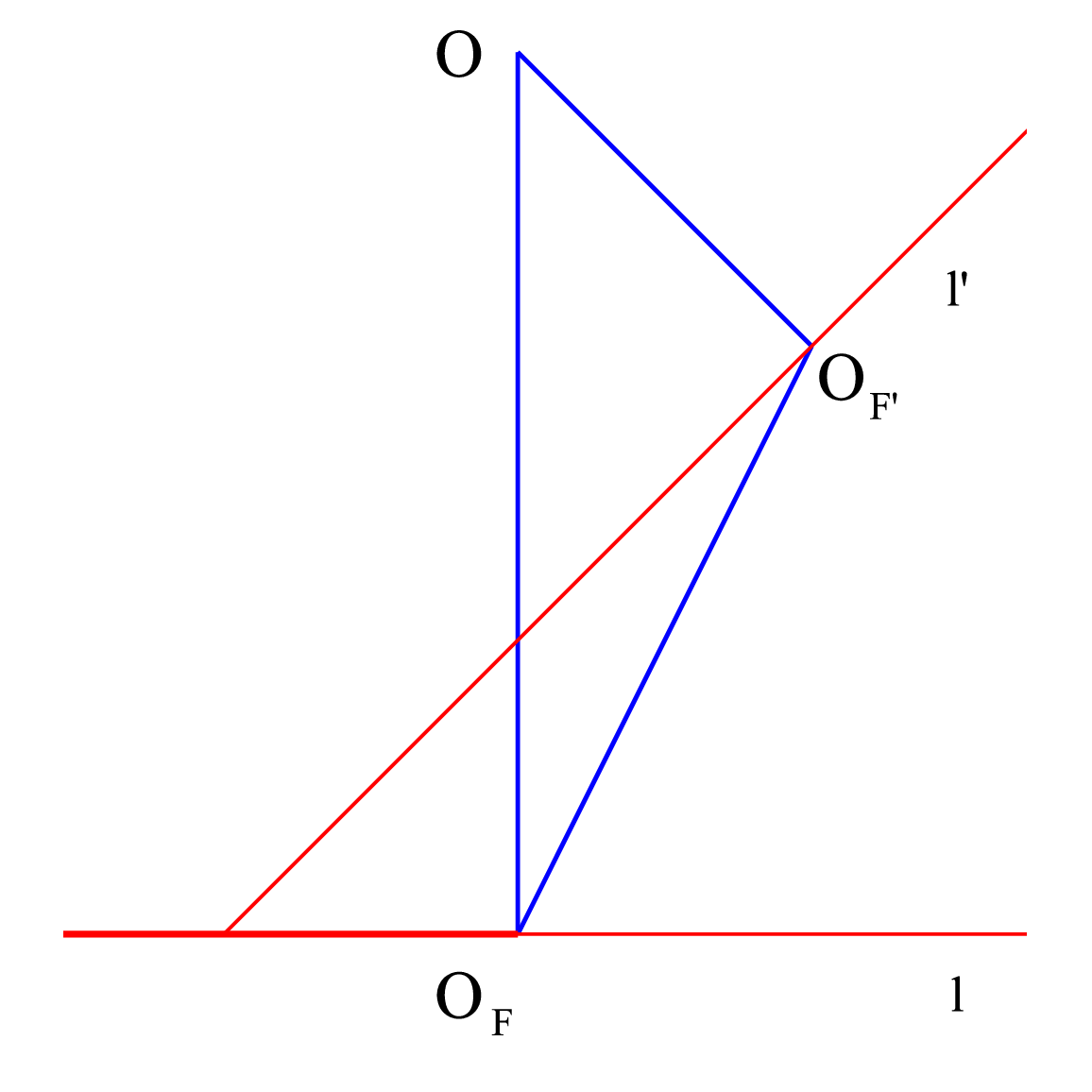} \\ a)}
\end{minipage}
\begin{minipage}[h]{0.45\textwidth}
\center{\includegraphics[width=0.9\textwidth, trim=0in 0in 0mm 0mm, clip]{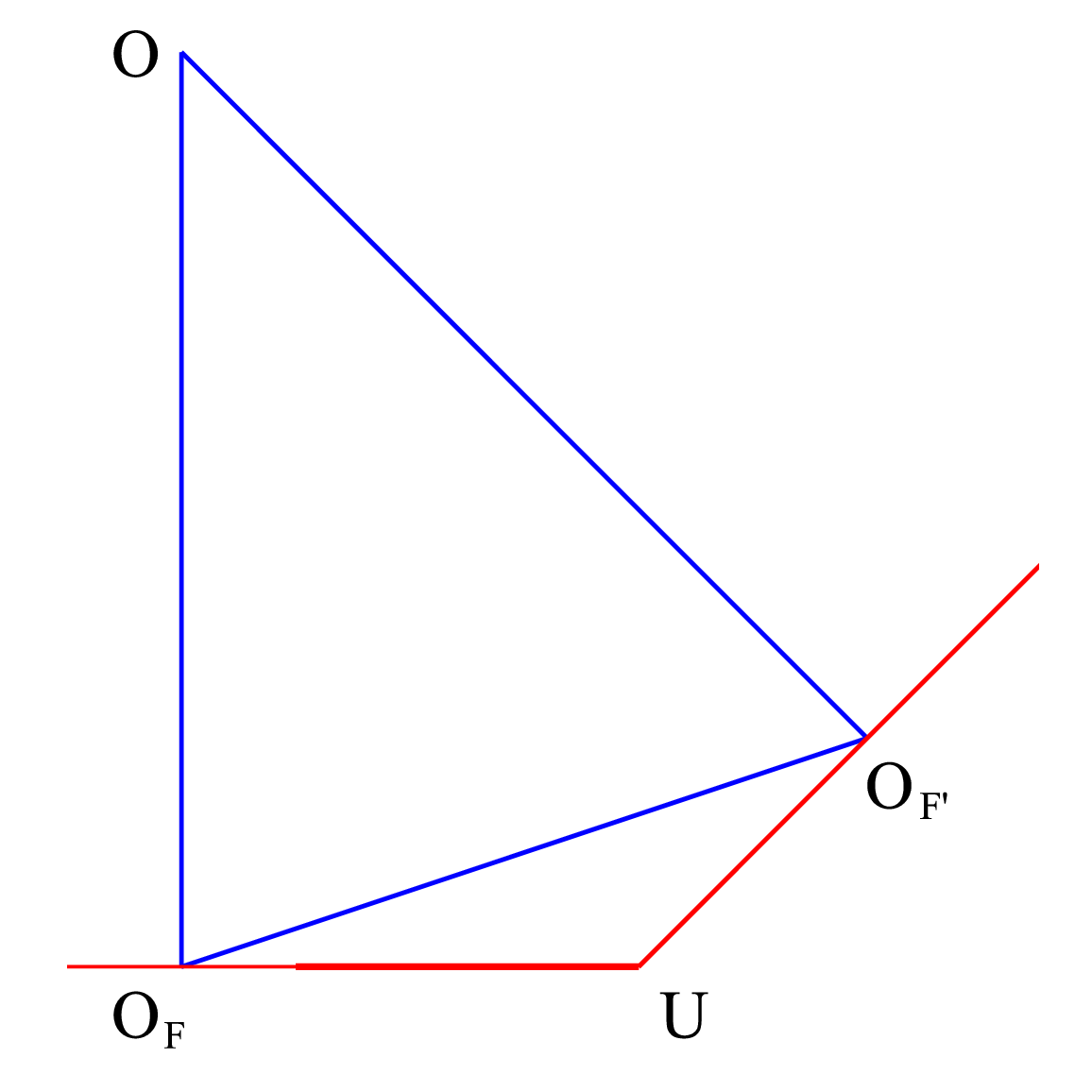} \\ b)}
\end{minipage}
\caption{{{\normalsize{On stable faces of homogeneous polytopes.}
}}}
\label{Fig5}
\end{figure}

Let us recall that any homogeneous polytope in $\mathbb{R}^n$ is inscribed in a suitable sphere $S=S(O,r)$ with the center $O$ and radius $r$
(this means that every vertex of $P$ is in the sphere $S$).
Moreover, any facet $F$ is inscribed in the sphere $S\cap L_F$ in the hyperplane $L_F$ through $F$ with the center $O_F$ such that $\overrightarrow{OO_F}$
is orthogonal to $L_F$
and radius $\sqrt{r^2-d^2(O,O_F)}$.

\begin{prop}\label{pr.grav.2}
Let $P$ be a homogeneous polytope in $\mathbb{R}^n$ with the center $O$, $F$ be a facet of $P$ with the center $O_F$ in the hyperplane through $F$.
Then $F$ is stable if and only if for any facet $F'$ of $P$ with the center $O_{F'}$ in the hyperplane through $F'$,
the inequality $d^2(O,O_F) < d^2(O,O_{F'})+d^2(O_F,O_{F'})$ holds.
\end{prop}

\begin{proof}
Let us suppose that there is a facet $F'$ of $P$ with the center $O_{F'}$ in the hyperplane through $F'$,
such that $d^2(O,O_F) \geq d^2(O,O_{F'})+d^2(O_F,O_{F'})$.
The above inequality means that $\angle OO_{F'}O_F \geq \pi/2$.
Denote by $L$ and $L'$ the hyperplanes through $F$ and $F'$ respectively (in particular, $L$ and $L'$ are supporting hyperplanes for $P$).
It should be noted that $L'$ is not parallel to $L$ (otherwise, $O$ is situated on the segment $[O_F, O_{F'}]$ and $d^2(O,O_F) < d^2(O,O_{F'})+d^2(O_F,O_{F'})$).
The points $O$, $O_F$, $O_{F'}$ are situated in a $2$-dimensional plane $OO_FO_{F'}$.
Denote by $l$ and $l'$ the intersections of this plane with $L$ and $L'$ respectively.
It is clear that the straight line $OO_F$ is orthogonal to $l$ and the straight line $OO_{F'}$ is orthogonal to $l'$.
Obviously, $O$ and $F \cap l$ are situated in the same closed half-plane in the plane $OO_FO_{F'}$ with respect to $l'$
(since $l' \subset L'$ and $L'$ is a supporting hyperplane for $P$).
Moreover, the intersection of the interior of $F$ with the plane $OO_FO_{F'}$ is in the same open half-plane  relative to $l'$ that contains $O$.
Since $\angle OO_{F'}O_F \geq \pi/2$, then
$O_F$ is not situated in the same open half-plane, see Fig.~\ref{Fig5}~a).
This means that $O_F$ is not in the interior of $F$, hence, the facet $F$ is not stable.

Now, suppose that $F$ is not stable. Then there is a facet $F^{\ast}$ of $F$ ($\dim \, F^{\ast} =n-2$)
such that the interior of $F$ is situated in an open half-space $HS$,
that is determined by $L^{\ast}$ in the hyperplane through $F$,  where $L^{\ast}$ is a $(n-2)$-dimensional plane through $F^{\ast}$, while $O_F \not\in HS$.

There is a facet $F'\neq F$ of $P$ (with the center $O_{F'}$ in the hyperplane through $F'$) such that $F^{\ast}$ is a facet of $F'$.
Let $\widetilde{P}$ be an orthogonal projection of $P$ to $L$, where $L$ is the $2$-dimensional plane $OO_FO_{F'}$. It is clear that the image of $F^{\ast}$
under this projection is a point $U \in L$, that is in the intersection of the corresponding projections of $F$ and $F'$, see Fig.~\ref{Fig5}~b).
We see that $\angle U O_{F'} O =\pi/2$ and, hence, $\angle O_F O_{F'} O \geq \pi/2$ (or, equivalently,
$d^2(O,O_F) \geq d^2(O,O_{F'})+d^2(O_F,O_{F'})$) by the choice of
$F^{\ast}$ and $F'$.
\end{proof}

\begin{remark}\label{re.hom.simp}
It is clear that any homogeneous simplex $P$ in $\mathbb{R}^3$ has acute triangular faces that are isometric each to other.
Indeed, otherwise $P$ has no one stable face (due to the homogeneity all faces are equivalent).
\end{remark}

\begin{remark}\label{re.right.ant}
Let $P$ be a right antiprism in $\mathbb{R}^3$ with $n$-gonal bases, $n \geq 2$.
It is easy to see that lateral faces are acute isosceles triangles for $n=2$ (otherwise $P$, that is a homogeneous simplex, has no stable face).
On the other hand, lateral faces can be acute, right, and obtuse (isosceles) triangles for different heights of $P$ for $n\geq 3$.
If lateral faces are obtuse (isosceles) triangles, then $P$ has exactly $2$ stable faces.
\end{remark}

\section{Some properties of $2$-point homogeneous polytopes in $\mathbb{R}^n$}\label{sec.4}


Let us consider any homogeneous metric space $(M,d)$ with the isometry group $\Isom(M)$.
For any $x\in M$, the group $\I(x)=\{\psi \in \Isom(M) \,|\, \psi(x)=x \}$ is called {\it the isotropy subgroup at the point $x$}.
If $\eta \in \Isom(M)$ is such that $y=\eta (x)$, then, obviously, $\I(y)= \eta \circ \I(x) \circ \eta^{-1}$. Hence, the isotropy subgroups for every two points of
a given homogeneous finite metric $M$ space are conjugate each to other in the isometry group.
In particular, the cardinality of $M$ is the quotient of the cardinality of $\Isom(M)$ by the cardinality of $\I(x)$ (for any $x\in M$).

Let $M$ be a vertex set of a homogeneous polytope $P$ in Euclidean space $\mathbb{R}^n$.
Let us consider all possible distances between distinct points of $M$: $0<d_1<d_2<\cdots < d_s$.
In the case of a regular tetrahedron, for example, we obviously have the equality $s=1$.
\smallskip

For a given $x\in M$ and $r >0$ we consider $S(x,r)=\{y\in M\,|\,d(x,y)=r\}$ (the sphere in $M$ with the center $x$ and radius $r$).
It is clear that  $S(x,r) \neq \emptyset$ if and only if $r \in \{0, d_1, d_2, \dots, d_s\}$.

\smallskip
\begin{prop}[\cite{BerNik22}]\label{pr:two-point homogeneous}
A homogeneous metric space $(M,d)$ is $2$-point homogeneous if and only if for any point $x\in M$
the following property holds: for every $r >0$ and every $y,z \in S(x,r)$, there is
an $f \in \I(x)$ such that $f(y)=z$. In other words, a homogeneous metric space $(M,d)$ is $2$-point homogeneous if and only if
for any point $x\in M$, the isotropy subgroup $\I(x)$ acts transitively on every sphere $S(x,r)$, $r>0$.
\end{prop}

\begin{remark}
If $(M,d)$ is a homogeneous subset in Euclidean space $\mathbb{R}^n$, then every sphere $S(x,r)$ is in some $(n-1)$-dimensional Euclidean subspace,
since all points of this sphere are also at one and the same distance from the barycenter of $M \subset \mathbb{R}^n$
(recall that $\|x-a\|=C$ and $\|x-b\|=D$ for $x, a, b \in \mathbb{R}^n$
implies $2(x,a-b)=\|a\|^2-\|b\|^2-C^2+D^2$).
\end{remark}

If $(M,d)$ is the vertex set of a homogeneous polytope $P \subset \mathbb{R}^n$ with center $O$, then for every vertex $A$,
all points of the straight line $OA$ are fixed under any element of the isotropy group $\I(A)$.
\smallskip

\begin{prop}\label{prop.gen0}
Let $P$ be a $2$-point homogeneous polytope in $\mathbb{R}^n$ with the vertex set~$M$ and center $O$. Let us consider some points $A,B \in M$,
such that the midpoint $D$ of the segment $[A,B]$ is distinct from $O$. Let $L$ be a $2$-dimensional plane $L$ through the points $A$, $B$, and $O$,
and $l$ be a straight line through $O$ and  $D$. Then $P\cap L$ is symmetric in $L$ with respect to the straight line $l$.
Moreover, if a point $C\in l$ is in the interior of a facet $F$ of $P$,
then $C$ is the center of the segment $F \cap L$.
\end{prop}

\begin{proof}
Since $P$ is $2$-point homogeneous, there is an isometry $\sigma$ of $P$, such that $\sigma(A)=B$ and $\sigma(B)=A$.
Obviously, $\sigma$ fixes all points on the straight line $l$ and the $2$-dimensional plane $L$ through the points $A$, $B$, and $O$.
Therefore, $P\cap L$ is invariant with respect to $\sigma$. Finally, the restriction of $\sigma$ to $L$ is the symmetry with respect to the straight line $l$.
This proves the first assertion. The second assertion follows from the fact, that $F \cap L$ is a line segment with the midpoint $C$.
\end{proof}

\begin{figure}[t]
\vspace{5mm}
\begin{center}
\begin{minipage}[h]{0.3\textwidth}
\center{\includegraphics[width=0.8\textwidth, trim=0in 0in 0mm 0mm, clip]{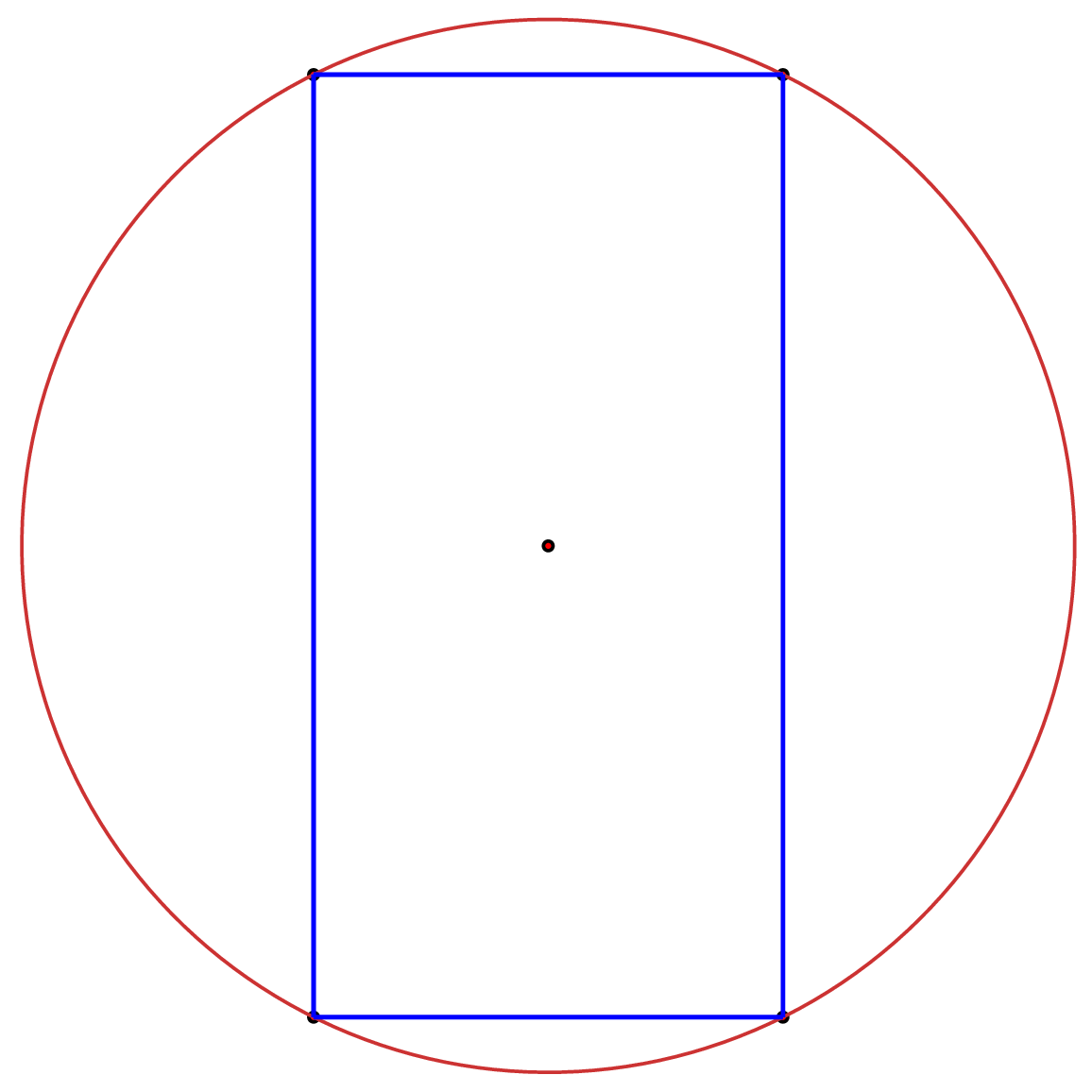} \\ a)}
\end{minipage}
\begin{minipage}[h]{0.3\textwidth}
\center{\includegraphics[width=0.8\textwidth, trim=0in 0in 0mm 0mm, clip]{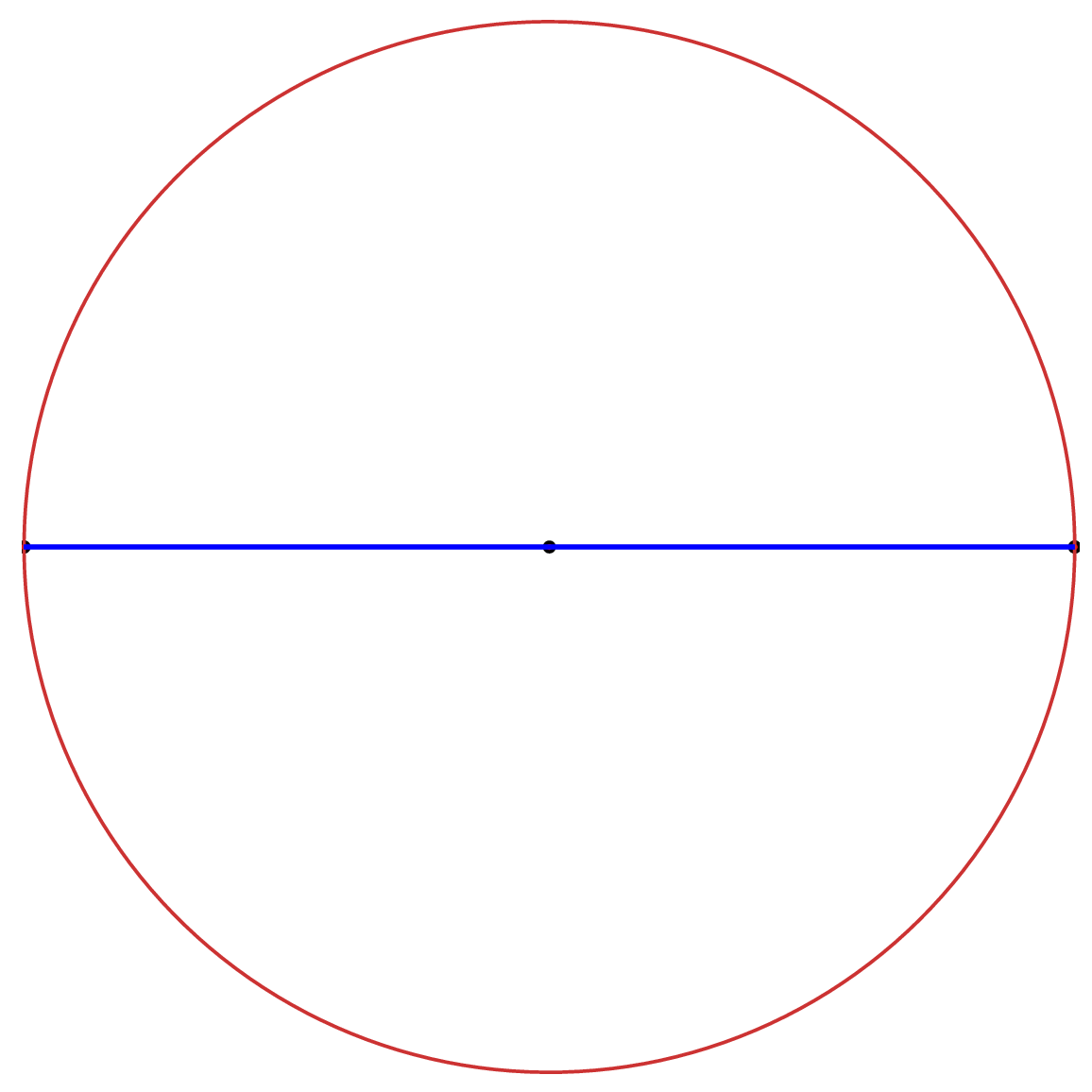} \\ b)}
\end{minipage}
\begin{minipage}[h]{0.3\textwidth}
\center{\includegraphics[width=0.8\textwidth, trim=0in 0in 0mm 0mm, clip]{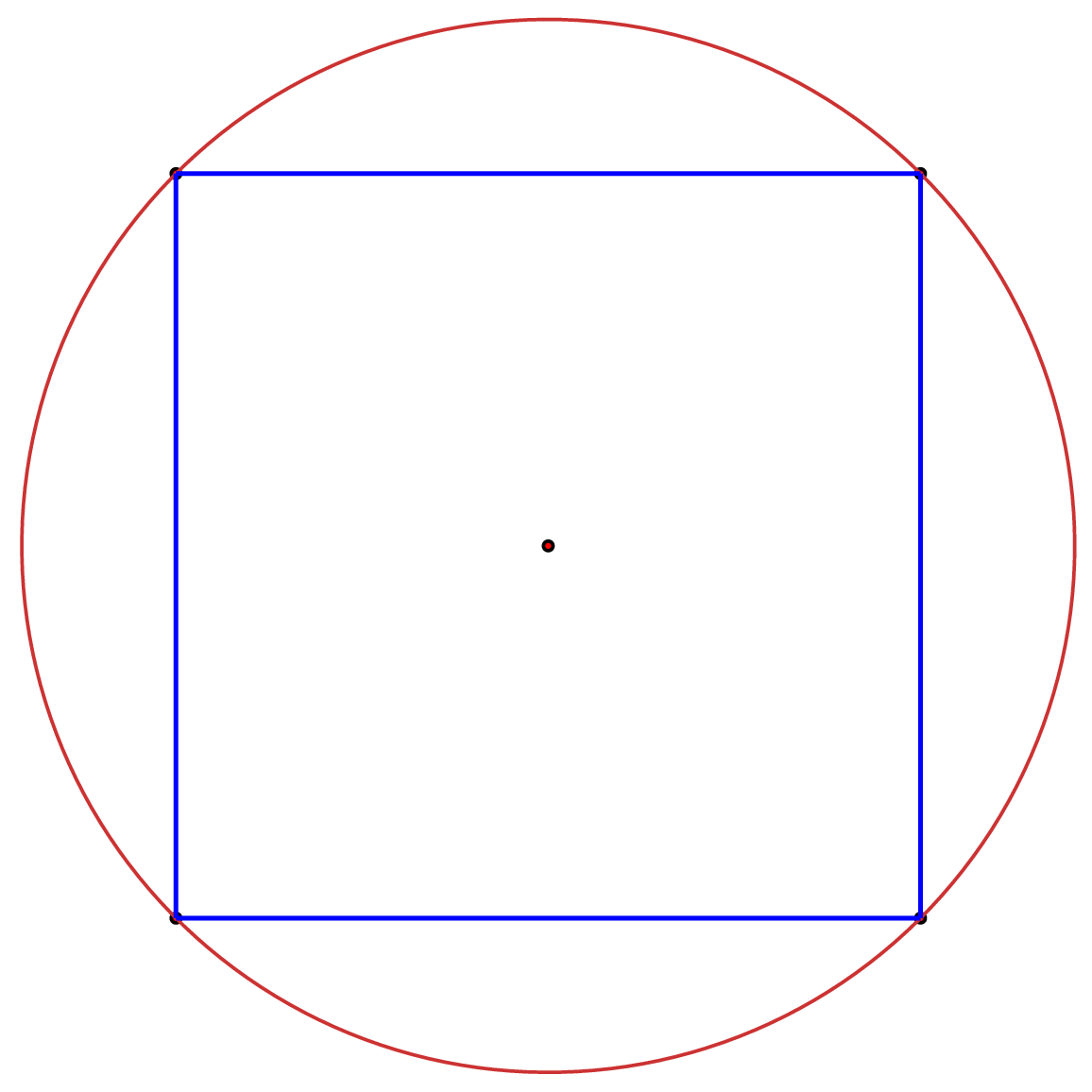} \\ c)}
\end{minipage}
\vfill
\begin{minipage}[h]{0.3\textwidth}
\center{\includegraphics[width=0.8\textwidth, trim=0in 0in 0mm 0mm, clip]{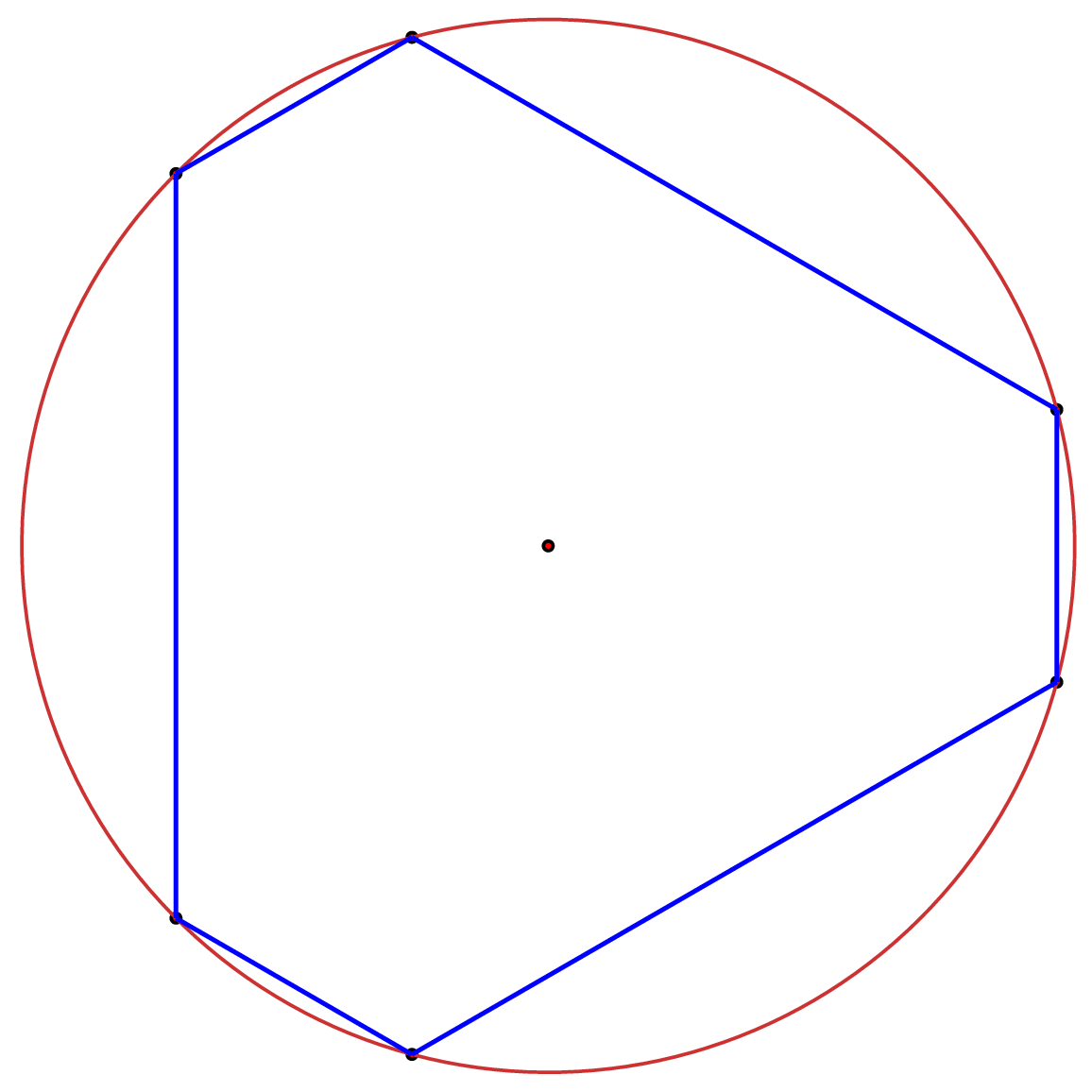} \\ d)}
\end{minipage}
\begin{minipage}[h]{0.3\textwidth}
\center{\includegraphics[width=0.8\textwidth, trim=0in 0in 0mm 0mm, clip]{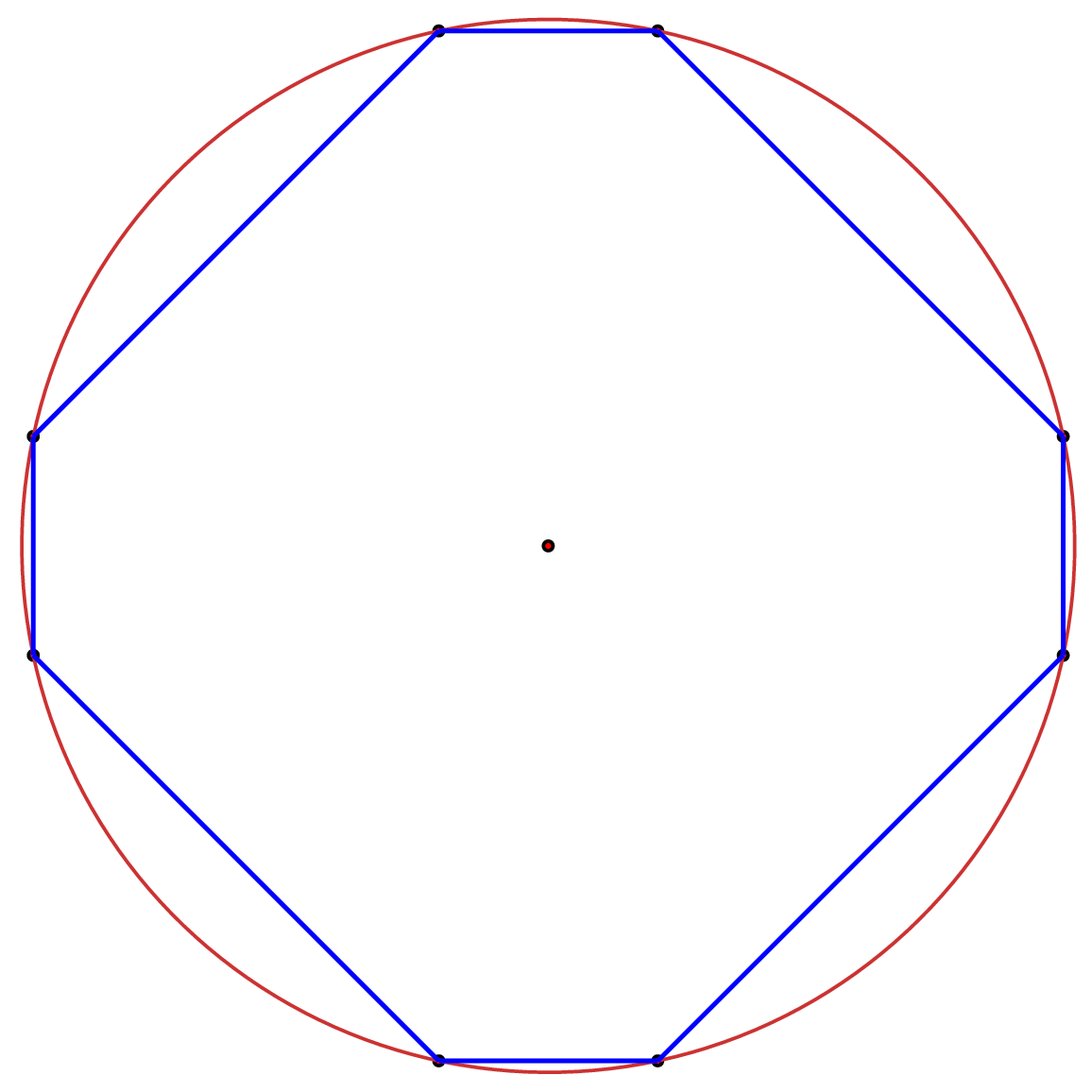} \\ e)}
\end{minipage}
\begin{minipage}[h]{0.3\textwidth}
\center{\includegraphics[width=0.8\textwidth, trim=0in 0in 0mm 0mm, clip]{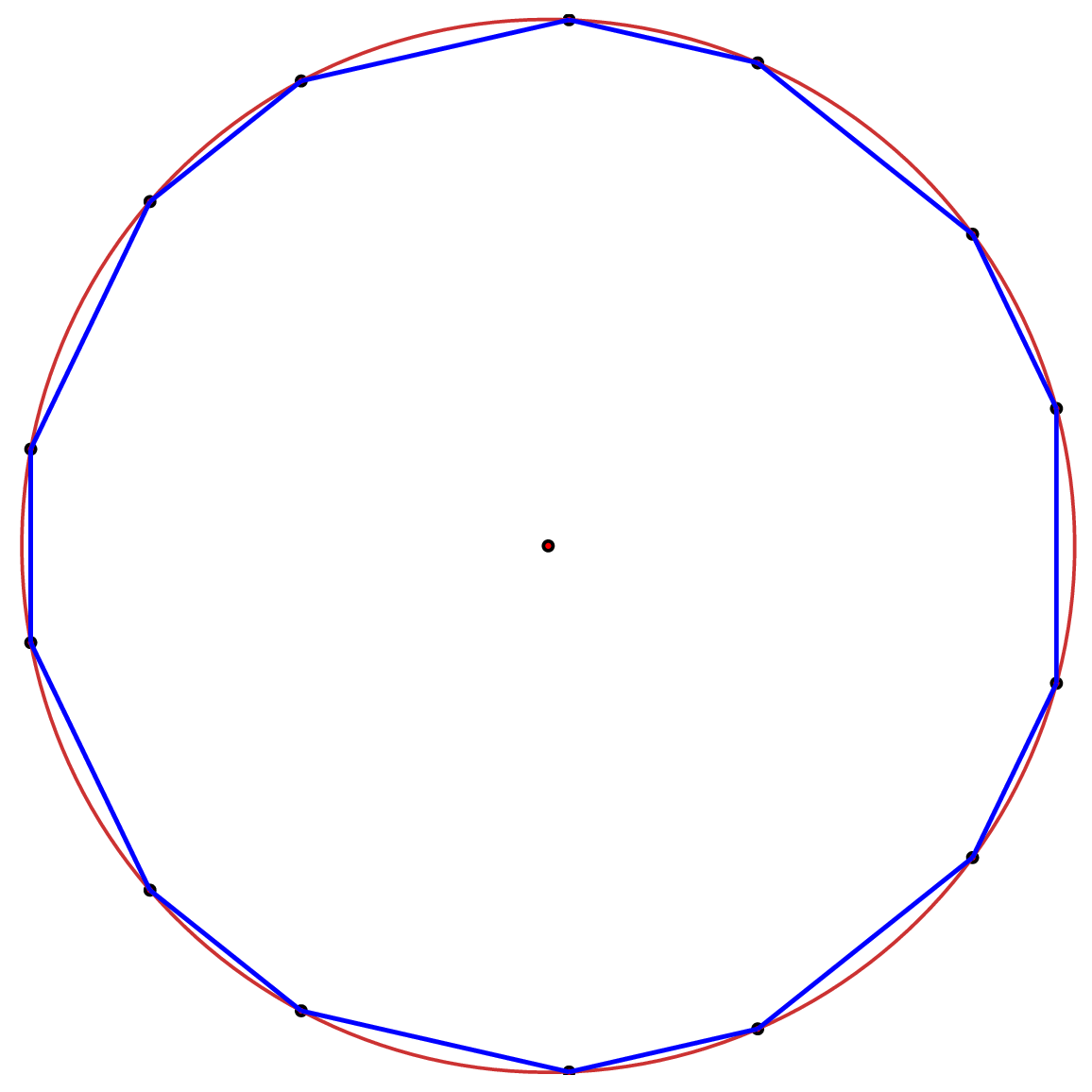} \\ f)}
\end{minipage}
\caption{{{\normalsize{Homogeneous polygons (orbits of $D_n$) for:}}
\normalsize{a) --- c) $n=2$;  d)  $n=3$; e)  $n=4$; f)  $n=7$.}}}
\label{Fig6}
\end{center}
\end{figure}

\section{Two-dimensional case}\label{sec.5}

Here we consider some properties of homogeneous and $2$-point homogeneous polygons.

There are only two families of finite subgroups of $O(2)$: cyclic groups $C_1,C_2,\dots, C_k, \dots$ and
dihedral groups $D_1, D_2, \dots, D_k, \dots$. The group $C_k$ is generated by the rotation by angle $2\pi/k$,
the group $D_k$ is generated by the same rotation and the reflection $(x_1,x_2) \mapsto (x_1,-x_2)$.
The groups $C_k$ and $D_k$ have the orders $k$ and $2k$ respectively.

\begin{lemma}\label{le.homepolygon.1}
Let $P$ be a non-degenerate homogeneous polygon in $\mathbb{R}^2$. Then either $P$ is a regular $n$-gon for some $n \geq 3$,
or $P$ is a convex hull of the union of two isometric regular $n$-gons  with a common center for some $n\geq 2$.
\end{lemma}

\begin{proof}
A non-trivial orbit of the group $C_n$ is a vertex set of a regular $n$-gon. In this case we have a non-degenerate polygon only when $n \geq 3$.
The orbit of the group $D_n$ is a union of two orbits of $C_n \subset D_n$, that could be moved each to the other by some element of $D_n \setminus C_n$.
\end{proof}

\begin{remark}\label{rem.semreg.pol}
Every regular $n$-gon has the isometry group $D_n$.
A non-regular polygon, whose vertex set is the orbit of some dihedral group $D_n$, is equiangular
(all its angles are congruent), see Fig.~\ref{Fig6}. There is always a circle passing through all vertices of
such a polygon and any two sets of alternate sides have one and the same length.
In fact, there are two classes of vertices (each class contains $n$ alternate vertices)
such that the vertices in a fixed class could be moved each to other by a suitable isometry of the polygon.
\end{remark}

Now we give a more detailed description of non-regular homogeneous polygons.

\begin{example}\label{ex.2dim.1} Let us fix a natural number $n \geq 2$ and two real numbers $\alpha$ and $\beta$ such that $0<\alpha <\beta$ and $\alpha+\beta =2\pi/n$.
Now we consider the standard unit circle $S=\{(x,y)\in \mathbb{R}^2\,|\, x^2+y^2=1\}$ with the origin $O$ and $2n$ rays $OA_i$, $1\leq i \leq 2n$, where $A_i \in S$,
$OA_1$ is  the first coordinate ray, and every ray $OA_{i+1}$ is obtained from the ray $OA_i$ by rotating it by the angle $\alpha$ ($\beta$) counterclockwise
if $i$ is odd (respectively, even), assuming that $A_{2n+1}=A_1$. Denote the convex hull of the points $A_i$, $i=1,\dots,2n$ as $R_n(\alpha,\beta)$.
It is clear that the points $A_i$, $i=1,\dots,2n$, form an orbit of the group $D_n$. In particular, any $R_2(\alpha,\beta)$ is a rectangle inscribed in $S$.

For any side $[A_i,A_{i+1}]$ the symmetry with respect to the line through the center of
$[A_i,A_{i+1}]$ and $O$ is an isometry of $R_n(\alpha,\beta)$. Hence, for a given side,  we have an isometry that interchanges its endpoints.
Moreover the rotation group, generated by the rotation by the angle $\alpha+\beta$ counterclockwise, is the cyclic group $\mathbb{Z}_n$,
that acts transitively on each set of sides of $R_n(\alpha,\beta)$ with equal lengths. Therefore, for every two sides of $R_n(\alpha,\beta)$ with equal lengths,
there is an isometry of
$R_n(\alpha,\beta)$ that moves one side to another one (with both orientations). For $n=2$, it implies that any rectangle $R_2(\alpha,\beta)$
is $2$-point homogeneous (it is easy to check directly). On the other hand, if $n\geq 3$, then there is no isometry $\eta$ of
$R_n(\alpha,\beta)$  such that $\eta(A_2)=A_2$ and $\eta (A_0)=A_4\neq A_0$. Indeed, for any such $\eta$ we get that $\eta (A_1)=A_3$.
On the other hand $d(A_2,A_1)\neq d(A_2,A_3)=d(\eta (A_2),\eta (A_1))$ since $\alpha <\beta$. Therefore, $R_n(\alpha,\beta)$ is not $2$-point homogeneous
for all $n\geq 3$.
\end{example}

\begin{theorem}\label{th.2dim.1}
If $P$ is a homogeneous polygon in $\mathbb{R}^2$, then $P$ is either regular, or it is similar to some polygon $R_n(\alpha,\beta)$ as in Example \ref{ex.2dim.1}
for suitable $\alpha$, $\beta$, and $n\geq 2$. As a corollary, $P$ is $m$-point homogeneous for any $m\in\mathbb{N}$ if $P$ is a regular polygon
or a rectangle {\rm(}similar to $R_2(\alpha,\beta)$ for some $0<\alpha <\beta${\rm)}. If $P$ is not a regular polygon or a rectangle, then $P$ is not $2$-point homogeneous.
\end{theorem}

\begin{proof}
Since $P$ is homogeneous, it is inscribed in a circle. Without loss of generality, we may suppose that this circle is the standard unit circle $S$.
If for any vertex of $P$, two sides of $P$, adjacent to it, have the same lengths, then $P$ is a regular polygon.

Let us suppose that for some vertex, the lengths of two sides, adjacent to it, are different, say $\alpha$ and $\beta$, $0<a <b$.
Since $P$ is homogeneous, then the same is true about any vertex of $P$.
Consequently, the lengths of the sides of the polygon are equal to the numbers $a$ and $b$ alternately,
i.~e. the adjacent sides for a given side have lengths different from the length of this side.
Now, it is clear that $P$ is isometric to some polygon $R_n(\alpha,\beta)$, where $a=2\sin(\alpha/2)$ and $b=2\sin(\beta/2)$,
as in Example \ref{ex.2dim.1} for some natural $n$.

As it was explained above, all regular polygons, as well as, all rectangles $R_2(\alpha,\beta)$ are $2$-point homogeneous and even $m$-point homogeneous
for any $m\in\mathbb{N}$ by Corollary~\ref{co:if-homog}.
On the other hand, all polygons $R_n(\alpha,\beta)$ are not $2$-point homogeneous for any $n\geq 3$, see Example \ref{ex.2dim.1}.
\end{proof}

\smallskip

\section{On right prisms and antiprisms}\label{sec.6}

Let us recall some definitions related to right prisms and antiprisms in $\mathbb{R}^3$.

A {\it right prism} is a polyhedron whose two faces (called bases) are congruent (equal) regular polygons,
lying in parallel planes, while the other faces (called lateral ones) are rectangles (perpendicular to
the bases).
A {\it homogeneous prism} is a right prism with homogeneous bases (see Lemma \ref{le.homepolygon.1} and Theorem \ref{th.2dim.1} that describe all homogeneous polygons).
A {\it regular prism} is a right prism with regular bases.
If in addition lateral faces are squares then the prism under consideration is also a semiregular convex polyhedron.

A {\it right antiprism} is a polyhedron, whose two parallel faces (bases) are equal regular $n$-gons twisted by the angle $\pi/n$ relative to each other,
while the other $2n$ (lateral) faces are isosceles triangles. It is also called {\it a right $n$-gonal antiprism}.
If lateral faces are equilateral triangles then the antiprism under consideration is called {\it uniform}  and is a semiregular convex polyhedron.

Now we consider some general construction related to prisms over homogeneous polytopes.

\begin{prop}\label{pr.constr.1}
Let  $P$ be a homogeneous polytope in $\mathbb{R}^n$ and $c >0$, then $\widetilde{P}=P\times [0,c]$ is a homogeneous polytope in $\mathbb{R}^{n+1}$.
\end{prop}

\begin{proof}
Without loss of generality we may assume that any vertex of $\widetilde{P}$ has coordinates $(a_1,a_2,\dots, a_n, a_{n+1})$, where
$(a_1,a_2,\dots, a_n)$ is a vertex of $P$ and $a_{n+1} \in \{0,c\}$. Hence the vertices of $\widetilde{P}$, whose last coordinate is $0$ (or $c$),
determines a polytope isometric to $P$. We call it the {\it bottom} base (respectively, the {\it top} base).
We see that the isometry $\bigl(x_1,x_2,\dots, x_n, x_{n+1}\bigr) \mapsto \bigl(x_1,x_2,\dots, x_n, c-x_{n+1}\bigr)$ interchanges the bottom and the top bases.
On the other hand, every isometry $\tau$ of $P$ generates an isometry $\bigl(x_1,x_2,\dots, x_n, x_{n+1}\bigr) \mapsto \bigl(\tau(x_1,x_2,\dots, x_n), x_{n+1}\bigr)$ of
$\widetilde{P}$. Now, it is easy to see that the isometry group of $\widetilde{P}$ acts transitively on the vertex set of $\widetilde{P}$.
\end{proof}

\medskip
Now we are going to introduce some natural generalization of (right) homogeneous prisms.

\begin{definition}\label{de.gen.prism.1}
A homogeneous polytope $P \subset \mathbb{R}^n$ is called a {\it generalized homogeneous prism} if it has two faces $F_1$ and $F_2$ with
$F_1\bigcap F_2=\emptyset$, any vertex of $P$ is a vertex either of $F_1$ or of $F_2$, and there exists an isometry $\eta$ of $P$
such that $\eta(F_1)=F_2$ and $\eta(F_2)=F_1$.
The faces $F_1$ and $F_2$ are called the bases of $P$.
\end{definition}

It is easy to see that all right prisms over homogeneous bases and all right antiprisms are generalized homogeneous prisms.
Note that any right antiprism over a segment in $\mathbb{R}^3$ (in particular, a regular tetrahedron) has two one-dimensional bases.

\begin{definition}\label{de.gen.prism.2}
A generalized homogeneous prism $P \subset \mathbb{R}^n$ is called  {\it  rigidly two-layered} if
for every triple of vertices
$A,B,C \in {P}$, the equality $d(A,B)=d(A,C)$ implies that $B$ and $C$ are situated in one and the same base of ${P}$
(in other words, the set of distances between vertices of $P$ from a fixed base has empty intersection with the set
of distances between vertices of $P$ from distinct bases).
\end{definition}

\begin{definition}\label{de.gen.prism.3}
A generalized homogeneous prism $P \subset \mathbb{R}^n$ is called  {\it strongly prismatic} if
for any isometry $\eta$ of $P$, either $\eta$ preserves both bases of $P$, or $\eta$ interchanges two bases of $P$.
\end{definition}

\begin{prop}\label{pr.prism.3.0}
Let us suppose that a generalized homogeneous prism $P \subset \mathbb{R}^n$ is strongly prismatic and $2$-point homogeneous.
Then it is rigidly two-layered.
\end{prop}

\begin{proof}
Since $P$ is  strongly prismatic, then for any isometry
$\eta$ of $P$, either $\eta$ preserves both bases of $P$, or $\eta$ interchanges two bases of $P$.
Suppose that $P$ is not rigidly two-layered. Then there are three vertices $A,B,C$ of
$P$ such that $B$ and $C$ are situated in distinct bases of ${P}$ and $d(A,B)=d(A,C)$.
Since $P$ is $2$-point homogeneous, there is an isometry $\theta$ of $P$, such that $\theta(A)=A$ and $\theta(B)=C$. Since $B$ and $C$ are in distinct bases of $P$, then
$\eta$ interchanged these bases. But $\theta(A)=A$ means that $\theta$ preserves both bases. This contradiction proves the proposition.
\end{proof}
\smallskip

\begin{lemma}\label{le.constr.2}
For any homogeneous polytope $P$, the property of the prism  $\widetilde{P}=P\times [0,c]$
to be rigidly two-layered is equivalent to the following one: $c^2 \not\in D(P)$,
where $D(P)$ is the set of all numbers of the type $|s^2-t^2|$, where
$t$ and $s$ are distances between some pairs of vertices of $P$.
\end{lemma}

\begin{proof}
Let us suppose that $\widetilde{P}$ is not rigidly two-layered.
Then there are $A,B,C \in \widetilde{P}$ such that $d(A,B)=d(A,C)$, $B \in P$ and $C \not\in P$. Since $\widetilde{P}$ is homogeneous, we may suppose that $A\in P$.
Let $C'$ be a vertex of ${P}$ such that $[C',C]$ is orthogonal to~$P$. If $t=d(A,B)$ and $s=d(A,C')$, then
$c^2=d^2(C',C)=d^2(A,C)-d^2(A,C')=d^2(A,B)-d^2(A,C')=t^2-s^2$, hence, $c^2 \in D(P)$.

Now, if $c^2 \in D(P)$, then for some $A,B,C,D \in P$ we have $c^2=t^2-s^2$, where $d(A,B)=t$ and $d(C,D)=s$. Since $P$ is homogeneous, then we may assume that
$C=A$. Let $D''$ be a vertex of $\widetilde{P}$ such that $[D'',D]$ is orthogonal to $P$ and $D''$ is not in the base $P$. It is clear that
$d^2(A,D'')=d^2(A,D)+d^2(D,D'')=s^2+c^2=t^2=d(A,B)$. Since $B\in P$ and $D'' \not\in P$, then the prism is not rigidly two-layered.
\end{proof}
\smallskip

The following result gives us examples of $m$-point homogeneous polytopes in every dimension.

\begin{theorem}\label{th.constr.2}
Let $P$ be an $m$-point homogeneous polytope in $\mathbb{R}^n$, $m \geq 1$, and let the prism $\widetilde{P}=P\times [0,c]$
be rigidly two-layered. Then  $\widetilde{P}$ is an $m$-point homogeneous polytope in $\mathbb{R}^{n+1}$.
\end{theorem}

\begin{proof} For $m=1$, the theorem follows from Proposition \ref{pr.constr.1}. Now, we suppose that $m \geq 2$.
Without loss of generality we may assume that
any vertex of $\widetilde{P}$ has coordinates $(a_1,a_2,\dots, a_n, a_{n+1})$, where
$(a_1,a_2,\dots, a_n)$ is a vertex of $P$ and $a_{n+1} \in \{0,c\}$.
For any vertex $A$ of $\widetilde{P}\subset \mathbb{R}^{n+1}$, we denote by $A'$ the vertex of $P\subset \mathbb{R}^{n}$,
whose coordinates are exactly the $n$ first coordinates of $A$.

Let us consider any $m$-tuples $(A_1, A_2, \dots, A_m)$ and $(B_1, B_2, \dots, B_m)$ from the vertex set of $\widetilde{P}$,
such that $ d(A_i,A_j)=d(B_i,B_j)$, $i,j=1,\dots,m$, where $d$ means the standard Euclidean distance.
It suffices to prove that there is an isometry $\tau$ of $\widetilde{P}$ that moves the first $m$-tuple to the second one.
For this, we consider the corresponding $m$-tuples $(A'_1, A'_2, \dots, A'_m)$ and $(B'_1, B'_2, \dots, B'_m)$ from the vertex set of~$P$.

Let us denote by $\alpha_i$ and $\beta_i$ the $(n+1)$-th coordinates of the points $A_i$ and $B_i$ respectively.
By our assumption, $d(A_i,A_j)=d(B_i,B_j)$ for all $i,j=1,\dots,m$. This could be rewritten as follows:
\begin{equation}\label{eq.prod.1}
d^2({A'}_i,{A'}_j)+(\alpha_i-\alpha_j)^2=
d^2({B'}_i,{B'}_j)+(\beta_i-\beta_j)^2.
\end{equation}
If $(\alpha_i-\alpha_j)^2\neq (\beta_i-\beta_j)^2$, then
$\bigl\{|\alpha_i-\alpha_j|,|\beta_i-\beta_j|\bigr\}=\left\{0,c \right\}$
and, consequently, \eqref{eq.prod.1} contradicts the condition that $c^2 \not\in D(P)$ (see Lemma \ref{le.constr.2}). Since
the prism $\widetilde{P}$ is rigidly two-layered, then
$(\alpha_i-\alpha_j)^2= (\beta_i-\beta_j)^2$
and, therefore, $d(A'_i,A'_j)=d(B'_i,B'_j)$.

Since the polytope $P$ is $m$-point homogeneous, there is an isometry $\tau$ of $P$ that moves $(A'_1, A'_2, \dots, A'_m)$ to $(B'_1, B'_2, \dots, B'_m)$.

Now, we consider $I_1=\{1\leq i \leq m\,|\, \alpha_i=0\}$ and $I_2=\{1\leq i \leq m\,|\, \alpha_i=c\}$. Obviously, $I_1 \bigcup I_2=\{1,2,\dots,m\}$.

Since $(\alpha_i-\alpha_j)^2= (\beta_i-\beta_j)^2$ for all $i,j=1,\dots,m$, then $\beta_i=\beta_j$ if and only if either $i,j \in I_1$, or  $i,j \in I_2$.
Therefore, we have two possibilities: 1) $\beta_i= \alpha_i$, 2) $\beta_i= c-\alpha_i$ for all indices $i$
(in this case the cardinalities of $I_1$ and $I_2$ coincide and $m$ is an even number).

In each of these two cases we take the isometry $\tau$ of $\mathbb{R}^{n+1}$ to be of the following form, respectively:
$$
\bigl(x_1,x_2,\dots, x_n, x_{n+1}\bigr) \mapsto \bigl(\tau(x_1,x_2,\dots, x_n), x_{n+1}\bigr),
$$
$$
\bigl(x_1,x_2,\dots, x_n, x_{n+1}\bigr) \mapsto \bigl(\tau(x_1,x_2,\dots, x_n), c-x_{n+1}\bigr).
$$
It is clear that $\tau$ is an isometry of $\widetilde{P}$ and it moves $(A_1, A_2, \dots, A_m)$ to $(B_1, B_2, \dots, B_m)$.
The theorem is proved.
\end{proof}

\begin{example}\label{ex.tpoin.1}
The triangle $P$ with the vertices $A=(0,0)$, $B=(2,0)$, and $C=(1,\sqrt{3})$, is regular, $2$-point homogeneous, and even $m$-point homogeneous for any $m \in \mathbb{N}$.
The prism $\widetilde{P}(c) =P \times [0,c]$ is  homogeneous for any positive $c$ by Proposition \ref{pr.constr.1}. We see that $D(P)=\{0,4\}$
(in the notation of Lemma \ref{le.constr.2}).
Therefore, $\widetilde{P}(c)$ is $m$-point homogeneous, for all $m \in \mathbb{N}$ and for all $c\in (0,2)\bigcup\,(2,\infty)$ by Theorem \ref{th.constr.2}.
On the other hand, the prism $\widetilde{P}(2)=P \times [0,2]$ is not $2$-point homogeneous. Indeed, there is no isometry that moves the pair of vertices
$\bigl((0,0,0),\,(2,0,0)\bigr)$ to the pair of vertices $\bigl((0,0,0),\,(0,0,2)\bigr)$. It follows from Proposition \ref{pr.prism.3.0}, Lemma \ref{le.constr.2}, and
the fact that the prism $\widetilde{P}(c)$ is strongly prismatic for any $c>0$.
\end{example}

\begin{example}\label{ex.tpoin.1n}
The square $P$ with the vertices $A=(0,0)$, $B=(1,0)$, $C=(1,1)$, and $D=(0,1)$, is regular and $m$-point homogeneous for any $m \in \mathbb{N}$.
The prism $\widetilde{P}(c) =P \times [0,c]$ is  homogeneous for any positive $c$ by Proposition \ref{pr.constr.1}. We see that $D(P)=\{0,1,2\}$
(in the notation of Lemma \ref{le.constr.2}).
Therefore, $\widetilde{P}(c)=[0,1]^2\times [0,c]$ is $m$-point homogeneous, for all $m \in \mathbb{N}$
and for all $c\in (0,1)\bigcup (1,\sqrt{2})\bigcup\,(\sqrt{2},\infty)$ by Theorem \ref{th.constr.2}.
For the special cases $c=1$ and $c=\sqrt{2}$, Theorem \ref{th.constr.2} is not useful, but we can apply other ideas.

If $c=1$, then $\widetilde{P}(1)=[0,1]^3$ is a regular cube, hence, $\widetilde{P}(1)$ is $m$-point homogeneous for any $m \in \mathbb{N}$.
Notice that the regular cube is not strongly prismatic.

If $c=\sqrt{2}$, we get a rectangular parallelepiped $\widetilde{P}(\sqrt{2})$.
It is clear that $\widetilde{P}(\sqrt{2})=[0,1]\times [0,1]\times [0,\sqrt{2}]$ is strongly prismatic and is not $2$-point
homogeneous by Proposition \ref{pr.prism.3.0} and Lemma \ref{le.constr.2}.
\end{example}

Let us note that a cube is a right prism over a square, a regular tetrahedron is a right antiprism over  a  segment, and
a regular octahedron is a right antiprism over  a regular triangle. On the other hand, a regular dodecahedron, as well as a regular icosahedron,
is not a right prism or a right antiprism.

\begin{prop}\label{prop.n5n}
Let $P$ be a right prism over a regular $n$-gon, $n\geq 3$, or a right antiprism $P$ over a regular $n$-gon, $n\geq 2$
{\rm(}a regular $2$-gon is understood to be a segment{\rm)}.
Then $P$ is homogeneous.
In addition, $P$ is $2$-point homogeneous if and only if one of the following conditions is fulfilled:
\begin{enumerate}

\item $P$ is rigidly two-layered;

\item $P$  is a cube, a regular tetrahedron, or  a regular octahedron.

\end{enumerate}
\end{prop}

\begin{proof}
To prove the homogeneity of $P$, it suffices to note that the cyclic group $C_n$ acts transitively on the vertex set of any base of $P$ and there is an isometry of $P$
that interchanges its bases.

To prove the second assertion, we consider a vertex $A$ of $P$ and the isotropy group $\I (A)$ at this point.
Let $l$ be a straight line through the centers of the bases of $P$. Then the reflection $\sigma_A$
with respect to the plane, that contains $l$ and $A$, is the isometry of $P$. Since $\sigma_A$ fixes $A$, then $\sigma_A \in \I (A)$.

Let us suppose that (1) fulfilled.
If $d(A,A_1)=d(A,A_2)$ for some vertices $A_1, A_2$ of $P$, then $A_1$ and $A_2$ are from the same base, and so the reflection $\sigma_A$ interchanges $A_1$ and $A_2$.
Hence, $P$ is $2$-point homogeneous by Proposition \ref{pr:two-point homogeneous}.

If (2) fulfilled, then
$P$ is a regular polyhedron, hence,
a $2$-point homogeneous polyhedron by Theorem \ref{th:reg_pol}.

On the other hand, if $P$ is $2$-point homogeneous, but is not a regular polyhedron,
then $P$ is strongly prismatic since any of its lateral faces is
not isometric to a base. Hence (1) is fulfilled by Proposition  \ref{pr.prism.3.0}.
\end{proof}
\smallskip

Recall that all $2$-point homogeneous polygons $P$ are found in Theorem \ref{th.2dim.1}: $P$ is
either a regular polygon or a rectangle (these polygons are $m$-point homogeneous for all $m \in \mathbb{N}$).
The following proposition and Proposition \ref{prop.n5n} give us the complete classification of all
$m$-point homogeneous prisms in $\mathbb{R}^3$ for all $m\geq 2$.
\medskip

\begin{prop}\label{pr.prism.3}
Let us consider a prism $P =P_1 \times [0,c] \subset \mathbb{R}^3$ over a polygon $P_1$.  Then the following assertions hold.

1) If $P =P_1 \times [0,c]$ is $m$-point homogeneous for some $m \geq 1$, then the polygon $P_1$ is also $m$-point homogeneous.
In particular, $P_1$ is either a regular polygon
or a rectangle for $m \geq 2$.

2) If  $P =P_1 \times [0,c] \subset \mathbb{R}^3$ is not rigidly two-layered and $P_1$ is not a rectangle, then $P$ is not $2$-point homogeneous.

3) Let $P$ be a rectangular parallelepiped in $\mathbb{R}^3$ of type $[0,a]\times [0,b] \times [0,c]$, where $0<a\leq b \leq c$. Then it is a homogeneous polyhedron.
Moreover, $P$ is $m$-point homogeneous  for any natural $m$ if and only if $a^2+b^2\neq c^2$. If $a^2+b^2 =c^2$, then $P$ is not $2$-point homogeneous.
\end{prop}

\begin{proof}
Let us prove the first assertion. Suppose that $P_1$ is not a rectangle. Then $P$ is {\it strongly prismatic}, i.e. for any isometry $\psi$ of $P$,
either $\psi$ preserves $P_1$ or $\psi(P)$
is the second base $P_2$
of $P$ (the first one is $P_1$ itself). Now if we have two $m$-tuples $(A_1,A_2,\dots,A_m)$ and $(B_1,B_2,\dots,B_m)$ of vertices of $P_1$, then there is an isometry $\psi$
of $P$, that moves the first $m$-tuple to the second one. It is clear that $\psi$ preserves $P_1$. Therefore, the base $P_1$ of $P$ is also $m$-point homogeneous.
On the other hand, if $P_1$ is a rectangle, then it is $k$-point homogeneous for all $k \in \mathbb{N}$.
Theorem \ref{th.2dim.1} implies that $P_1$ is either a regular polygon or a rectangle  for $m \geq 2$.

Let us prove the second assertion.
Since $P_1$ is not a rectangle
while all lateral faces are rectangles, then $P$ is strongly prismatic (a rectangle under the action of any isometry can only transform into a rectangle).
Since $P$ is not rigidly two-layered, it is
not $2$-point homogeneous by Proposition  \ref{pr.prism.3.0}.

Let us prove the last assertion.
Obviously, $P$ can be considered as a right prism with height $c$ over the rectangle $P_1=[0,a]\times [0,b]$. We know that $P_1$
is $m$-point homogeneous for all $m\in \mathbb{N}$. If $P=P_1\times [0,c]$ is rigidly two-layered,
then it is $m$-point homogeneous for any natural $m$ according to Theorem \ref{th.constr.2}.

Now  let us consider the case when $P$ is not rigidly two-layered.
By Lemma \ref{le.constr.2}, it is equivalent to one of the
equalities $c = b$ or $c= \sqrt{a^2+b^2}$. Moreover, $P$ is strongly prismatic if and only if $c= \sqrt{a^2+b^2}$.
In this case, $P$ is not $2$-point homogeneous by Proposition \ref{pr.prism.3.0}.

Suppose that  $c = b$. If  $a = b$, then $P$ is the cube which is  $m$-point homogeneous for all
$m \in \mathbb{N}$. If $a<b$, then we can consider $P$ as a prism with a base equal to the
square with side lengths $b$ and having height $a$. For this interpretation of $P$, it is rigidly
two-layered by Lemma \ref{le.constr.2}, and so is $m$-point homogeneous for all $m \in \mathbb{N}$ by Theorem~\ref{th.constr.2}.

All assertions of Proposition \ref{pr.prism.3} are proved.
\end{proof}
\smallskip

We are going to classify all $m$-point homogeneous right antiprisms with $n$-gonal regular bases.
Any such antiprism has two bases (each isometric to other regular $n$-gons), $2n$ lateral triangular faces (each isometric to other isosceles triangles),
$2n$ vertices ($n$ vertices on each base). For $n=2$ (in this case bases are segments) and $n=3$ we get homogeneous tetrahedra and octahedra respectively.

\begin{prop}\label{pr.prism.3.00}
Let $P$ be a right  antiprism in $\mathbb{R}^3$  over a regular $n$-gon, $n\geq 2$. If $P$ is  rigidly two-layered {\rm(}as a generalized right prism{\rm)} or
$P$ is a regular polyhedron,
then it is $m$-point homogeneous for all $m \in \mathbb{N}$.
Otherwise, if $P$ is not rigidly two-layered and $P$ is not a regular polyhedron, then $P$ is not $2$-point homogeneous.
\end{prop}

\begin{proof}
According to Proposition \ref{prop.n5n}, we need to prove only the following assertion: Any $2$-point homogeneous right antiprism $P$
is $m$-point homogeneous for any $m \geq 3$.
Let $l$ be a straight line through the centers of bases of $P$. For a vertex $U$ of $P$, we consider the plane $L_U$, that contains $l$ and $U$.
It is clear that the reflection with respect to the plane $L_U$ is an isometry of $P$.

We are going to prove that $P$ is $m$-point homogeneous using induction. Now, we assume that $m \geq 3$ and $P$ is
$k$-point homogeneous for all $k=1,2,\dots, m-1$.

Let us consider two $m$-tuples
$(A_1,A_2,\dots,A_m)$ and $(B_1,B_2,\dots,B_m)$ of vertices of $P$, such that $d(A_i,A_j)=d(B_i,B_j)$, $i,j=1,\dots,m$.

If there are $i\neq j$ such that $A_i=A_j$, then we can omit $A_j$ and $B_j$ in the above $m$-tuples and use the $(m-1)$-point homogeneity of $P$
in order to find an isometry $\eta$ of $P$, moving the first $(m-1)$-tuple to the second one. Since $A_i=A_j$ and, hence, $B_i=B_j$, then $\eta$
also moves the first $m$-tuple to the second one.

In what follows, we assume that $A_i \neq A_j$ for $i\neq j$.

Now, let us prove that there are $i\neq j$ such that the plane $L_{A_i}$ is not equal to
the plane $L_{A_j}$. Indeed, otherwise all vertices $A_i$, $i=1,\dots,m$, are in one plane through the line $l$.
Using the structure of the right antiprism and the fact that all $A_i$ are pairwise distinct, it is easy to see that $m\leq 2$ in this case,
that contradicts to our assumptions. Hence, such $i\neq j$ do exist, i.~e., $L_{A_i} \neq L_{A_j}$.
Further, we can reorder elements of the above $m$-tuples such that $i=1$ and $j=2$.

Using the $2$-point homogeneity of $P$, we can find an isometry $\nu$ such that $\nu(B_1)=A_1$ and $\nu(B_2)=A_2$.
We have $d(A_1,A_i)=d(A_1,\nu(B_i))$, $d(A_2,A_i)=d(A_2,\nu(B_i))$ for every fixed $i\geq 3$.
It suffices to prove that there is an isometry $\mu$ of $P$ such that $\mu(A_1)=A_1$, $\mu(A_2)=A_2$, and $\mu(A_i)=\nu(B_i)$, $i=3,4,\dots,m$.

For any vertices $C$ and $D$ of $P$,
the equalities $d(A_1,C)=d(A_1,D)$ and $d(A_2,C)=d(A_2,D)$ mean, that $C$ and $D$ are symmetric simultaneously with respect to the planes
$L_{A_1}$ and $L_{A_2}$.
If $C\neq D$, then both planes $L_{A_1}$ and $L_{A_2}$ contain the midpoint $E$ of the segment $[C,D]$ and are orthogonal to the straight line $CD$.
Hence $L_{A_1}=L_{A_2}$, which contradicts to the choice of $A_1$ and $A_2$.
Therefore, we get $C=D$. This argument is valid for all pairs $(C,D)=(A_i,\nu(B_i))$, $i \geq 3$, i.e.
$A_i=\nu(B_i)$ for any $i=3,4,\dots,m$. Hence, the isometry $\nu$ moves $(B_1,B_2,\dots,B_m)$ to $(A_1,A_2,\dots,A_m)$,
that proves the $m$-point homogeneity of $P$.
\end{proof}

\section{On homogeneous polyhedra in $\mathbb{R}^3$}\label{sec.7}

An interesting description of homogeneous polyhedra can be found in \cite{Rob,RobCar}
(more general results for homogeneous polytopes in multidimensional Euclidean spaces can be found also in \cite{RobCarMor}).
Another description of homogeneous polyhedra follows from the classification of all tilings on the $2$-sphere
with natural transitivity properties \cite{GrShe81}.
Here we collect some useful results on the structure of homogeneous polyhedra
in order to study more narrow classes.

Let $F$ be a convex polygon in $\mathbb{R}^2$ and $A, B$ be two of its vertices. We say that $A$ is equivalent to $B$ if there is an isometry $\psi$ of $F$ such that
$\psi(A)=B$. It is easy to check that this is equivalence relation.
Hence, the vertices of $F$ are collected into equivalence classes with respect to this equivalence relation.
We will denote the {\it vertex type} of $F$ as $(n_1,n_2,\dots, n_s)$, where $n_1\geq n_2 \geq \cdots \geq n_s$ means the quantity of elements in different
{\it equivalence classes}.
Note that $n:=n_1+n_2+\cdots +n_s$ is the number of vertices of $F$. For example, the vertex type of a regular $n$-gon is $(n)$, the vertex type of an
isosceles but not regular triangle is $(2,1)$,
the vertex type of a rectangle is $(4)$,
the vertex type of an isosceles trapezoid (which is not a rectangle) is $(2,2)$.

We will say that a planar polygon is {\it inscribed} if all its vertices lie on some circle.

\begin{lemma}\label{le_m.m}
Let $F$ be an inscribed (planar) $n$-gon with the vertex type  $(m,m)$, where $n=2m$, $m\in \mathbb{N}$.
Then $F$ has sides of $2$ or $3$ different lengths and the number $m$ is even. If $m=2$, then $F$ is an isosceles trapezoid.
\end{lemma}

\begin{proof}
It is clear that $F$ is not homogeneous and has edges of different lengths (an inscribed polygon with edges of equal length is regular with the vertex type  $(2m)$).
If $F$ has more than three consecutive edges of one and the same length,
then we have at least three types of vertices of $F$ that could not be interchanged by some isometries, which contradicts the vertex type being $(m,m)$.

Let us denote by $A_1, A_2,\dots, A_n$ consecutive vertices of $F$ (we assume that $i \in \mathbb{Z}_n$). Since $F$
has the vertex type  $(m,m)$, then there are two vertices $A_i$ and $A_{i+1}$ of $F$ such that the edges $[A_{i-1},A_i]$ and $[A_{i},A_{i+1}]$ have lengths
$l_1$ and $l_2$ respectively, $l_2\neq l_1$.

Let us suppose that all edges of $F$ have only lengths $l_1$ and $l_2$.
Note that the edges of different lengths can not be alternate. Otherwise, $F$ is homogeneous, contradicting having vertex type  $(m,m)$.
Let $L$ be the maximum number of consecutive edges of $F$ of the same length (say, $l_2$). The above argument shows that $L\geq 2$.
If  $L>3$, then we have at least three types of vertices of~$F$ that could not be interchanged by some isometries,
which contradicts the vertex type being $(m,m)$. Therefore, $L \leq 3$.
Hence, either $L=2$ or $L=3$.
We see that there are vertices of $F$ with pairs of adjacent edges of lengths
$\{l_1,l_2\}$ and $\{l_2,l_2\}$. Since $F$
has the vertex type  $(m,m)$, then any vertex of $F$ has edges of lengths
either $\{l_1,l_2\}$ or $\{l_2,l_2\}$. Hence, for any edge of length $l_1$, both of its adjacent edges must have length $l_2$. Therefore, any
such edge has two vertices with pairs of adjacent edges corresponding to lengths
$\{l_1,l_2\}$. Therefore, we have $m/2$ (in particular, $m$ is even) edges of length $l_1$, while all other edges ($3m/2$ edges) have length $l_2$.
If $m=2$, $F$ is an isosceles trapezoid with three edges of equal length.

Now, let us suppose that all edges of $F$ have three lengths $l_1$, $l_2$, and $l_3$.
Since $F$ has the vertex type  $(m,m)$ and $A_i$ has adjacent edges with lengths $l_1$ and $l_2$,
then every vertex of $F$ has edges of lengths either $\{l_1,l_2\}$, or $\{l_k,l_3\}$, where  either $k=1$, or $k=2$.
Without loss of generality, we may assume that $k=2$.
Then any vertex of $F$ has edges of lengths
either $\{l_1,l_2\}$ or $\{l_2,l_3\}$. Moreover, the quantities of vertices of these two kinds coincide. It is easy to see that any edge of length $l_1$ or $l_3$
have both adjacent edges of length $l_2$, and the edges of the same length could not be adjacent.
Hence, all edges through one have length $l_2$, and all other edges have lengths either $l_1$ or $l_3$, moreover, there are an equal number (namely, $m/2$)
of edges of the last two lengths.
In particular, the number $m$ is even. If $m=2$, then $F$ is an isosceles trapezoid.
\end{proof}

\begin{prop}\label{pr.idex.1}
Suppose that a homogeneous polyhedron $P$ in $\mathbb{R}^3$ has a face $F$ with the vertex type $(n_1,n_2,\dots, n_s)$. If $w=\gcd (n_1,n_2,\dots,n_s)$,
then there is a natural number $l\geq 1$ such that any vertex $A$ of $P$ has exactly $l\cdot n_i/w$ isometric copies of $F$ among adjacent faces such that
the vertex $A$ on every such face corresponds to $i$-th equivalence class of vertices of $F$.
\end{prop}

\begin{proof} Since $P$ is homogeneous, then there are several faces of $P$ isometric to $F$, moreover, the number of such faces having a given vertex of $P$ as a proper
vertex of the $i$-th class does not depend on the vertex of $P$.
Moreover, the total number of vertices of the $1$-st, $2$-nd, $3$-rd, $\dots$, $s$-th classes across all such faces is proportional to the numbers $n_1,n_2,n_3,\dots, n_s$.
Due to the homogeneity of the polyhedron $P$,
each vertex of the polyhedron is adjacent to the same number of vertices of the above faces from each class.
This implies the proposition.
\end{proof}
\smallskip

\begin{lemma}\label{le.hom_pol-eu1}
Let $P$ be a homogeneous polyhedron in $\mathbb{R}^3$ such that everyone of its verteces is adjacent to exactly $t$ $k$-gonal faces and to $s$ $l$-gonal faces,
where $t,s \geq 1$, $k,l \geq3$.
Denote by $\mathcal{V}$, $\mathcal{E}$, $\mathcal{F}$ the numbers of vertices, edges, and faces of $P$ respectively, while $\mathcal{F}_1$ and $\mathcal{F}_2$
means the quantities of $k$-gonal and $l$-gonal faces.
Then we have
$$
\mathcal{F}=(1+\gamma) \cdot \mathcal{F}_2,\quad \mathcal{E}=\frac{k\gamma+l}{2}  \cdot \mathcal{F}_2, \quad \mathcal{V}=\frac{k\gamma+l}{t+s} \cdot  \mathcal{F}_2,
$$
$$
\mathcal{F}_1=\gamma \cdot  \mathcal{F}_2, \quad \mathcal{F}_2=\frac{4s}{2s(1+\gamma)+2l-s(k\gamma+l)},\quad \mbox{ where } \gamma=\frac{t \cdot l}{k \cdot s}\,.
$$
\end{lemma}

\begin{proof}
It is clear that $\mathcal{F}_1+\mathcal{F}_2=\mathcal{F}$. Counting all vertices of $k$-gons and $l$-gons we get
$\mathcal{V}=\frac{\mathcal{F}_1k}{t}=\frac{\mathcal{F}_2 l}{s}$, that implies $\mathcal{F}_1=\gamma \mathcal{F}_2$, where $\gamma=\frac{tl}{ks}$.
Counting all edges of $k$-gons and $l$-gons we get $2\mathcal{E}=\mathcal{F}_1k+\mathcal{F}_2l$. Hence we get $\mathcal{V}=\frac{\mathcal{F}_2 l}{s}$,
$\mathcal{E}=\frac{\mathcal{F}_2 (k \gamma+l)}{2}$, $\mathcal{F}=\mathcal{F}_2(\gamma+1)$.
Now, it suffices to apply the Euler formula $\mathcal{V}-\mathcal{E}+\mathcal{F}=2$.
\end{proof}

\medskip

For a convex polyhedron $P$, {\it the defect at a vertex} $A$ equals $2\pi$ minus the sum of all the angles at the vertex $A$
for all the faces of $P$, for which $A$ is a vertex. Since $P$ is convex, then the defect of each vertex is always positive.
Moreover, the sum of the defects of all of the vertices of $P$ is $4\pi$ (Descartes's theorem).
If, moreover, $P$ is homogeneous, the defect of any vertex of $P$ is equal to $4\pi/\mathcal{V}$, where $\mathcal{V}$ is the number of vertices of $P$.

We know that the sum of all angles of the faces at any vertex $A$ of any homogeneous polyhedron $P$ is $<2\pi$.
If $P$ has a face isometric to a polygon $F$ with the type $(n_1,n_2,\dots, n_s)$, then we can estimate from below the sum of all angles at $A$ only of faces
isometric to $F$ (we call it the impact of faces isometric to $F$).
Indeed, by Proposition \ref{pr.idex.1}, this sum is $l/w\cdot \sum_{i=1}^s n_i \alpha_i$, where $\alpha_i$ is the measure of an angle of $F$ with $i$-th type.
We know that $\sum_{i=1}^s n_i \alpha_i= \pi(n-2)$, where $n=\sum_{i=1}^s n_i$. Since $l\geq 1$, we see that
the impact of the faces, isometric to $F$ at the vertex $A$ is at least $\pi(n-2)/w$. We have the following result.

\begin{prop}\label{pr.idex.2}
Let $P$ be a homogeneous polyhedron $P$ in $\mathbb{R}^3$. Then all its faces are inscribed polygons. Moreover, the following assertions are fulfilled.

1) The impact of any face isometric to a regular $n$-gon {\rm(}a polygon with the type set~$(n)${\rm)} is at least $\pi(n-2)/n$ for $n \geq 3$.

2) The impact of any non-regular triangle is $\pi$.

3) The impact of an inscribed quadrilateral with the vertex type $(2,2)$ {\rm(}an isosceles trapezoid{\rm)} is $\pi$.

4) If $P$ has a face isometric to a quadrilateral $Q$, then the vertex type of $Q$ is either $(4)$ or $(2,2)$.

5) $P$ has no face isometric to a non-regular $n$-gon for any odd $n>3$.

6) The impact of any face isometric to a homogeneous $n$-gon, where $n=2m$ and $m\geq 2$, is at least $\pi \left(1 -\frac{1}{m}\right)\geq \pi/2$.

7) The impact of any face isometric to a non-homogeneous $n$-gon, where $n=2m$ and $m\geq 3$, is at least $2\pi\left(1 -\frac{1}{m}\right)\geq 4\pi/3$ and
the vertex type of this $n$-gon is $(m,m)$.
\end{prop}

\begin{proof} All faces of $P$ are inscribed polygons by Corollary \ref{co.inscr.face.1}. In what follows, we need
$w$ from Proposition \ref{pr.idex.1} (note also that $l\geq 1$ in the same proposition).
The first assertion follows from the equality $w=n$, where $w$ is defined in Proposition \ref{pr.idex.1}.
If a triangle is not regular then its vertex type is either $(2,1)$ or $(1,1,1)$. In any case, $w=1$ and $n-2=1$,
that proves the second assertion. To prove the third assertion it is enough to note that $n-2=w=2$.

Let us prove the fourth assertion. If the vertex type $(n_1,n_2,\dots, n_s)$ of $Q$ is distinct from $(4)$ and $(2,2)$, then $2\leq s \leq 4$, $n_s=1$, and $w=1$
(since  $w=\gcd (n_1,n_2,\dots,n_s)$). Hence, the impact of the faces, isometric to $Q$ is at least $\pi(n-2)/w=\pi(4-2)=2\pi$, that is impossible.

Let us prove the fifth assertion. Let $F$ be a non-regular face of $P$ with $n$ vertices. It is clear that $w$ is odd (since $n$ is odd and
$w$ is a divisor of $n$) and $w<n$ (since $w=n$ implies
that $F$ has the vertex type $(n)$ and $F$ should be homogeneous and even regular due to the fact that $n$ is odd). If $w=1$, then $\pi(n-2)/w =\pi(n-2) \geq 3\pi>2\pi$.
If $w \geq 3$, then $n/w \geq 3$ (since $n > w$ and $n$ is odd) and $2/w <1$, hence $\pi(n-2)/w>2\pi$.
In both cases, the impact of the faces, isometric to $F$ is at least $\pi(n-2)/w>2\pi$.
This contradiction proves the fifth assertion.

The sixth assertion follows from the fact that the vertex type of any homogeneous
$n$-gon is $(n)$ and $n=2m$. Indeed, we have $w=n$ and the impact of the faces, isometric to a given one is at least $\pi(n-2)/w=\pi(m-1)/m\geq \pi/2$.

Now, let us prove the last assertion. It is clear that $w<n$ ($w=n$ implies that $s=1$ and $P$ is homogeneous), hence,
$w\leq m$, which implies  that $\pi(n-2)/w\geq \pi(2-2/m)$.
It is clear that the inequality $w <m$ is equivalent to the fact that the vertex type of the non-homogeneous $n$-gon under consideration is not $(m,m)$.
If $w=1$ then $\pi(n-2)=\pi(2m-2)\geq 4\pi$. If $w=2$ then $\pi(n-2)/w= \pi(m-1)\geq 2\pi$.
If $w \geq 3$, then $\pi(n-2)/w > \pi(n/w -1)\geq \pi (3-1)=2\pi$ (recall that $w$ is a divisor of $n$ and $w<m=n/2$).
\end{proof}

\begin{corollary}\label{co.face.hom.1}
If a polygon $F$ is isometric to a face of some homogeneous polyhedron in $\mathbb{R}^3$, then $F$ is either a triangle or one of the following polygons:

1) a regular $n$-gon for some $n \geq 4$;

2) a non-regular $n$-gon, where $n=2m$, $m\geq 2$, with the vertex type $(2m)$, i.~e., the convex hull of a non-regular orbit of the group $D_m$;

3) an isosceles trapezoid, which is not a rectangle.
\end{corollary}

\begin{proof}
If $F$ is a homogeneous $n$-gon (in particular, a regular $n$-gon), $n\geq 4$, then it can be a face of right prism over this $n$-gon itself,
see Fig.\,\ref{Fig3}~a).
\smallskip
Now, we suppose that $F$ is a non-homogeneous (inscribed) $n$-gon, $n \geq 4$.
Since, the impact of $F$ should be $<2\pi$, then
Proposition \ref{pr.idex.2} implies that $n=2m$, $m\in \mathbb{N}$, $m\geq 2$, and the vertex type of $F$ is $(m,m)$.
By Lemma \ref{le_m.m}, we get that $m$ is even. Moreover, for $m=2$, we see that $F$ is an isosceles trapezoid.
Such an $F$ is indeed a face of a suitable polyhedron, see Example \ref{ex.isotrap.1}.

On the other hand, if $m \geq 4$, then
$2\pi\left(1 -\frac{1}{m}\right)\geq 3\pi/2$.
Therefore, any vertex of $P$ has exactly three adjacent faces: two faces  that are isometric to $F$ and the third one that is a regular triangle
(by assertion 2) of
Proposition \ref{pr.idex.2}).
Now, Lemma \ref{le.hom_pol-eu1} implies that $P$ has $6$ $n$-gonal and $8$ triangular faces, $36$ edges, and $24$ vertices, while $m=4$ and $n=8$.
We obtain a polyhedron that is combinatorially equivalent to a partial truncation of the cube, see Example \ref{ex.par_trun1}.
Let us fix a vertex $A$ of $P$, which is adjacent to the pair of $8$-gonal faces $F_1$, $F_2$, and to a triangular face $F_3$.
Let us consider also a vertex $B$ of $P$, such that
$[A,B]$ is the common edge of $F_1$ and $F_3$. Since $P$ is homogeneous, there is an isometry $\psi$ of $P$ such that $\psi(B)=A$.
Since $F_3$ is a unique triangular face adjacent to $A$ and to $B$, then $\psi(F_3)=F_3$ and $\psi(B)=A$, $\psi(A) \in F_3$.

Since $\psi$ acts on $F_3$ as an isometry, there are two variants: (1) $\psi$ rotates $F_3$ to the angle $2\pi/3$,
(2) the restriction of $\psi$ on $F_3$ is the reflection
with respect to a straight line orthogonal to $[A,B]$.
In both these variants, $\psi$ moves the edge of $F_1$ adjacent to $B$ and distinct from $[A,B]$
to the edge of $F_1$ adjacent to $A$ and distinct from $[A,B]$.

This means that two edges of $F_1$ that are adjacent to $[A,B]$ have equal lengths.
This argument can be applied to any vertex of $P$.
Therefore, all edges of $F_1$ (hence, of any octagonal face of $P$) have
sides that alternate in length.
This means that $F_1$ (being inscribed) is a homogeneous polygon. We get a contradiction with our assumption that $F$ (isometric to $F_1$) is non-homogeneous,
that proves the corollary.
\end{proof}
\medskip

Let us discuss the case when $F$ is a triangle.
If it is an isosceles (a regular, in particular) triangle, then $F$ is isometric to a face of suitable right antiprism, see Fig.\,\ref{Fig3}~b).

Let us consider a triangle $F$ with pairwise distinct side lengths $0<a < b < c$ that is isometric to a face of a homogeneous polytope $P$.
Then any vertex $A$ of $P$ has exactly three adjacent faces isometric to $F$, moreover, $A$ corresponds to pairwise distinct vertices of $F$
in these three faces.

If there are no other faces adjacent to $A$, then we get a homogeneous tetrahedron as in Example \ref{ex.tetra.1}.
If $\alpha, \beta, \gamma$ are angles opposite to the sides of the triangle with lengths $a < b < c$, then
$\pi-\gamma =\alpha+\beta> \gamma$, hence, $\gamma <\pi/2$,
this is equivalent to the inequality $c^2<a^2+b^2$. Hence, $F$ is an acute triangle, and every acute triangle is a face of a suitable homogeneous tetrahedron.

If there is some other face adjacent to $A$, then it is isometric either to a regular polygon or to a non-regular homogeneous $n$-gon, where $n=2m$, $m\geq 2$,
by Corollary~\ref{co.face.hom.1}.

Suppose that any vertex of $P$ is adjacent to 4 faces ($3$ triangular and $1$ $n$-gonal)
Now, Lemma \ref{le.hom_pol-eu1} implies that $P$ has $2$ $n$-gonal and $2n$ triangular faces, $4n$ edges, and $2n$ vertices.
We obtain a polyhedron that is combinatorially equivalent to a right antiprism over a homogeneous $n$-gon.
If $n=3$, we get a homogeneous octahedron.

Suppose that any vertex of $P$ is adjacent to more than 4 faces. It is easy to see that it should be exactly 5 faces:
either all of them are triangles (hence, we obtain a polyhedron that is combinatorially equivalent to the icosahedron)
or four triangular faces and one rectangular face.
In the latter case, Lemma \ref{le.hom_pol-eu1} implies that $P$ has $6$ rectangular faces and $32$ triangular faces, $60$ edges, and $24$ vertices.
Hence, we get a polyhedron that is combinatorially equivalent to the snub cube.
\smallskip

\begin{remark}
Any homogeneous octahedron can be realized as a right antiprism over a regular triangle.
Hence, we have a $1$-parameter family of homogeneous octahedra (up to a similarity).
A homogeneous icosahedron should be invariant under some subgroup of the full isometry group of the regular icosahedron, see polyhedra $J$, $K$, $R$ of Fig.~3 on P.~9 in
\cite{RobCar}. Finally, there is a $2$-parameter
family of homogeneous polyhedra that are combinatorially equivalent to the snub cube, see e.~g. Fig.~3 on P.~84 in \cite{RobCarMor}.
\end{remark}

\begin{problem}\label{prob.triangle}
Classify all triangles that are isometric to some faces of homogeneous polyhedra.
\end{problem}

Possible approaches for this problem could be found in  \cite{GrShe81}, \cite{KoKo17}, and \cite{RobCar}.
Note that we have the examples of obtuse triangles with pairwise different side lengths,
that are isometric to faces of some homogeneous icosahedrons (see Proposition \ref{pr.homicos.1}).

\smallskip

\begin{prop}\label{pr.four.edges}
The edges of a homogeneous polyhedron $P$ in $\mathbb{R}^3$ may have at most three different lengths.
\end{prop}

\begin{proof}
Let us suppose that there are $u$ different edge lengths for $P$, where $u \geq 4$.
Let us fix a  vertex $A$ of $P$. Denote by $F_i$, $i=1,\dots,v$, all faces of $P$
that contain $A$ and such that two sides of $F_i$, adjacent to $A$, have different lengths.
Obviously, $v\geq u \geq 4$.

Suppose that one of these faces (say, $F_1$) is triangular. Its impact is $\pi$ by 2) of Proposition  \ref{pr.idex.2}.
Moreover, there are two faces (say, $F_2$ and $F_3$) with sides
adjacent to $A$ whose lengths are different from the lengths of the sides of $F_1$ (since $u \geq 4$). The impacts of $F_2$ and $F_3$ should be
$< \pi$, therefore, each of these faces is isometric to a homogeneous $n$-gon, where $n=2m$ and $m\geq 2$ by Proposition \ref{pr.idex.2}.
Each angle of this $n$-gon is equal to $\pi(n-2)/n=\pi (m-1)/m \geq \pi/2$. Hence, the sum of angles of $F_2$ and $F_3$ at the vertex $A$ is at least
$\pi$, that is impossible.

The same arguments can be applied to the case when $F_1$ is an isosceles trapezoid or
a non-homogeneous $n$-gon with the vertex type $(m,m)$, where $n=2m$, and $m\geq 3$
(the impact of any such polygon is at least $\pi$ and its sides have  $2$ or $3$ types of length, see Lemma \ref{le_m.m}).

If every face $F_i$, $i=1,\dots,v$, is not a triangle, an isosceles trapezoid, or
a non-homogeneous $n$-gon, where $n=2m$ and $m\geq 3$,
then $F_i$ is a homogeneous non-regular $n$-gon, where $n=2m$ and $m\geq 2$ by Proposition  \ref{pr.idex.2}.
Each angle of $F_i$ is at least $\pi \left(1 -\frac{1}{m}\right)\geq \pi/2$.
Hence, the sum of angles  of $F_i$ at the vertex $A$, $i=1,\dots,v$,  is $\geq v\cdot \pi/2 \geq 2\pi$, that is impossible.
This contradiction proves the proposition.
\end{proof}
\smallskip

Since the orbits of the groups $D_m$ and $C_m$ in $\mathbb{R}^2$ are homogeneous polygons, then Proposition \ref{pr.constr.1} give us examples of $3$-dimensional
homogeneous polyhedra that are right prisms over regular polygons and over polygons $R_m(\alpha,\beta)$ (see Example \ref{ex.2dim.1}).

For homogeneous tetrahedra with non-regular faces see Example \ref{ex.tetra.1} and
\cite{Edmonds05, Edmonds09} for more details on homogeneous simplices.

\begin{prop}\label{pr.antipr.1}
Let $P$ be a homogeneous polyhedron in $\mathbb{R}^3$ with faces of two types: 1) regular $n$-gons with side length $d$ and $n\geq 4$, 2)
non-regular triangles with side lengths $a,b,c$.
Then the triangular faces are isosceles (say, $a=b$), there are exactly $2$ regular $n$-gonal faces ($d=c$) and $P$ is a right antiprism.
In particular, triangular faces can be right or obtuse isosceles triangles.
\end{prop}

\begin{proof}
By Proposition \ref{pr.idex.2}, we see that every vertex of $P$ is adjacent to one $n$-gonal face and to $3$ triangular faces.
Using Euler's formula, we easily get that there are exactly $2$ regular $n$-gonal faces and $2n$ triangular faces.
Indeed, in the notation of Lemma~\ref{le.hom_pol-eu1}, we get $k=n$, $t=1$, $l=s=3$, $\gamma=1/n$, hence (as it is easy to see) $\mathcal{F}_1=2$ and $\mathcal{F}_2=2n$.

We denote these two $n$-gonal faces by $X_1X_2\dots X_n$ and $Y_1Y_2\dots Y_n$.
Let $O_1$ and $O_2$ be the respective centers of these faces, while $O$ is the center of $P$
(since $P$ is inscribed). If $\I(P)$ is the isometry group of $P$ and $\psi \in \I(P)$, then either
$\psi(X_1X_2\dots X_n)=X_1X_2\dots X_n$ and $\psi(Y_1Y_2\dots Y_n)=Y_1Y_2\dots Y_n$
or
$\psi(X_1X_2\dots X_n)=Y_1Y_2\dots Y_n$ and $\psi(Y_1Y_2\dots Y_n)=X_1X_2\dots X_n$ (because we have only two such faces).

For any $i=2\dots n$ there is an $\psi \in \I(P)$ such that $\psi(X_1)=X_i$. It is clear that $\psi(X_1X_2\dots X_n)=X_1X_2\dots X_n$.
Therefore, a subgroup of $\I(P)$ of index $2$ acts transitively on $X_1X_2\dots X_n$ (as well as on $Y_1Y_2\dots Y_n$). This implies that
$O$ is the center of the segment $[O_1,O_2]$ and the faces $X_1X_2\dots X_n$ and $Y_1Y_2\dots Y_n$ are orthogonal to $[O_1,O_2]$.

Let $L_1$ and $L_2$ be planes in $\mathbb{R}^3$, containing the faces $X_1X_2\cdots X_n$ and $Y_1Y_2\dots Y_n$ respectively, and
$L$ be the plane through $O$, that is parallel to $L_1$ and $L_2$. Let us consider $X'_1X'_2\dots X'_n$, the orthogonal projection of
$X_1X_2\dots X_n$ to $L_2$. If $C$ is the circumscribed circle for $Y_1Y_2\dots Y_n$ in $L_2$, then $C$
is also the circumscribed circle for $X'_1X'_2\dots X'_n$. Without loss of generality, we may assume that $X_i'$ is situated in the smallest arc of $C$
between $Y_i$ and $Y_{i+1}$, $i=1,\dots,n-1$. Indeed, the case $X_1'=Y_j$ for some $j$ is impossible, because we have only triangular faces that are
distinct from regular $n$-gons.
We may suppose that $a=d(X_1,Y_1)$ and $b=d(X_1,Y_2)$.

We are going to prove that $d(X_1',Y_1)=d(X_1',Y_2)$. This equality implies that all triangular faces are equilateral and $P$ is a right antiprism.
Let us suppose that $d(X_1',Y_1)\neq d(X_1',Y_2)$, which is equivalent to $a\neq b$.

There is $\varphi \in \I(P)$ such that $\varphi(Y_1)=X_1$. It is easy to see that $\varphi(X_1)=Y_1$. Indeed,
$\varphi(X_1)\in Y_1Y_2\dots Y_n$,
$$
d(\varphi(X_1),X_1)=d(\varphi(X_1),\varphi(Y_1))=d(X_1,Y_1)=a\neq b=d(X_1,Y_2),
$$
hence, $\varphi(X_1)\neq Y_2$.
For any $i \geq 3$, we have $d(X_1,Y_i)> \max\{a, b\}$, therefore, $d(X_1,Y_i)=a$ only for $i=1$ and  $\varphi(X_1)=Y_1$.

Further, since $\varphi(X_1X_2\dots X_n)=Y_1Y_2\dots Y_n$ and $\varphi(Y_1Y_2\dots Y_n)=X_1X_2\dots X_n$, then $\varphi(L)=L$, $\varphi(O_1)=O_2$, and
$\varphi(O_2)=O_1$
Therefore, $\varphi^2$ is the identity on the straight line $O_1O_2$, $\varphi^2(L)=L$,  $\varphi^2(X_1)=X_1$, $\varphi^2(Y_1)=Y_1$,
hence, $\varphi^2=\Id$.

Keeping in mind that $\varphi(Y_1)=X_1$ and $\varphi(Y_1Y_2\dots Y_n)=X_1X_2\dots X_n$,
we have two cases: either 1) $\varphi(Y_i)=X_i$ for all $i$,  or 2) $\varphi(Y_i)=X_{n+2-i}$ for $i=2,\dots,n$.
Since $\varphi^2=\Id$, we get that $\varphi(Z_i)=Z_i$, where $Z_i$ is the midpoint of the segment $[Y_i, \varphi(Y_i)]$, $i=1,\dots,n$.
It is clear that $Z_i\in L$ for all $i$ and the set $\{Z_i\,|\, 1\leq i \leq n\}$ is not on a straight line through the point $O$.
This observation implies that $\varphi$ is the identity on $L$. Consequently, $\varphi$ is the reflection with respect to $L$ (recall that $\varphi(O_1)=O_2$).
But this is impossible because $X_1' \neq Y_1 = \varphi(X_1)$. This contradiction
proves that $d(X_1',Y_1)=d(X_1',Y_2)$, hence, $a=b$, and all faces distinct from two regular $n$-gon are isosceles triangles.
We have proved that $P$ is a right antiprism. It should be noted also that the angle between two sides of equal lengths of any triangular face can be right or obtuse.
\end{proof}

\begin{remark}
The assertion of Proposition \ref{pr.antipr.1} is valid also if we
assume in addition that there are exactly four faces adjacent to any vertex of $P$ and replace assumptions 1) to the following one: the regular triangle with side length $d$.
\end{remark}

This remark implies the following proposition.

\begin{prop}\label{pr.antipr.2}
Let $P$ be a homogeneous polyhedron in $\mathbb{R}^3$
that is combinatorially equivalent to the regular octahedron.
Then $P$ is isometric to some right antiprism with regular triangular bases.
\end{prop}

\begin{proof}
All faces of $P$ are triangular and there are exactly four faces adjacent to any vertex of $P$.
If all faces are regular, then $P$ is a regular octahedron, that is partial case of right antiprisms with regular triangular bases.
It is not difficult to check that if there are at least two different edge lengths, then there are at least two unequal
triangles which include a given vertex of $P$. Then by point 2) from Proposition \ref{pr.idex.2},
one of these two triangles is regular.
Now, it suffices to apply arguments from the proof of Proposition \ref{pr.antipr.1}.
\end{proof}

\begin{figure}[t]
\begin{minipage}[h]{0.45\textwidth}
\center{\includegraphics[width=0.85\textwidth, trim=0in 0in 0mm 0mm, clip]{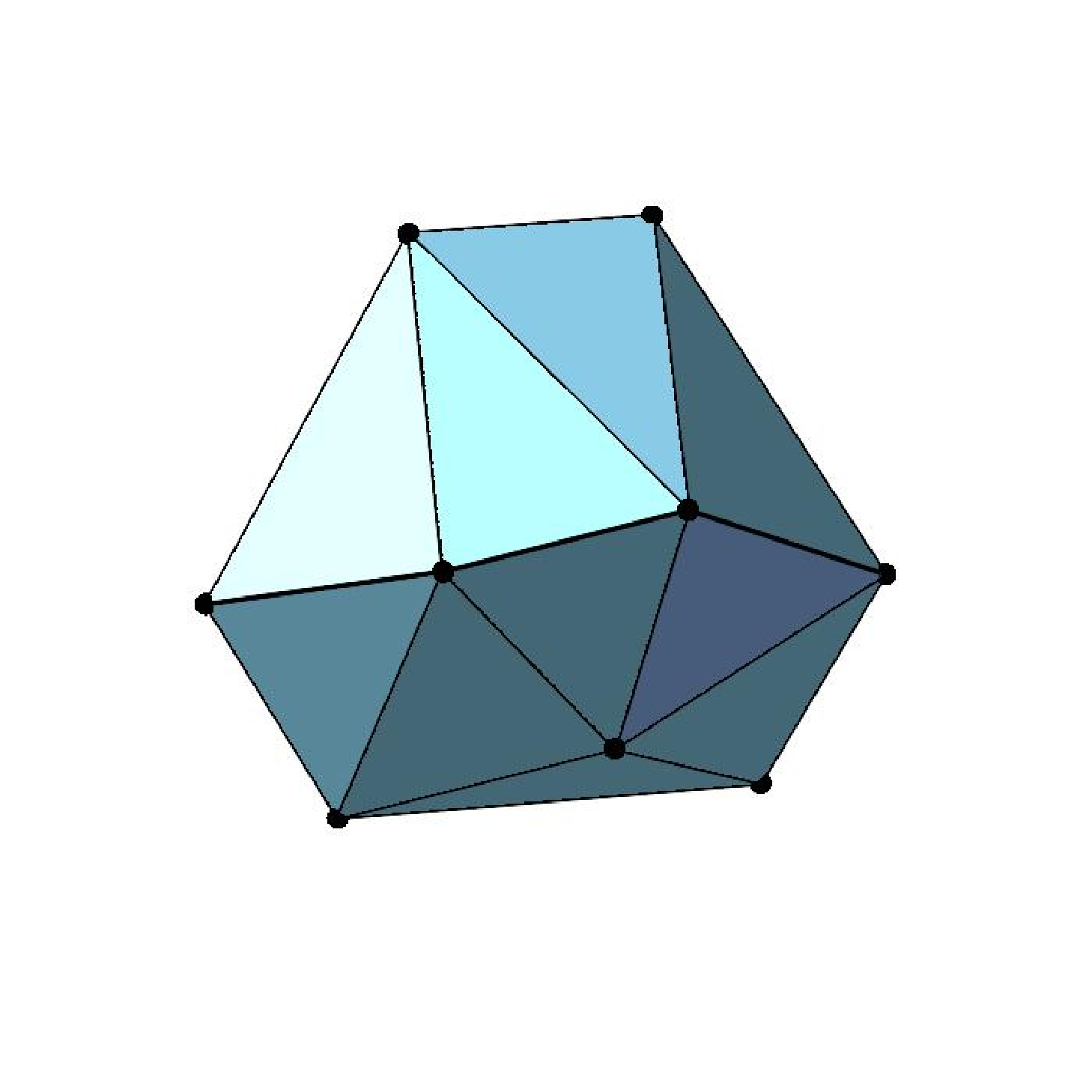} \\ a)}
\end{minipage}
\begin{minipage}[h]{0.45\textwidth}
\center{\includegraphics[width=0.9\textwidth, trim=0in 0in 0mm 0mm, clip]{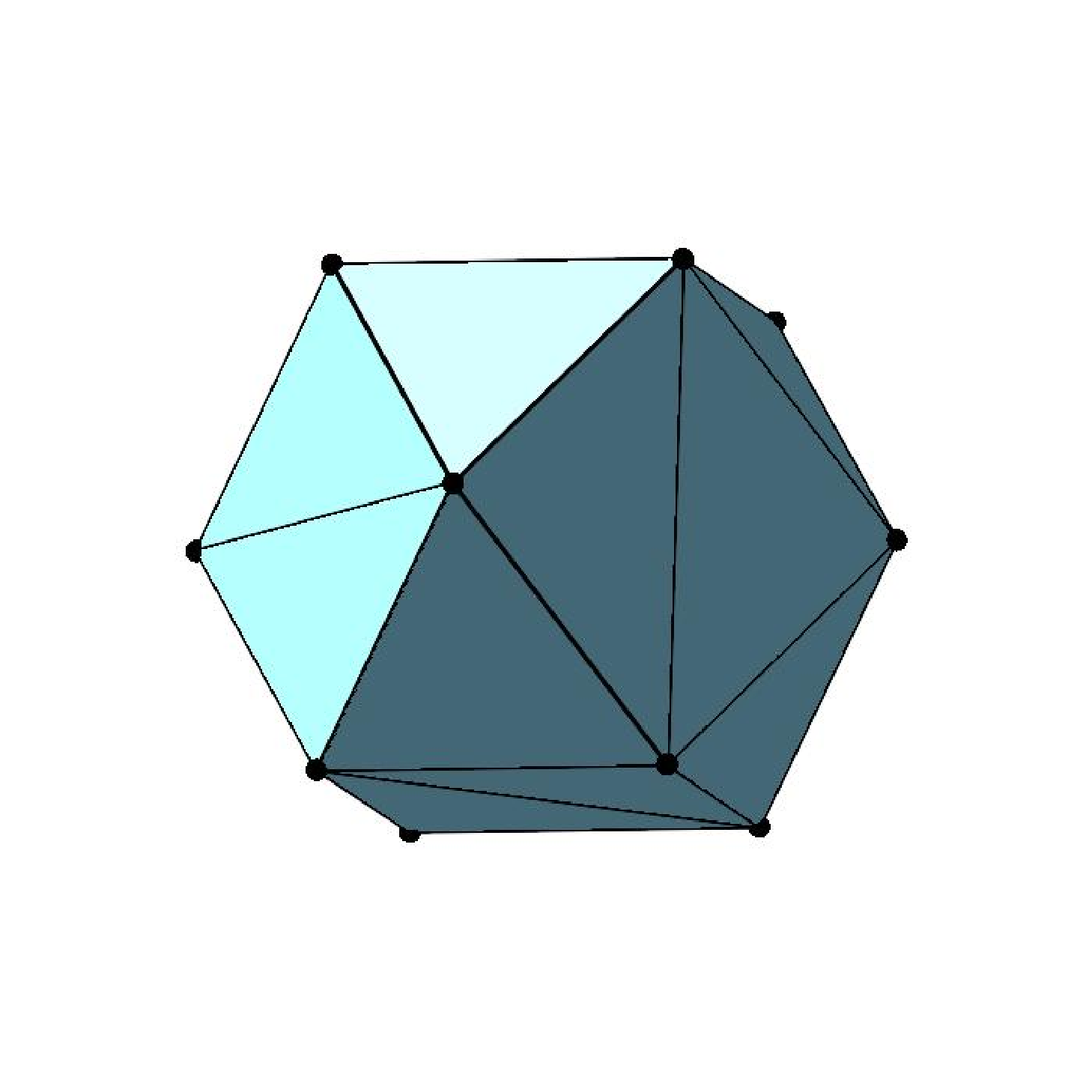} \\ b)}
\end{minipage}
\caption{{{\normalsize{Homogeneous icosahedra:}}
\normalsize{a) for $(\varepsilon, \alpha,\beta)=(1,7/2,4)$;
b)~for $(\varepsilon, \alpha,\beta)=(1,999,1000)$.}}}
\label{Fig7}
\end{figure}

\begin{prop}\label{pr.homicos.1}
There is a 3-parameter family of (pairwise non-isometric) homogeneous icosahedra $P$ in $\mathbb{R}^3$ (it has 8 regular triangular faces of 2 different side lengths
and 12 pairwise isometric non-regular triangular faces).
For some of these polyhedra, non-regular triangular faces are obtuse triangles with pairwise distinct side lengths.
\end{prop}

\begin{proof} We consider an explicit construction for such polyhedra. Let us fix three real numbers $0\leq \varepsilon <\alpha < \beta$ and the following
points in $\mathbb{R}^3$:
\begin{eqnarray*}
A_1=(\varepsilon, \alpha, \beta),\quad \,\, A_2=(\varepsilon, -\alpha, -\beta),\quad \,\, A_3=(-\varepsilon, \alpha, -\beta),\quad \,\,A_4=(-\varepsilon, -\alpha, \beta),\\
A_5=(\alpha, \beta, \varepsilon),\quad \,\,A_6=(-\alpha, -\beta, \varepsilon),\quad \,\,A_7=(\alpha, -\beta, -\varepsilon),\quad \,\,A_8=(-\alpha, \beta,-\varepsilon),\\
A_9=(\beta, \varepsilon, \alpha),\quad A_{10}=(-\beta, \varepsilon, -\alpha),\quad A_{11}=(-\beta, -\varepsilon, \alpha),\quad A_{12}=(\beta, -\varepsilon, -\alpha).
\end{eqnarray*}
Note that with any point $(x,y,z)$ this set contains also the point $(y,z,x)$ as well as all points of the type $(\pm x, \pm y, \pm z)$,
where an even number of ``$-$'' is used.

Let $P$ be the convex hull of the above $12$ points in $\mathbb{R}^3$. Since $\beta > \alpha >\varepsilon$,
then $P \subset C:=[-\beta, \beta]^3$, moreover, any point $A_i$ is in some face of $C$. This means that $P$ has exactly $12$ vertices $A_i$, $i=1,\dots,12$,
see Fig.\,\ref{Fig7}.
Moreover, on any face of $C$, there are exactly two vertices of $P$, say, $A_i$ and $A_j$. It is easy to see that $[A_i,A_j]$ is the edge of $P$ with length
$d_1=2\sqrt{\varepsilon^2+\alpha^2}$. We have exactly $6$ such edges that used all $12$ vertices of $P$.
It is clear that $P$ is homogeneous with respect to the isometry group $I(P)$, generated by the isometries
$(x,y,z)\mapsto (-x,-y,z)$ and  $(x,y,z)\mapsto (y,z,x)$.

Further, for any $A_i$ we have $x+y+z=h$, where $h$ is one of the numbers
$\beta+\alpha+\varepsilon>\beta-\alpha-\varepsilon \geq -\beta+\alpha-\varepsilon> -\beta-\alpha+\varepsilon$.
It is clear that the planes $x+y+z=\beta+\alpha+\varepsilon$ and $x+y+z=-\beta-\alpha+\varepsilon$
contain the faces $A_1A_5A_9$ and $A_2A_6A_{10}$ of $P$ respectively. Note that the faces  $A_1A_5A_9$  and $A_2A_6A_{10}$ are regular triangles with side lengths
$d_2:=\sqrt{2}\cdot\sqrt{\varepsilon^2+\alpha^2+\beta^2-\varepsilon \alpha - \varepsilon \beta-\alpha \beta}$ and
$d_3:=\sqrt{2}\cdot\sqrt{\varepsilon^2+\alpha^2+\beta^2+\varepsilon \alpha + \varepsilon \beta-\alpha \beta}$ respectively.
These two faces produce three  faces isometric to itself via the isometries
$(x,y,z) \mapsto (x,-y,-z)$, $(x,y,z) \mapsto (-x,y,-z)$,  and $(x,y,z) \mapsto (-x,-y,z)$.

Therefore, $P$ has four faces, isometric to the face $A_1A_5A_9$ (the regular triangle with the side length $d_2$)  and has four faces, isometric to the face
$A_2A_6A_{10}$ (the regular triangle with the side length $d_3$). It is easy to see that $P$ has also $12$ faces isometric to the face
$A_1A_4A_9$ with side lengths $d_1$, $d_2$, and $d_3$ (the latter face has a common edge $A_1A_9$ with the face $A_1A_5A_9$.
Therefore, we obtain a 3-parameter family of (pairwise non-isometric) homogeneous icosahedra $P$ in $\mathbb{R}^3$
with 8 regular triangular faces of two different side lengths
and 12 pairwise isometric non-regular triangular faces.

Let us show that the face $A_1A_4A_9$ (with side lengths $d_1$, $d_2$, and $d_3$) can be obtuse triangle with different side lengths.
It is enough to find a choise $0 < \varepsilon <\alpha < \beta$ such that $\beta^2-(\alpha+\varepsilon)\beta-\varepsilon \alpha\neq \varepsilon^2 +\alpha^2$
(that is equivalent to $d_1\neq d_2$) and
$d_1^2+d_2^2 <d_3^2$.
The latter inequality is equivalent to the next one: $\varepsilon( \alpha+\beta)>\varepsilon^2 +\alpha^2$.
Hence, for any choice  of the above parameters with $\beta>\varepsilon+\alpha^2/\varepsilon- \alpha$ and
$\beta^2-(\alpha+\varepsilon)\beta-\varepsilon \alpha\neq \varepsilon^2 +\alpha^2$,
we obtain  a polyhedron $P$ such that some of its faces are obtuse triangles with different side lengths.

Now, we obtain more detailed information on possible types of the face $A_1A_4A_9$.
Let $\gamma_i$ be the angle in $\triangle A_1A_4A_9$ opposite to the side with the length $d_i$.
Direct calculations show that
\begin{eqnarray*}
\cos (\gamma_1)&=&\frac{2(1 - y)}{\sqrt{(y - x)^2 + (1 - y)^2 + (1 - x)^2}\cdot \sqrt{(y + x)^2 + (1 - y)^2 + (1 + x)^2)}},  \\
\cos (\gamma_2)&=&\frac{x^2 + y^2 + x(1 + y)}{\sqrt{x^2 + y^2}\cdot \sqrt{(y + x)^2 + (1 - y)^2 + (1 + x)^2}}, \\
\cos (\gamma_3)&=&\frac{x^2 + y^2 - x(1 + y)}{\sqrt{x^2 + y^2}\cdot \sqrt{(y - x)^2 + (1 - y)^2 + (1 - x)^2}},
\end{eqnarray*}
where $0<x:=\varepsilon/\beta< y:=\alpha/\beta <1$.

Direct calculations imply $d_2^2+d^2_3=4(\varepsilon^2+\alpha^2+\beta^2-\alpha \beta) = d_1^2+4\beta(\beta -\alpha)> d_1^2$.
Therefore, $\cos (\gamma_1)=\frac{d_2^2+d^2_3-d_1^2}{2d_2d_3}>0$. Moreover, it is easy to check that $\gamma_1$ can be any number in $(0,\pi/2)$
for suitable $\varepsilon, \alpha, \beta$.

Note also that $d_3^2-d^2_2=4\varepsilon (\alpha+\beta) >0$, hence,  $\cos (\gamma_2)=\frac{d_1^2+d^2_3-d_2^2}{2d_2d_3}\geq \frac{d_1^2}{2d_2d_3} >0$.
Note that $\gamma_2$ also can be any number in $(0,\pi/2)$
for suitable $\varepsilon, \alpha, \beta$ (it suffices to consider the corresponding limits for $\cos (\gamma_2)$ when $(x=y^2,y)\to (0,0)$ and $(x=y^2,y)\to (1,1)$).

Note that not every triangle can be a face $A_1A_4A_9$, but $A_1A_4A_9$ can be an obtuse triangle with different side lengths.
Indeed, it could be checked by standard methods, that
$$
-\frac{1}{2}\leq \cos (\gamma_3)=\frac{x^2 + y^2 - x(1 + y)}{\sqrt{x^2 + y^2}\cdot \sqrt{(y - x)^2 + (1 - y)^2 + (1 - x)^2}} \leq \frac{1}{\sqrt{2}}
$$
for all $(x,y)\in \mathbb{R}^2$ such that $0 <x \leq y \leq 1$. Note that for $y=1$ this function tends to $\frac{1}{\sqrt{2}}$ when $x\to 0$.
On the other hand, for $y=x$ this function is equal to $-\frac{1}{2}$ for any $x\in (0,1)$. These arguments show that
$\gamma_3$ can be any number in $(\pi/4,2\pi/3)$ for suitable $\varepsilon, \alpha, \beta$.

If we put $\varepsilon =0$, $\alpha = \frac{\sqrt{5}-1}{2}$, and $\beta=1$, then we get $d_1=d_2=d_3=\sqrt{5} - 1$.
Hence, the triangle $A_1A_4A_9$ is regular and $P$ is a regular icosahedron.

Let us put $x = 1 - (1 + \delta)\cdot t$ and $y = 1 - t$, where $\delta \in (0,1)$, $t \in (0, 1/2)$. Then we have
\begin{eqnarray*}
\cos (\gamma_1)&=& \frac{1}{2\sqrt{\delta^2 + \delta + 1}} + \frac{2\delta + 3}{8\sqrt{\delta^2 + \delta + 1}}\cdot t + O(t^2), \\
\cos (\gamma_2)&=& 1  -\frac{1+\delta}{4} \cdot t + O(t^2), \\
\cos (\gamma_3)&=& -\frac{1}{2\sqrt{\delta^2 + \delta + 1}} + \frac{\delta(1+2\delta)}{4\sqrt{\delta^2 + \delta + 1}} \cdot t + O(t^2)
\end{eqnarray*}
when $t \to 0+0$.

Now, for any fixed $\delta \in (0,1)$ and for a small positive $t$ we see that $\gamma_2$ is a small positive number,
$\gamma_1$ is a little smaller than $\arccos(\frac{1}{2\sqrt{\delta^2 + \delta + 1}})$ and, finally,
$\gamma_3$ is a little smaller than $\pi - \arccos(\frac{1}{2\sqrt{\delta^2 + \delta + 1}})$.
Since $\delta$ is any number from $(0,1)$, we get that the biggest angle of the triangle $A_1A_4A_9$ can be any number from $(\pi/3,2\pi/3)$,
while all sides of $A_1A_4A_9$ have different  lengths.
\end{proof}

\begin{remark}
It is easy to see that $d_2=d_3$ is equivalent to $\varepsilon =0$. In this case
we obtain the isosceles triangle $A_1A_4A_9$. Moreover, if $\beta=\frac {1+\sqrt{5}}{2} \cdot \alpha$, then $P$ is a regular icosahedron.
\end{remark}

\begin{remark}
For the limit case $\varepsilon =\alpha=0$, the above construction gives (as a polyhedron~$P$) a regular octahedron with edge length $\sqrt{2} \beta$.
For another special case, when $\beta=\alpha$, we get a nonregularly truncated cube $TC_3(c)$ for some $c$, see Proposition \ref{ex.trcube.1}.
For a partial subcase $\beta=\alpha$ and $\varepsilon =0$ we get a cuboctahedron (a fully truncated cube).
\end{remark}

\begin{remark}
A description of homogeneous icosahedrons in terms of quaternions was obtained in Section~4 of \cite{KoKo17}.
\end{remark}
\medskip

Descriptions of homogeneous snub cubes and icosahedrons in terms of quaternions can be found in Section~5 of \cite{KoKo17}.
It is clear that every vertex of the first polyhedron $P$ is the vertex of four triangular faces and one rectangular face.
Using Euler's formula, we easily get that $P$ has exactly $6$ (pairwise isometric) rectangular faces and $32$ triangular faces
(at least $8$ of them are pairwise isometric regular triangles, the other $24$ triangles are also pairwise isometric).
If $A$ and $B$ are the edge lengths of regular triangular faces and rectangular faces respectively, then (up to a homothety) we have
$A:=\sqrt{2(1+y+y^2)}$ and $B:=\sqrt{2(1+x+x^2)}$,
whereas the lengths of edges for non-regular faces are
$A, \, B, \, C:=\sqrt{2x^2+y^2}$, see details in  \cite[Section~5]{KoKo17}.

\medskip

\section{On $2$-point and $3$-point homogeneous polyhedra in $\mathbb{R}^3$}\label{sec.8}

The following result is very important for our goals.

\begin{corollary}[Corollary 3.2 in \cite{BerNik22}]\label{co:if-homog}
Every $n$-dimensional convex $n$-point homogeneous polytope in $\mathbb{R}^n$ which has the vertex set $M$ with cardinality $m\geq n+1$, is $m$-point homogeneous.
\end{corollary}

Hence, for $n=3$, the $3$-point homogeneity of a non-degenerate polyhedron $P$ implies its $m$-point homogeneity  for all $m\in \mathbb{N}$.

\medskip

Now, let $P$ be a $2$-point homogeneous polytope in $\mathbb{R}^3$
with the vertex set $M$.
The following proposition (which is a refinement of Proposition \ref{prop.gen0} for $3$-dimensional polyhedra)
gives a powerful tool to study $2$-point homogeneous polyhedra.

\begin{prop}\label{prop.n2}
Let $P$ be a $2$-point homogeneous polyhedron in $\mathbb{R}^3$ with the vertex set~$M$. Then for any points $A,B \in M$,
such that the center $O$ of $P$ is not the midpoint of the segment $[A,B]$,
either the reflection $\varphi_{A,B}$ with respect to the plane $\pi_{A,B}$,
that is orthogonal to $\overrightarrow{AB}$ and passes through the midpoint $D$ of the segment $[A,B]$,
or the rotation $\psi_{A,B}$ by angle $\pi$ around $OD$ is the isometry of~$P$.
\end{prop}

\begin{proof}
Since $P$ is $2$-point homogeneous, there is an isometry $\sigma$ of $P$, such that $\sigma(A)=B$ and $\sigma(B)=A$.
Obviously, $\sigma$ fixes all points on the straight line $OD$ and the plane through the points $A$, $B$, and $O$.
If $\sigma$ preserves the orientation, then $\sigma=\psi_{A,B}$; if not, then $\sigma=\varphi_{A,B}$.
\end{proof}

\begin{prop}\label{prop.n3}
Let $P$ be a $2$-point homogeneous polyhedron in $\mathbb{R}^3$ and $F$ be a face of $P$.
If $F$ is not a triangle or a rectangle, then $F$ is a regular $n$-gon for some $n\geq 5$.
\end{prop}

\begin{proof} Let us suppose that $F$ is neither a triangle, nor a rectangle. Denote by $O_1$ the center of gravity of $F$.
Let us consider four consecutive vertices $A, B, C, D$ on this face. By our assumption $[A,C]$ is not an edge of $P$.
Let us prove that $F$ is not invariant under the rotation by the angle $\pi$ around the line $OE$, where $O$ is the center of $P$ and $E$ is the midpoint of $[A,C]$.
Indeed, it is clear if $F$ is a $n$-gon with $n \geq 5$. On the other hand, if $F$ is a quadrilateral invariant under the above rotation, then it is a rectangle
(since $F$ is inscribed polygon by Corollary \ref{co.inscr.face.1}), that is impossible by assumption.

Since $P$ is $2$-point homogeneous, then there is an isometry $\sigma$ of $P$ that interchanges the points $A$ and $C$.
Due to Proposition \ref{prop.n2}, $\sigma$ is
the reflection with respect to the plane $\pi_{A,C}$,
that is orthogonal to $\overrightarrow{AC}$ and passes through $E$.
Hence, $\sigma(F)=F$ and $\sigma(O_1)=O_1$. In particular, we see that $d(O_1,A)=d(O_1,C),$ $\sigma(B)=B,$ and $d(A,B)=d(B,C)$.
The same arguments work for any consecutive triples of vertices of $F$. Therefore, all edges of $F$ have the same lengths.
Moreover, the face $F$ is an inscribed polygon by Corollary \ref{co.inscr.face.1}.
Hence, $F$ is a regular polygon.
\end{proof}
\smallskip

We have the following obvious corollary.

\begin{corollary}\label{co.2-pont.f}
Let $P$ be a $2$-point homogeneous polyhedron and $F$ be a face of $P$.
If the polygon $F$ is not a triangle, then $F$ is $2$-point homogeneous.
\end{corollary}

\begin{prop}\label{prop.n4}
Let $P$ be a $2$-point homogeneous polyhedron in $\mathbb{R}^3$ with a face $F$ that is a non-regular triangle. If all sides of $F$ have different lengths,
then $P$ is a tetrahedron. If two sides of $F$ have equal length, then $P$ is either a tetrahedron or a right antiprism, where $F$ is a lateral face of the antiprism.
\end{prop}

\begin{proof}
Let $F$ be a non-regular triangle $ABC$. We may suppose that $d(A,B)\neq d(B,C)$ and $d(A,B)\neq d(A,C)$.
Denote by $D_1$ and $D_2$ the midpoints of the segments $[A,C]$ and $[B,C]$ respectively.
By the above assumptions, neither
the reflection $\varphi_{A,C}$ with respect to the plane $\pi_{A,C}$,
that is orthogonal to $\overrightarrow{AC}$ and passes through $D_1$, nor
the reflection $\varphi_{B,C}$ with respect to the plane $\pi_{B,C}$,
that is orthogonal to $\overrightarrow{BC}$ and passes through $D_2$
are not isometry of $P$.

Now, denote by $O$ the center of the circumsphere (i.~e. the gravity center) for $P$.
According to Proposition \ref{prop.n2},
both
the rotation $\psi_1=\psi_{A,C}$ by the angle $\pi$ around $OD_1$  and
the rotation $\psi_2=\psi_{B,C}$ by the angle $\pi$ around $OD_2$
are the isometries of~$P$.
It is clear that $\psi_1(F)$ and $\psi_2(F)$ are faces of $P$ with a common vertex $C$.
If $\alpha:=\angle CAB$, $\beta:=\angle ABC$, $\gamma:=\angle BCA$, then the above faces have the following angles at the vertex $C$:
the angle $\gamma$ of $F$ is between the angles $\alpha$ from $\psi_1(F)$ and $\beta$ from $\psi_2(F)$ ($[A,C]$ is a common edge of $F$ and $\psi_1(F)$,
$[B,C]$ is a common edge of $F$ and $\psi_2(F)$).

If all sides of $F$ have different lengths, then (by the same arguments)
the rotation $\psi'_1=\psi_{A,B}$ by
the angle $\pi$ around $OD_3$, where $D_3$ is the midpoind of the segment $[A,B]$, is an
isometry of $P$. We have three faces of $P$, say $F$, $\psi_1'(F)$, and $\psi_2(F)$, with a common vertex $B$,
such that $[A,B]$ is a common edge of $F$ and $\psi_1'(F)$,
$[B,C]$ is a common edge of $F$ and $\psi_2(F)$. Moreover, the angle $\beta$ of $F$ is between the angles $\alpha$ and $\gamma$ of $\psi_1'(F)$ and $\psi_2(F)$
respectively.
Since $P$ is homogeneous, there are three  faces with a common point $C$, whose union is isometric to $F\cup \psi_1'(F) \cup \psi_2(F)$,
in particular
the angle $\beta$ is between the angles $\gamma$ and $\alpha$. At the same vertex we have three angles from $F$, $\psi_1(F)$ and $\psi_2(F)$, where
the angle $\gamma$ is between the angles $\alpha$ and $\beta$. Since $\alpha+\beta+\gamma=\pi$, and by homogeneity
at $C$ we have the equal quantity of the angles $\alpha$, $\beta$, and $\gamma$, then there are exactly 3 faces of $P$ adjacent to $C$
(with the values $\alpha$, $\beta$, and $\gamma$). The same we have in any other vertex of $P$ due to its homogeneity. Hence, $P$ is a tetrahedron.

Now, let us suppose that $\alpha:=\angle CAB=\angle ABC$ and $\gamma: =\angle BCA \neq \alpha$. Then the first argument about the faces $F$,
$\psi_1(F)$, and $\psi_2(F)$ again follows. Moreover, $d(A,C)=d(B,C)$ and $d(D_1,O)=d(D_2,O)$.
Let us consider the plane $\Pi$ through the points $O,D_1,D_2$. It is easy to see that  $D_2':=\psi_1(D_2), D_1':=\psi_2(D_1) \in \Pi$ and
$d(D_2',D_1)=d(D_1,D_2)=d(D_1',D_2)$. Applying the maps $\psi_1$ and $\psi_2$ to the obtained points and iterating this process, we obtain
a vertex set of some regular $n$-gon for $n\geq 3$ in the plane $\Pi$ and with side length $d(D_1,D_2)$.
In order to get a convex polyhedron, $n$ should be even number: $n=2m$, where $m\geq 2$.
If $m=2$ and $n=4$, then we see that $P$ is a tetrahedron, as in the case of pairwise different side lengths.
For $n=2m$ with $m \geq 3$, we get a new construction.
In fact, the union of the faces obtained by iterations of $\psi_1$ and $\psi_2$ applied to $F$ is the
lateral surface of a right antiprism over a regular $m$-gon.
\end{proof}

\begin{remark}\label{re.tetr.1}
It should be noted that a right antiprism on a segment (that can be considered as a degenerate polygon) in $\mathbb{R}^3$ is a tetrahedron.
\end{remark}

\begin{prop}\label{prop.n6.0}
Let $P$ be a homogeneous tetrahedron  with four pairwise isometric faces that are acute triangles
as in  Examples~\ref{ex.tetra.1} and \ref{ex.new.1}. Then $P$ is $n$-point homogeneous for all $n \in \mathbb{N}$.
\end{prop}

\begin{proof}
We know that $P$ is homogeneous, see Example~\ref{ex.tetra.1}.
Let us prove that a homogeneous tetrahedron $P$ (with four pairwise isometric faces, that are acute triangles), as in Examples \ref{ex.tetra.1} and \ref{ex.new.1},
is $3$-point homogeneous.
We use a representation of $P$ as a convex hull of the points
$A_1=(k,l,m)$, $A_2=(-k,-l,m)$, $A_3=(-k,l,-m)$, $A_4=(k,-l,-m)$,
where $k,l,m \in \mathbb{R}$. The polyhedron $P$ is homogeneous according to Example \ref{ex.new.1}.
If $d_{ij}=d(A_i,A_j)$, then $d_{12}=d_{34}=2\sqrt{k^2+l^2}$, $d_{13}=d_{24}=2\sqrt{k^2+m^2}$, and $d_{14}=d_{23}=2\sqrt{l^2+m^2}$.
If $k=l=m$, then $P$ is a regular tetrahedron, hence, is $n$-point homogeneous for any $n \in \mathbb{N}$.
In what follows we assume that not all numbers among $k,l,m$ are coinciding.

Let us consider the sphere $S(A_1,r)=\{A_i\,|\, d(A_1,A_i)=r, i=1,2,3,4\,\}$ with center $A_1$.
We know that $P$ is $2$-point homogeneous if and only if the isotropy group $\I(A_1)$
acts transitively on every (non-empty) distance sphere with center $A_1$.
If $k\neq l \neq m \neq k$, then all distance spheres $S(A_1,r)$ have at most one point. Therefore, $\I(A_1)$ acts transitively on every $S(A_1,r)$, hence,
$P$ is $2$-point homogeneous. Finally, we should consider the case, when there are exactly two coinciding numbers among $k,l,m$.
Without loss of generality, we can suppose that $k=l\neq m$. Then $S(A_1,0)=\{A_1\}$, $S(A_1,2\sqrt{2k^2})=\{A_2\}$, $S(A_1,2\sqrt{k^2+m^2})=\{A_3,A_4\}$.
There is no problem with spheres with one element. Hence, it suffices to prove that $\I(A_1)$ acts transitively on
$S(A_1,2\sqrt{k^2+m^2})=\{A_3,A_4\}$. Let us consider the map $\eta:(x_1,x_2,x_3) \mapsto (x_2, x_1,x_3)$, which is an isometry of $\mathbb{R}^3$.
Since $k=l$, then  $\eta \in \Isom (P)$, $\eta(A_1)=A_1$, $\eta(A_2)=A_2$, $\eta(A_3)=A_4$, $\eta(A_4)=A_3$. Hence, $\I(A)$ acts transitively
on $S(A_1,2\sqrt{k^2+m^2})=\{A_3,A_4\}$ and $P$ is $2$-point homogeneous in this case too.

Now we are going to prove that $P$ is $3$-point homogeneous. Recall that we already assume that $P$ is not a regular tetrahedron (this variant has been considered).
Let us consider two triples $(B_1,B_2,B_3)$,  $(C_1,C_2,C_3)$ of vertices of $P$
such that $d(B_i,B_j)=d(C_i,C_j)$, $i,j=1,2,3$, and prove that there is an isometry $\theta \in \Isom(P)$ such that
$\theta (B_i)=C_i$, $i=1,2,3$. If there are $i\neq j$ such that $B_i=B_j$, then we may omit $B_j$ and $C_j$ in the above triples
and find an isometry $\kappa$, moving the first obtained pair of vertices to the second obtained pair (this is possible since $P$ is $2$-point homogeneous).
It is clear that we may take $\kappa$ as $\theta$.
Now, let us assume that $B_1\neq B_2 \neq B_3\neq B_1$. This means that $B_1, B_2, B_3$ are vertices of one face of $P$.
The same can be said about $C_1,C_2,C_3$. Let us consider $B_4$, the fourth vertex of $P$, additional to $B_1, B_2, B_3$,
and $C_4$, the additional vertex to $C_1, C_2, C_3$. Since $P$ is homogeneous, there is an isometry $\psi \in \Isom(P)$ such that
$\psi(B_4)=C_4$. If the numbers $k,l,m$ are pairwise distinct, then $\psi(B_i)=C_i$, $i=1,2,3$.
Hence, we may take $\psi$ as $\theta$. Finally, we suppose that there are coinciding numbers among $k,l,m$ (but not $k=l=m$).
Under this assumption, it is clear that either  $\psi(B_i)=C_i$ for $i=1,2,3$ (in this case we take $\theta:=\psi$), or
$B_1B_2B_3$ is an isosceles triangle which has the isometry group of order $2$.
We may suppose that $d(B_1,B_2)=d(B_1,B_3)$.
Let us take $\varphi \in \Isom(P)$ such that $\varphi(B_1)=B_1$, $\varphi(B_2)=B_3$, $\varphi(B_3)=B_2$.
Now, we can take $\theta:= \psi \circ \varphi$, and this isometry moves $(B_1,B_2,B_3)$ to $(C_1,C_2,C_3)$.
Therefore, $P$ is $3$-point homogeneous.

By Corollary \ref{co:if-homog}, we observe  that $P$ is $n$-point homogeneous for all $n \in \mathbb{N}$.
\end{proof}

\begin{prop}\label{prop.n6}
Let $P$ be a $2$-point homogeneous polyhedron in $\mathbb{R}^3$ with a face $F$ that is a regular $n$-gon with the center $O_1$, $n\geq 5$.
Then $P$ is invariant under the rotation around the straight line $OO_1$ by the angle $2\pi/n$ and under the reflection with respect to the plane through
the points $O$, $O_1$, and any vertex of $F$.
\end{prop}

\begin{proof}
Let us consider four consecutive vertices $A_{i-1}, A_i, A_{i+1}, A_{i+2}$ of $F$. Since $d(A_{i-1},A_{i+1})=d(A_{i},A_{i+2})$, then there is an isometry $\psi$ of $P$ such
that $\psi(A_{i-1})=A_{i}$ and $\psi(A_{i+1})=A_{i+2}$. It is clear that $\psi (F)$ is a face of $P$ with vertices $A_{i}$ and $A_{i+2}$, therefore, $\psi (F)=F$
(there is only one face of $P$ with points $A_{i}$ and $A_{i+2}$). Hence, $\psi$ acts on $F$ by the rotation by the angle $2\pi/n$, proving
the first assertion. Further, since $P$ is $2$-point homogeneous, there is an isometry $\eta$ of $P$ such
that $\eta(A_{i-1})=A_{i+1}$ and $\eta(A_{i+1})=A_{i-1}$. It is clear that $\eta (F)=F$ is a face of $P$ with vertices $A_{i-1}, A_{i+1}$.
Therefore, $\eta$ acts on $F$ by the reflection with respect to the
plane through the points $O$, $O_1$, $A_i$
that proves
the second assertion.
\end{proof}

\begin{prop}\label{prop.n7}
Let $P$ be a $2$-point homogeneous polyhedron in $\mathbb{R}^3$ with a face $F$ that is a rectangle with the center $O_1$.
Then $P$ is invariant under the rotation around the straight line $OO_1$ by the angle $\pi$ and under the reflection with respect to the plane through the points $O$, $O_1$,
and the midpoint of any edge of $F$.
\end{prop}

\begin{proof}
Let us consider the four vertices $A, B, C, D$ of $F$ consecutively. Since $d(A,C)=d(B,D)$, there is an isometry $\psi$ of $P$ such
that $\psi(A)=B$ and $\psi(C)=D$. It is clear that $\psi (F)$ is a face of $P$ with vertices $B, D \in F$, therefore, $\psi (F)=F$, $\psi(B)=A$, and $\psi(D)=C$.
Thus $\psi$ acts on $F$ by the reflection with respect to the straight line through
the midpoints of the segments $[A,B]$ and $[C,D]$ (these midpoints are fixed by $\psi$).
If we change vertices by the rule  $A \mapsto B \mapsto C\mapsto D\mapsto A$, then we
get an isometry of $P$, that preserves $F$ and acts on $F$ by the reflection with respect to the straight line through
the midpoints of the segments $[B,C]$ and $[D,A]$.
Moreover, since $P$ is $2$-point homogeneous, there is an isometry $\eta$ of $P$ such
that $\eta(A)=C$ and $\eta(C)=A$. It is clear that $\eta (F)=F$, since the vertices $A$ and $C$ are only in one face of $P$.
Therefore, $\eta$ acts on $F$ by  the rotation  to the angle $\pi$, that proves all assertions of the proposition.
\end{proof}

\begin{lemma}\label{two.faces}
Let us suppose that a $2$-point homogeneous polyhedron $P$ in $\mathbb{R}^3$ has faces $F_1$ and $F_2$ with a common side,
such that $F_1$ is a rectangle and $F_2$ is a regular triangle. Then if $F_1$ is (is not) a square, then $P$ is a cuboctahedron
($P$ is a right prism over a regular triangle, respectively).
\end{lemma}

\begin{proof}
Let us suppose that ${F}_1$ has consecutive vertices $A,B,C,D$ and  ${F}_2$ has vertices $C,D,E$,
where $d(A,B)=d(C,D)=d(C,E)=d(D,E)=a$, $d(A,D)=d(B,C)=b$.

At first, we suppose that $b=a$. Since $P$ is $2$-point homogeneous, then any side of a face $F$,
isometric to $F_2$, is also a side of a squared face isometric to $F_1$ (there is an isometry of $P$ that moves the segment $[C,D]$ to a given side),
as well as any side of  a face $F$, that isometric to $F_1$, is also a side of a triangular face isometric to $F_2$.
It is easy to see that there are exactly four faces at any vertex, two of them are isometric to $F_1$ (and they have no common side)
and two other are isometric to $F_2$. Due to the $2$-point homogeneity, we get that $P$ is a cuboctahedron.

Now, suppose that  $b\neq a$. Since $P$ is homogeneous, then any side of $F_2$ is also a side of a rectangular face isometric to $F_1$.
Denote these faces by $F_3$ (with consecutive vertices $D,E,M,N$)
and $F_4$ (with consecutive vertices $C,E,K,L$). Since $P$ is $2$-point homogeneous,
then there is an isometry $\eta$ of $P$ that maps the pair of vertices $(C,D)$ to the pair $(A,B)$.
We denote by $F_5$ the image of $F_2$ under this isometry (two of three vertices of $F_5$ are $A$ and $B$).
Finally, we denote by ${F}_6$ a face with the side $[A,D]$.
We emphasize that $F_1$, $F_2$, $F_3$, and $F_6$ are all faces containing the vertex $D$.

If $F_6=F_3$ (or, equivalently, $N=A$), then $P$ is a right prism over a regular triangle.

Now, suppose that $F_6\neq F_3$.
We consider an isometry $\theta$ of $P$ such that $\theta(A)=E$ and $\theta(E)=A$ (it does exist due to the $2$-point homogeneity of $P$).
It is clear that
$\theta(F_2)=F_5$. Note that $F_5$ is a unique face of $P$ that contains $A$ and is isometric to $F_2$.
Otherwise, the sum of the plane angles of the faces at the chosen vertex (e.~g., $A$ or $D$) would be greater than $2\pi$
(there are at least three right angles and at least two angles with radian measure $2\pi/3$).
Therefore, either $\theta(C)=B$ or $\theta(D)=B$, in particular, $B \in \{\theta(C),\theta(D)\}$.

Since there are exactly two faces with the vertex $E$,
that are isomorphic to $F_1$, namely, $F_3$ and $F_4$, then either $\theta(F_1)=F_3$ or $\theta(F_1)=F_4$. In the first case we get
$\theta(B)=D$, $\theta(D)=M$, $\theta(C)=N$; in the second case we have $\theta(B)=C$, $\theta(D)=L$, $\theta(C)=K$.
In any case, $\{\theta(C),\theta(D)\} \subset \{K,L,M,N\}$.

On the other hand, it is easy to see that $B \not \in \{K,L,M,N\}$, therefore, the isometry $\theta$ does not exist.
This contradiction proves the lemma.
\end{proof}

\begin{prop}\label{prop.n8}
Let us suppose that a $2$-point homogeneous polyhedron $P$ in $\mathbb{R}^3$ has rectangular faces $F_1$ and $F_2$, that have a common side.
Then $P$ is either a rectangular parallelepiped, or a right prism over a regular polygon.
\end{prop}

\begin{proof}
Denote the centers of $F_1$ and $F_2$  by  $O_1$ and $O_2$ respectively, and
denote the common edge of $F_1$ and $F_2$ by $e_1$. We consider planes $L_1$ and $L_2$ through the endpoints of the edge $e_1$ and orthogonal to $e_1$.

By Proposition \ref{prop.n7}, $P$ is invariant
under the reflection $\psi_1$ with respect to the plane through the straight line $OO_2$ which is parallel to $e_1$.
Let us consider $F_3:=\psi_1(F_1)$, $O_3:=\psi_1(O_1)$, and $e_2:=\psi_1(e_1)$. It is clear that the rectangular faces $F_2$ and $F_3$ have a common edge $e_2$ which is
parallel to $e_1$. Repeating this procedure
(the next isometry $\psi_2$ is the reflection with respect to the plane through the straight line $OO_3$ which is parallel to $e_1$
and $F_4:=\psi_2(F_2)$, etc.),
we get  the chain of faces $F_1,F_2,F_3,\dots$, such that $F_i\bigcap F_{i+1}=e_i$ is an edge of $P$ parallel to $e_1$.
Since $P$ is convex and all $F_i$ are rectangles, then there is $k\in \mathbb{Z}$, such that $F_k=F_1$ (it is clear that $k \geq 4$).

Denote by $\Gamma_1$ and $\Gamma_2$ the convex  hulls of
$\mathcal{B} \bigcap L_1$ and $\mathcal{B}\bigcap L_2$ respectively, where $\mathcal{B}: =\bigcup_{i=1}^{k-1} F_i$.
It is clear that $\Gamma_1$ is isometric to $\Gamma_2$ (the corresponding isometry can be realized by an isometry $\eta$ of $P$ in the proof of Proposition  \ref{prop.n7})
and is an inscribed polygon, the lengths of whose $k-1$ sides are either equal or alternate. Therefore, $\Gamma_1$ is
a homogeneous $(k-1)$-gon.

We are going to prove that $P$ is the convex hull of $\mathcal{B}=\bigcup_{i=1}^{k-1} F_i$.
It is clear that the convex hull of $\mathcal{B}$ is a right prism  over $\Gamma_1$ (of height that is the length of $e_1$).
In what follows, we consider also the edges $f_i:=L_1 \bigcap F_i$, $i=1,2,\dots,k-1$.

{\it Let us suppose that $P$ is not the convex hull of } $\mathcal{B}$. Then there is a vertex of $P$ which is not situated between the planes $L_1$ and $L_2$.
Without loss of generality, we may assume that there is a vertex of $P$ such that it and $L_2$ lie in different half-spaces determined by $L_1$.
Hence for any side $f_i$ of $\Gamma_1$ there is a face of $P$, that contains $f_i$ and that is not a subset of $\mathcal{B}$.

Now, consider the vertex $A:=f_1 \bigcap f_2$. There are at least four faces of $P$ with this vertex: the rectangles $F_1$ and $F_2$,
the second face $F'_1$ that contains $f_1$,
and the second face $F'_2$ that contains  $f_2$. The sum of the angles of $F_1$ and $F_2$ at $A$ is $\pi$.
Therefore, the impacts of $F'_j$, $j=1,2$, are less than $\pi$.
By Proposition \ref{pr.idex.2}, we get that $F'_j$, $j=1,2$, can be only a homogeneous $m_j$-gon with $m_2 \geq m_1 \geq 3$
(we assume $m_2\geq m_1$ without loss of generality). Each angle of such a polygon
is $\pi (1-2/m_j)$. Since $\pi(1-2/m_1)+\pi(1-2/m_2)<\pi$ (the sum of all angles of all faces at the vertex $A$ is less than $2\pi$), we get $1/m_1+1/m_2>1/2$.
Therefore, $m_1=3$ and $m_2 \in\{3,4,5\}$. This means that $F'_1$
is a regular triangle and $F'_2$ is either regular triangle, or a regular pentagon, or a rectangle.

Since $F'_1$ has a common side with $F_1$, then
by Lemma \ref{two.faces}, $P$ is a right prism over a regular triangle if $F_1$ is not a square and $P$ is a cuboctahedron  if $F_1$ is a square.

Note that  in the first case $P$ is the convex hull of $\mathcal{B}$, but in the second case a cuboctahedron has no two rectangular faces with a common side.
This contradiction proves that  $P$ {\it is the convex hull of} $\mathcal{B}=\bigcup_{i=1}^{k-1} F_i$.

It is obvious that $\Gamma_1$ is a regular triangle for $k=4$  (since $\Gamma_1$ is
a homogeneous $(k-1)$-gon), and it is a rectangle for $k=5$, because all its angles are pairwise equal.
In these cases $\Gamma_1$ is $2$-point homogeneous.
In other cases ($k\geq 6$), there are
consecutive vertices $A,B,C,D,E$ of $\Gamma_1$ such that $E\neq A$, $[A,B]$ is an edge of $F_1$ and $[B,C]$ is an edge of $F_2$.
It is clear that $[A,C]$ and $[B,D]$ are
diagonals of $\Gamma_1$, that lie inside $\Gamma_1$ except for their endpoints, and
$d(A,C)=d(B,D)$. Since $P$ is $2$-point homogeneous, then there is an isometry $\psi$ of
$P$ such that $\psi(A)=B$, $\psi(C)=D$. Since $\Gamma_1$ is a unique face of $P$ that contains $[B,D]$, then
$\psi(\Gamma_1)=\Gamma_1$ and $\psi(B)=C$ which implies that  $F_1$ is isometric to $F_2$. Therefore, $\Gamma_1$ is regular.

If $\Gamma_1$ is regular (equivalently, if $F_1$ is isometric to $F_2$), then $P$ is a right prism over a regular polygon.
Otherwise, $\Gamma_1$ is a non-square rectangle by Theorem~\ref{th.2dim.1}, hence, $P$ is a rectangular parallelepiped.
\end{proof}
\smallskip

In what follows, we obtain the classification of $2$-point homogeneous (Theorem \ref{th.3dim.2point}) and the classification of $3$-point homogeneous
(Theorem \ref{th.3dim.mpoint}) polyhedra in $\mathbb{R}^3$.
\smallskip

\begin{theorem}\label{th.3dim.2point}
Let $P$ be a $2$-point homogeneous non-degenerate polyhedron in $\mathbb{R}^3$. Then one of the following properties holds:

1) $P$ is a regular polyhedron (tetrahedron, cube, octahedron, dodecahedron, or icosahedron);

2) $P$ is a cuboctahedron;

3) $P$ is a homogeneous tetrahedron  with four pairwise isometric faces that are acute triangles;

4) $P$ is a right prism over a regular $n$-gon, $n\geq 3$, such that the set of distances between vertices of $P$ from a fixed base has empty
intersection with the set of distances between vertices of $P$ from distinct bases;

5) $P$ is a right antiprism over a regular $n$-gon, $n\geq 2$, such that the set of distances between vertices of $P$ from a fixed base has empty
intersection with the set of distances between vertices of $P$ from distinct bases;

6) $P$ is a rectangular parallelepiped of size $a\times b \times c$, where $a\leq b \leq c$ and $a^2+b^2\neq c^2$.
\end{theorem}

\begin{proof}
All polyhedra in the list of the theorem are $2$-point homogeneous
according to Theorems \ref{th:reg_pol} (regular polyhedra) and \ref{th:edge} (homogeneous polyhedra with edges of equal length),
Proposition \ref{prop.n6.0} (homogeneous tetrahedra),
Proposition \ref{pr.prism.3} (right prisms) and Proposition \ref{pr.prism.3.00} (right  antiprisms).

Hence, it suffices to show that there are no other $2$-point homogeneous polyhedra.
Suppose that $P$ is $2$-point homogeneous and is not in the statement of the theorem.

Taking into account the results of Theorems \ref{th:reg_pol},
\ref{th:edge}, Example \ref{ex.new.1}, and Propositions  \ref{prop.n5n}, \ref{pr.prism.3}, \ref{prop.n4},
\ref{prop.n8},
we may suppose {\it that no one face of a $2$-point homogeneous polyhedron $P$ is non-regular triangle, $P$ has edges of different lengths, and $P$ is not a right prism}.

If $P$ has no non-square rectangular face, then all faces of $P$ are regular polygons by
the above assumption about triangular faces and Proposition \ref{prop.n3}.
This implies
that all the edges of $P$ have one and the same length, which is impossible. Therefore, $P$ has a rectangular face of size $a\times b$, where $a\neq b$.

Now, take any vertex $A$ of $P$. Suppose that $A$ is adjacent to $m$  non-square rectangular  faces. Obviously, $m\leq 3$, and we have proved that $m\geq 1$.

Let us consider all edges adjacent to $A$, arranged in order of going around $A$: $e_1,e_2,\dots, e_k$, $i\in \mathbb{Z}_k$.
It is clear that $k \geq 3$.
Denote by $l_i$ the length of the edge $e_i$, where $i=1,2,\dots,k$.
It should be noted that by our assumptions and by Proposition \ref{pr.four.edges},
the edges of $P$ may have either two or three different lengths.

Let us consider the following observations.

1) If $l_i \neq l_{i+1}$, then the face between $e_i$ and $e_{i+1}$ is a rectangle of size $l_i \times l_{i+1}$. In particular, the angle of this face at the vertex $A$
is $\pi/2$;

2) If $l_i = l_{i+1}$, then the face between $e_i$ and $e_{i+1}$ is a regular $n$-gon, $n\geq 3$. The angle of this face at the vertex $A$ is at least $\pi/3$;

3) If $l_i \neq l_{i+1}$, then $l_{i-1}=l_i$ and $l_{i+1}=l_{i+2}$, otherwise, we get two rectangular faces with a common edge, and $P$ is a right prism by Proposition
\ref{prop.n8}.

If we have more than $2$ changes of lengths in the sequence $l_1,l_2,l_3 \dots, l_k, l_1$, then $k \geq 6$ due to Observation 3) and
the sum of angles at the vertex $A$ is at least
$3\cdot \pi/2 +3 \cdot \pi/3=5\pi/2 >2\pi$, which is impossible. Since we have at least one face of size $a\times b$,
then there are exactly $2$ changes of lengths in the sequence $l_1,l_2,l_3 \dots, l_k, l_1$. Hence,
the sum of the angles at the vertex $A$ is at least
$2\cdot \pi/2 +(k-2) \cdot \pi/3$. Since this sum should be less than $2\pi$, we have $k<5$. On the other hand, $k\geq 4$ according to Observation 3).

Therefore, $k=4$ and $(l_1,l_2,l_3,l_4)=(a,a,b,b)$ up to permutation of indices of type $i\mapsto i+j$, $i \in \mathbb{Z}_4$, for some $j \in \{1,2,3,4\}$.
Hence, we have two rectangular faces of size $a\times b$ and two regular faces with edge lengths $a$ and $b$ respectively. Since the sum of angles at the vertex $A$ of
these two faces is less than $\pi$, then one of them is a regular triangle, while another one is a regular $n$-gon, where $n=3,4,5$.

In particular, there are two faces of $P$ with a common side, such that one of this faces is a regular triangle and the second one is a non-squared rectangle.
By Lemma~\ref{two.faces}, $P$ should be right prism over a regular triangle, which is impossible.
This contradiction proves the theorem.
\end{proof}
\medskip

\begin{theorem}\label{th.3dim.mpoint}
Let $P$ be a $3$-point homogeneous non-degenerate polyhedron in $\mathbb{R}^3$. Then $P$ is a $m$-point homogeneous for all $m \geq 1$ and
one of the following properties holds:

1) $P$ is a regular tetrahedron, a cube, a regular octahedron, or a regular icosahedron;

2) $P$ is a cuboctahedron;

3) $P$ is a homogeneous tetrahedron  with four pairwise isometric faces that are acute triangles;

4) $P$ is a right prism over a regular $n$-gon, $n\geq 3$, such that the set of distances between vertices of $P$ from a fixed base has empty
intersection with the set of distances between vertices of $P$ from distinct bases;

5) $P$ is a right antiprism over a regular $n$-gon, $n\geq 2$, such that the set of distances between vertices of $P$ from a fixed base has empty
intersection with the set of distances between vertices of $P$ from distinct bases;

6) $P$ is a rectangular parallelepiped of size $a\times b \times c$, where $a\leq b \leq c$ and $a^2+b^2\neq c^2$.
\end{theorem}

\begin{proof}
It is clear that $P$ is one of polyhedra from the list of Theorem \ref{th.3dim.2point}.
By Corollary~\ref{co:if-homog}, we need only to check, which polyhedra from Theorem \ref{th.3dim.2point} are $3$-point homogeneous.

By Theorem \ref{th:reg_pol},
the regular tetrahedron, cube, regular octahedron, regular icosahedron are $m$-point homogeneous polyhedra for every natural $m$, while
the dodecahedron is $2$-point homogeneous but is not $3$-point homogeneous.
By  \cite[Proposition 6.1]{BerNik22}, the cuboctahedron is $m$-point homogeneous polyhedron for every natural $m$.

By Proposition \ref{prop.n6.0},
a homogeneous tetrahedron  with four pairwise isometric faces, that are acute triangles,
is $m$-point homogeneous for all $m \in \mathbb{N}$.

Recall that a generalized homogeneous prism $P$ (in particular, right prism or right antiprism) is called  {\it  rigidly two-layered} if
for all vertices $A,B,C \in {P}$, the equality $d(A,B)=d(A,C)$ implies that $B$ and $C$ are situated in one and the same base of ${P}$
(equivalently, the set of distances between vertices of $P$ from a fixed base has empty intersection with the set
of distances between vertices of $P$ from distinct bases).

If $P_1$ is 2-point homogeneous (hence, $m$-point homogeneous for any $m \in \mathbb{N}$) polygon
and the prism $P=P_1\times [0,c]$ is rigidly two-layered, then $P$ is a $m$-point homogeneous polytope in $\mathbb{R}^{3}$
by Theorem \ref{th.constr.2}.
By 2) of  Proposition \ref{pr.prism.3},
if the prism  $P =P_1 \times [0,c] \subset \mathbb{R}^3$ is not rigidly two-layered and $P_1$ is not a rectangle, then $P$ is not $2$-point homogeneous.

If a right antiprism $P$ over a regular $n$-gon, where $n\geq 2$,
is  rigidly two-layered (as a generalized right prism) or
$P$ is a regular polyhedron,
then it is $m$-point homogeneous for all $m \in \mathbb{N}$ by Proposition \ref{pr.prism.3.00}.
Otherwise, if $P$ is not rigidly two-layered and $P$ is not a regular polyhedron, then $P$ is not $2$-point homogeneous.

By 3) of  Proposition \ref{pr.prism.3},
a rectangular parallelepiped $P=[0,a]\times [0,b] \times [0,c]$ in $\mathbb{R}^3$, where $0<a\leq b \leq c$,
is $m$-point homogeneous  for any natural $m$ if and only if $a^2+b^2\neq c^2$; however, if $a^2+b^2 =c^2$, then $P$ is homogeneous, but is not $2$-point homogeneous.
\end{proof}
\medskip

It should be noted that the list of $3$-point homogeneous polyhedra in Theorem \ref{th.3dim.mpoint} is shorter than
the list of $2$-point homogeneous polyhedra in Theorem \ref{th.3dim.2point} exactly in one item:
the regular dodecahedron is $2$-point homogeneous but is not $3$-point homogeneous.

\vspace{15mm}


\bigskip













\begin{thebibliography}{99}

\bibitem{BerNik19}
{\sl Berestovskii V.~N., Nikonorov Yu.~G.}
Finite homogeneous metric spaces~//
{\it Sib. Mat. Zh.}, 2019, V.\,60, N\,5. P.\,973--995 (in Russian). English translation:
Siberian Math. J., 2019, V.\,60, N\,5. P.\,757--773, {\bf MR}4055423, {\bf Zbl.}1431.51008.
DOI: 10.1134/S0037446619050021


\bibitem{BerNik21}
{\sl Berestovskii V.~N., Nikonorov Yu.~G.}
Finite homogeneous subspaces of Euclidean spaces~//
{\it Mat. Trudy}, 2021, V.\,24, N\,1. P.\,3–-34  (in Russian). English translation:
Siberian Advances in Mathematics, 2021, V.\,31, N\,3, P.\,155--176, {\bf Zbl.}07656933.
DOI: 10.1134/S1055134421030019

\bibitem{BerNik21n}
{\sl Berestovskii V.~N., Nikonorov Yu.~G.}
Semiregular Gosset polytopes~// {\it Izv. Ross. Akad. Nauk, Ser. Mat.}, 2022, V.\,86, N\,4, P.\,51--84 (in Russian).
English translation:  Izv. Math., 2022, V.\,86, N\,4, P.\,667--698, {\bf MR}4599137 , {\bf Zbl.}1528.52006,
DOI: 10.1070/IM9169


\bibitem{BerNik21nn}
{\sl Berestovskii V.~N., Nikonorov Yu.~G.}
On finite homogeneous metric spaces~// {\it Vladikavkaz Mathematical Journal}, 2022, V.\,2, N\,2, P.\,51--61, {\bf MR}4448043, {\bf Zbl.}07598648,
DOI: 10.46698/h7670-4977-9928-z

\bibitem{BerNik22}
{\sl Berestovskii V.~N., Nikonorov Yu.~G.}
On $m$-point homogeneous polytopes in Euclidean spaces~// {\it Filomat}, 2023,  V.\,37, N\,25, P.\,8405--8424, {\bf MR}4626612,
DOI: 10.2298/FIL2325405B


\bibitem{BerNik23}
{\sl Berestovskii V.~N., Nikonorov Yu.~G.}
Perfect and almost perfect homogeneous polytopes~// {\it Journal of Mathematical Sciences}, 2023,  V.\,271, P.\,762--777, {\bf MR}4706298, {\bf Zbl.}07806069,
DOI: 10.1007/s10958-023-06765-8

\bibitem{BerNik24}
{\sl Berestovskii V.~N., Nikonorov Yu.~G.}
On the Geometry of Finite Homogeneous Subsets of Euclidean Spaces, Surveys in Geometry II, eds. A. Papadopoulos, Springer, Cham, 2024, P.\,305--335,
\url{https://link.springer.com/chapter/10.1007/978-3-031-43510-2_10}


\bibitem{Cox40}
{\sl Coxeter H.~S.~M.}
Regular and semi-regular polytopes. I,
Math. Z., 1940,  V.~46, N~1,
P.\,380--407, {\bf MR}0002181, {\bf Zbl.}0022.38305, DOI: 10.1007/BF01181449.




\bibitem{Cox63}
{\sl Coxeter H.~S.~M.}
Regular polytopes. Second edition.
New York: The Macmillan Company; London: Collier-Macmillan Ltd., 1963,
{\bf MR}0151873, {\bf Zbl.}0118.35902.


\bibitem{Cox85}
{\sl Coxeter H.~S.~M.}
Regular and semi-regular polytopes. II,
Math. Z., 1985,  V.~188, N~4, P.\,559--591,
{\bf MR}774558, {\bf Zbl.}0547.52005, DOI: 10.1007/BF01161657.

\bibitem{Cox88}
{\sl Coxeter H.~S.~M.}
Regular and semi-regular polytopes. III,
Math. Z., 1988, V.~200, N~1, P.\,3--45, {\bf MR}0972395, {\bf Zbl.}0633.52006, DOI: 10.1007/BF01161745.



\bibitem{Cox88n}
{\sl Coxeter H.~S.~M.}
Regular and semiregular polyhedra, Shaping space (Northampton, Mass., 1984),
Design Sci. Collect., Birkh\"{a}user Boston, Boston, MA, 1988, pp.~67--79, {\bf MR}937076.




\bibitem{CLHM53}
{\sl Coxeter H.~S.~M., Longuet-Higgins M.~S., Miller J.~C.~P.}
Uniform polyhedra, Philos. Trans. Roy. Soc. London (A), 1954, V.~246, P.\,401--450,
{\bf MR}0062446, {\bf Zbl.}0055.14204.




\bibitem{GrShe81}
{\sl Gr\"{u}nbaum B., Shephard G.~C.}
Spherical tilings with transitivity properties. In:
{\it The geometric vein}, pp. 65--98, Springer, New York-Berlin, 1981,  {\bf MR}0661770, {\bf Zbl.}0501.51012.



\bibitem{GrShe84}
{\sl Gr\"{u}nbaum B., Shephard G.~C.}
Polyhedra with transitivity properties,
C. R. Math. Rep. Acad. Sci. Canada, 1984, V.~6, N~2, P.\,61--66, {\bf MR}740597, {\bf Zbl.}0528.52009.


\bibitem{Edmonds05}
{\sl Edmonds A.~L.}
The geometry of an equifacetal simplex~//
{\it Mathematika}, 2005, V.\,52, N\,1--2, P.\,31--45, {\bf MR}2261840, {\bf Zbl.}1113.52030,
DOI: 10.1112/S0025579300000310

\bibitem{Edmonds09}
{\sl Edmonds A.~L.}
The partition problem for equifacetal simplices~//
{\it Beitr. Algebra Geom.}, 2009, V.\,50, N\,1, P.\,195--213, {\bf MR}2499788, {\bf Zbl.}1159.52015,
\url{http://www.emis.de///journals/BAG/vol.50/no.1/12.html}

\bibitem{KoKo17}
{\sl Koca N.~O., Koca M.}
Regular and irregular chiral polyhedra from Coxeter diagrams via quaternions~//
{\it Symmetry}, 2017, V.\,9, No. 8, Paper No. 148, 22 p, {\bf MR}3691675, {\bf Zbl.}1423.52021,
DOI: 10.3390/sym9080148




\bibitem{Leo17}
{\sl Leopold U.}
Vertex-transitive polyhedra of higher genus, I, Discrete Comput. Geom., 2017, V.~57,
N~1, P.\,125--151, {\bf MR} 3589059, {\bf Zbl.}1360.52005, DOI: 10.1007/s00454-016-9828-9.



\bibitem{Mar}
{\sl Martini H.}
A hierarchical classification of Euclidean polytopes with
regularity properties~//
{\it Polytopes: Abstract, Convex and Computational}~/
Proc. of the~NATO Advanced Study Institute,
Scarborough, Ontario, Canada, August 20--September 3, 1993~/
Eds. T.~Bisztriczky et al.
Dordrecht: Kluwer Academic Publishers.
NATO ASI Ser., Ser. C, Math. Phys. Sci. 440, 1994. P.\,71--96, {\bf MR}1322058, {\bf Zbl.}0812.51015.

\bibitem{Resh14}
{\sl Reshetov A.}
A unistable polyhedron with 14 faces~//
{\it Int. J. Comput. Geom. Appl.}, 2014,  V.\,24, No. 1, P.\,39--59, {\bf MR}3261948, {\bf Zbl.}1309.68205,
DOI: 10.1142/S0218195914500022

\bibitem{Rob}
{\sl Robertson S.~A.}
{\it Polytopes and symmetry.
London Mathematical Society Lecture Note Series, 90}. Cambridge etc.: Cambridge University Press, 1984, {\bf MR}0833289, {\bf Zbl.}0548.52002.

\bibitem{RobCar}
{\sl Robertson S.~A., Carter~S.}
On the Platonic and Archimedean solids~//
{\it J. Lond. Math. Soc., II. Ser. 2,} 1970, V.\,1, P.\,125--132, {\bf MR}0257877, {\bf Zbl.}0194.51602,
DOI: 10.1112/jlms/s2-2.1.125


\bibitem{RobCarMor}
{\sl Robertson S.~A., Carter~S., Morton H.~R.}
Finite orthogonal symmetry~//
{\it Topology,} 1970, V.~9, P.\,79--95, {\bf MR}0253148, {\bf Zbl.}0186.55603,
DOI: 10.1016/0040-9383(70)90052-2



\bibitem{Ski75}
{\sl Skilling J.}
The complete set of uniform polyhedra, Philos. Trans. Roy. Soc. London (A),
1975, V.~278, P.~111--135, {\bf MR}0365333, {\bf Zbl.}0308.50005, DOI: 10.1098/rsta.1975.0022.


\bibitem{Sop70}
{\sl Sopov S.P.}
Proof of the completeness of the enumeration of uniform polyhedra (in Russian), Ukrain. Geom. Sbornik, 1970, V.~8, P.~139--156,
{\bf MR}0326550, {\bf Zbl.}0214.19501.




\bibitem{Zef}
{\sl Zefiro L.}
Vertex- and edge-truncation of the Platonic and Archimedean solids leading to vertex-transitive polyhedra,
\url{https://www.mi.sanu.ac.rs/vismath/zefirosept2011/_truncation_Archimedean_polyhedra.htm}



\end{thebibliography}
\end{document}